\documentclass[preprint,10pt,3p,numbers,square]{elsarticle}

\usepackage{amssymb}
\usepackage{amsmath}
\usepackage{amsthm}
\usepackage{xfrac}

\usepackage{tikz}
\usepackage{comment}
\usepackage{hyperref}

\usepackage{empheq}

\usepackage{array}    % for m{...}
\usepackage{makecell} % optional, but convenient

\newcommand{\pupt}[1]{\frac{\partial {#1}}{\partial t}}

\newcommand{\dudt}[1]{\frac{d {#1}}{dt}}

\newcommand{\popt}[2]{\frac{\partial {#1}}{\partial {#2}}}
\newcommand{\ppoppt}[2]{\frac{\partial^2 {#1}}{\partial {#2}^2}}
\newcommand{\ppto}[2]{\frac{\partial}{\partial {#2}}\left({#1}\right)}

\newcommand{\bm}[1]{\boldsymbol{#1}}

\newdefinition{rmk}{Remark}
\newdefinition{dfn}{Definition}
\newdefinition{thm}{Theorem}
\newdefinition{cor}{Corollary}
\newdefinition{lem}{Lemma}
\newdefinition{prop}{Proposition}

\newcommand{\kg}{\ensuremath{\text{kg}}}
\newcommand{\mol}{\ensuremath{\text{mol}}}
\newcommand{\kelv}{\ensuremath{\text{K}}}
\newcommand{\joule}{\ensuremath{\text{J}}}
\newcommand{\meter}{\ensuremath{\text{m}}}
\newcommand{\pascal}{\ensuremath{\text{Pa}}}
\newcommand{\second}{\ensuremath{\text{s}}}

\usepackage{lineno}

\begin{document}

\begin{frontmatter}



%\title{Generalized Tadmor Conditions for Numerical Fluxes with Non-Conservative Secondary Physics: Applications to Real-Gas Dynamics} %% Article title

\title{Generalized Tadmor Conditions and Structure-Preserving Numerical Fluxes for the Compressible Flow of Real Gases}


\author[1,2]{R.B. Klein\corref{cor1}} %% Author name
\ead{rbk@cwi.nl}
\author[2,3]{B. Sanderse}
\ead{B.Sanderse@cwi.nl}
\author[1]{P. Costa}
\ead{P.SimoesCosta@tudelft.nl}
\author[1]{R. Pecnik}
\ead{R.Pecnik@tudelft.nl}
\author[1]{R.A.W.M Henkes}
\ead{R.A.W.M.Henkes@tudelft.nl} 

\affiliation[1]{organization={Delft University of Technology, Process \& Energy},%Department and Organization
            addressline={Leeghwaterstraat 39}, 
            city={Delft},
            postcode={2628 CB}, 
            state={Zuid-Holland},
            country={The Netherlands}}

\affiliation[2]{organization={Centrum Wiskunde \& Informatica, Scientific Computing},%Department and Organization
            addressline={Science Park 123}, 
            city={Amsterdam},
            postcode={1098 XG}, 
            state={Noord-Holland},
            country={The Netherlands}}

\affiliation[3]{organization={Eindhoven University of Technology, Department of Mathematics and Computer Science},%Department and Organization
            addressline={Groene Loper 5}, 
            city={Eindhoven},
            postcode={PO Box 513}, 
            state={Brabant},
            country={The Netherlands}}

\cortext[cor1]{Corresponding author}

%% Abstract
%\begin{abstract}
%We extend Tadmor's flux condition for entropy conservative fluxes to the evolution of general non-convex, non-conserved secondary quantities. The new approach allows for a systematic method to the construction of numerical fluxes for conservation laws that consistently discretize the evolution equations of conserved and non-conservative secondary quantities alongside the primary conservative variables. Our analysis provides necessary conditions for the existence of such numerical fluxes. We show how discrete gradient functions can be used to solve the novel generalized flux conditions. As an application we are interested in the simulation of supercritical compressible turbulence. To simulate this effectively we use the proposed framework to derive an entropy-conserving and kinetic energy consistent numerical flux for arbitrary equations of state. A number of supercritical compressible flow test cases are solved with several non-ideal equations of state ranging in complexity.
%\end{abstract}

\begin{abstract}
We generalize Tadmor's algebraic numerical flux condition for entropy-conservative discretizations of conservation laws to a broader class of \emph{secondary structures}, i.e.\ possibly non-convex secondary quantities whose evolution can consist of both conservative and non-conservative contributions. The resulting \emph{generalized Tadmor condition} yields a discrete local balance law for secondary structures alongside the discrete conservation law that is solved. In contrast to the convex entropy setting, non-convex secondary quantities can have singular Hessians and non-injective gradients; this introduces an additional necessary structural requirement, which we term \emph{(discrete) null-consistency}. Null-consistency constrains admissible numerical work terms and is required for the existence and well-posedness of fluxes satisfying the generalized Tadmor condition. To construct such fluxes in practice, we show how \emph{discrete gradient operators} provide systematic construction methods even when some of the functions entering the secondary structure are arbitrary, as in compressible flow closed by an arbitrary equation of state. As an application, we derive an \emph{entropy-conserving} and \emph{kinetic-energy-consistent} numerical flux for the Euler equations with an arbitrary (non-ideal) equation of state. We demonstrate the performance of the resulting scheme on a set of supercritical/transcritical compressible-flow test cases using several non-ideal equations of state, including a fully turbulent transcritical flow with a state-of-the-art equation of state and models for viscosity and heat conductivity. Computations are performed with our new open-source, flexible, JAX-based, multi-GPU compressible flow solver for Helmholtz-based equations of state available at \href{https://github.com/rbklein/HelmEOS2}{github.com/rbklein/HelmEOS2}.
\end{abstract}


%\item Generalized Tadmor conditions
%\item Entropy conservation
%\item Kinetic energy consistency
%\item Numerical fluxes
%\item Supercritical fluids
%\item General equation of state

%% Keywords

\end{frontmatter}

%\linenumbers

%\newpage
%\input{sections/introduction}
%\input{sections/secondary_structures}
%\input{sections/gen_tadmor_conds}
%\input{sections/real_gasses}
%\input{sections/num_fluxes}
%\input{sections/Experiments}
%\input{sections/conculsion}

\section{Introduction}
Many continuum models in physics take the form of systems of conservation laws, i.e.\ PDEs that evolve a set of \emph{primary} conserved variables through flux balances \cite{dafermoshyperbolic}. In practical computations—e.g.\ compressible turbulence, multi-scale flows, or in strongly non-uniform thermodynamic regimes—standard discretizations may develop severe stability and robustness issues.

A promising strategy to alleviate these issues is the design of \emph{structure-preserving} schemes: discretizations that retain a certain property of the conservation law. For example, schemes can retain the Hamiltonian structure \cite{morrisonhamiltonian}, the differential geometric structure \cite{palhamass}, or the dynamics of extra quantities \cite{energysanderse} in addition to solving the conservation law. In this research, we are particularly interested in the latter: discretizations that, in addition to the primary conservation law, remain consistent with the evolution of selected \emph{secondary} quantities. Typical examples in compressible flow are discrete consistency with transport relations for kinetic energy \cite{ranochaentropy,jamesonformulation,kennedyreduced,pirozzoligeneralized}, entropy \cite{kuyakinetic,kuyahigh,tamakicomprehensive,tadmorentropy}, and pressure \cite{ranochapreventing,bernadeskinetic}. When such secondary structure is preserved, one can often reproduce discrete analogues of continuous stability mechanisms, stabilizing the discretization. For instance, entropy conservation/dissipation can be used to derive nonlinear stability statements \cite{svardweak,serrerelative}, while kinetic-energy consistency balances and controls the discrete exchange between mechanical and thermodynamic energy components \cite{ranochaentropy, jamesonformulation} and prevents excessive numerical dissipation.

A particularly general approach to secondary-structure preservation is Tadmor's framework for \emph{entropy conservative} fluxes \cite{tadmorentropy,tadmornumerical}. Convex mathematical entropies are functions generalizing the thermodynamic entropy of gas dynamics to other conservation laws \cite{LeVequefinite, evanspartial, warneckegodunov}. They are conserved when solutions are smooth and thus satisfy an \textit{additional conservation law}. Tadmor derives an algebraic numerical flux condition (the Tadmor condition) that guarantees a discrete local entropy conservation law is also satisfied by the discretization. This approach has enabled systematic constructions of highly stable entropy-conserving fluxes for, among others, the Burgers equation \cite{jamesonconstruction}, compressible Euler flow with ideal-gas thermodynamics \cite{ismailaffordable,chandrashekarkinetic,ranochathesis,ranochaentropy,ranochacomparison}, shallow-water/ocean models \cite{fjordholmenergy}, and MHD \cite{wintersaffordable}. Furthermore, entropy-conserving fluxes form a basis for designing schemes that handle non-smooth, more specifically discontinuous, solutions \cite{fjordholarbitrarily}.

However, many secondary quantities of practical interest are \emph{non-convex} and/or satisfy \emph{non-conservative balance laws} rather than additional conservation laws (e.g.\ kinetic energy and pressure in compressible turbulence \cite{jamesonformulation, ranochaentropy, bernadeskinetic}). For such quantities, a general analogue of Tadmor's framework is missing, and structure-preserving treatments are often performed on a case-by-case basis. Two fundamental obstacles arise:
(i) for non-convex secondary quantities, the necessary gradients may fail to be injective and their Hessian may be singular, which imposes additional compatibility constraints; and
(ii) in real-gas or general-equation-of-state settings, the fluxes and thermodynamic potentials entering the secondary relations may only be defined implicitly (or only up to an arbitrary choice of thermodynamic model), which limits standard construction methods \cite{ranochathesis, fjordholmenergy, chandrashekarkinetic, ranochacomparison}.

In this work, we generalize Tadmor's approach from convex entropies to \emph{general secondary structures} whose evolution consists of conservative and non-conservative parts. Our first contribution is a generalized compatibility framework at the differential level and its discrete analogue, leading to a \emph{generalized Tadmor condition} that enforces a discrete balance law for a prescribed secondary structure. A key new feature is a structural requirement we call \emph{(discrete) null-consistency}: whenever the Hessian of the secondary quantity has a non-trivial kernel, the remaining terms in the generalized compatibility relations must satisfy additional constraints along the associated null directions. We show that discrete null-consistency is not optional, but that it is a necessary condition for well-posedness and for the existence of numerical fluxes that satisfy the generalized Tadmor condition. In this article, we develop our framework for smooth solutions only, to avoid the difficulties associated with discontinuous solutions and non-conservative balance laws \cite{dalmasodefinition, theinbalance}.

Our second contribution is methodological: we show how \emph{discrete gradient operators} \cite{itohhamiltonian,mclachlangeometric,gonzaleztime} can be used to solve the generalized Tadmor conditions in settings where the relevant functions are only partially specified, as is the case for compressible flow with an arbitrary equation of state. In particular, discrete gradients provide systematic jump expansions for thermodynamic potentials without requiring closed-form expressions \cite{ranochathesis, fjordholmenergy, chandrashekarkinetic, ranochacomparison}.

The framework is motivated by, and applied to, compressible turbulent flow of \emph{supercritical} fluids, where thermodynamic properties can vary extremely rapidly (\autoref{fig:eos_comparison}) near so-called pseudo-phase-transition regimes which can only be described with more general non-ideal equations of state \cite{senguptafully,boldinidirect,boldinicubens,holystthermodynamics,guardonenonideal}. In this case, structure-preserving schemes must be adapted to deal with these general equations of state \cite{aielloentropy, oblapenkoentropy}. As a concrete application, we derive an entropy-conserving and kinetic-energy-consistent numerical flux for the Euler equations with an arbitrary equation of state (equation \eqref{eq:numflux}). Here, entropy conservation demonstrates the use of discrete gradients, while kinetic-energy consistency can be analyzed using our non-conservative extension of Tadmor's framework. The resulting flux removes a singularity present in a recently proposed entropy-conserving and kinetic-energy-consistent flux for general EoS \cite{aielloentropy} by enforcing thermodynamic gradient consistency via discrete gradients. We analyze this singularity mechanism and demonstrate the robustness of the proposed flux on a set of supercritical compressible-flow test cases with multiple non-ideal equations of state.

The remainder of the article is organized as follows. In \autoref{sec:sscl}, we review entropy analysis for conservation laws and introduce our notion of general secondary structures and their compatibility conditions. In \autoref{sec:spns}, we derive the generalized Tadmor condition and the associated discrete null-consistency requirements, and we show how discrete gradients enable systematic constructions. In \autoref{sec:fluxeszz}, we apply the framework to real-gas Euler flow and derive an entropy-conserving, kinetic-energy-consistent flux for arbitrary equations of state. In \autoref{sec:numexps}, we present numerical experiments for supercritical/transcritical test cases.
\section{Secondary structures of conservation laws}\label{sec:sscl}
In this section, we develop a theoretical framework that ensures a solution of a conservation law also satisfies an additional \emph{non-conservative balance law}. We will use this compatibility framework later in the article to inform a discrete framework that, in turn, allows us to construct structure-preserving discretizations. The theoretical framework extends and uses tools of an existing framework for convex extensions of conservation laws \cite{friedrichssystems, godlewskinumerical}. Hence, we will first introduce this framework. After that, we will continue with our novel non-conservative extension.

\subsection{Conservation laws and tools in entropy analysis}\label{sec:entanalysis}
We are interested in general systems of conservation laws, which are systems of partial differential equations (PDEs) formulated in the following form:
\begin{equation}
    \pupt{\bm{u}} + \sum_{i=1}^d \ppto{\bm{f}^i(\bm{u})}{x_i} = 0, 
    \label{eq:cl}
\end{equation}
where $d\in \mathbb{N}$ and $\bm{x} \in \mathbb{R}^d$ is a spatial coordinate, $t \in (0, t_f]$ is a temporal coordinate with $t_f > 0$, $\bm{u} : \mathbb{R}^d \times [0,t_f] \rightarrow \mathcal{U}$ is the solution function taking values in an open (often convex) set $\mathcal{U} \subseteq \mathbb{R}^n$, $n\in \mathbb{N}$ is the number of conservative variables, $\bm{f}^i : \mathcal{U} \rightarrow \mathbb{R}^n$ is the flux function in coordinate-direction $i = 1,...,d$ and $x_i$ denotes the $i$-th component of the vector $\bm{x}$. At $t=0$, the solution function is equal to some continuously differentiable initial condition $\bm{u}(\bm{x},0) = \bm{u}_0(\bm{x})$. When the solution function $\bm{u}$ is continuously differentiable on $\mathbb{R}^d \times [0,t_f]$, many conservation laws describing physical systems often also satisfy additional conservation laws:
\begin{equation}
    \pupt{s(\bm{u})} + \sum_{i=1}^d\ppto{\mathcal{F}^i(\bm{u})}{x_i} = 0,
    \label{eq:cle}
\end{equation}
where $s : \mathcal{U} \rightarrow \mathbb{R}$ is a strictly convex $C^2(\mathcal{U})$ function and $\mathcal{F}^i : \mathcal{U} \rightarrow \mathbb{R}$ is $C^1(\mathcal{U})$. The prototypical example of such a function is the negated entropy density of the Euler equations of gas dynamics. For \eqref{eq:cle} to be satisfied by smooth solutions of \eqref{eq:cl}, it is sufficient if, for all $i$, the following compatibility condition holds:
\begin{equation}
    \bm{\eta}(\bm{u})^T \popt{\bm{f}^i}{\bm{u}} = \popt{\mathcal{F}^i}{\bm{u}}^T, 
    \label{eq:pcond}
\end{equation}
where $\bm{\eta} : \mathcal{U} \rightarrow \mathbb{R}^n, \bm{u} \mapsto \popt{s}{\bm{u}}(\bm{u})$. In the case of entropy functions, the gradient function $\bm{\eta}$ is given a name. They are called entropy variables, as $\bm{\eta}$ is an injective map due to the convexity of $s$. If the compatibility condition is satisfied, the general function $s$ is called a mathematical entropy of the conservation law \eqref{eq:cl} and $\mathcal{F}^i$ the associated entropy flux in analogy to the Euler equation. Gathering all entropy fluxes in a vector $\bm{\mathcal{F}}$, a pair $(s,\bm{\mathcal{F}})$ satisfying \eqref{eq:pcond} is referred to as an entropy pair of the conservation law \eqref{eq:cl}. The entropy conservation law \eqref{eq:cle} now follows from the straightforward manipulations:
\begin{align*}
    0 = \bm{\eta}(\bm{u})^T\left(\pupt{\bm{u}} + \sum_{i=1}^d \ppto{\bm{f}^i(\bm{u})}{x_i}\right) &= \bm{\eta}(\bm{u})^T\pupt{\bm{u}} + \sum_{i=1}^d \bm{\eta}(\bm{u})^T \popt{\bm{f}^i(\bm{u})}{\bm{u}}  \popt{\bm{u}}{x_i} \\
     &= \pupt{s(\bm{u})} + \sum_{i=1}^d\ppto{\mathcal{F}^i(\bm{u})}{x_i}.
\end{align*}
We reiterate that these manipulations are only valid when $\bm{u}$ is continuously differentiable. A useful tool in entropy analysis is the so-called entropy flux potential, defined as:
\begin{equation}
    \psi^i (\bm{u}) := \bm{\eta}(\bm{u})^T\bm{f}^i(\bm{u}) - \mathcal{F}^i(\bm{u}).
    \label{eq:pote}
\end{equation}
Its use comes from the fact that the compatibility condition \eqref{eq:pcond} can be alternatively phrased in terms of the entropy flux potential as follows:
\begin{equation}
    \bm{f}^i(\bm{u})^T\popt{\bm{\eta}}{\bm{u}} = \popt{\psi^i}{\bm{u}}^T, 
    \label{eq:dcond}
\end{equation}
which should hold for all $i$. We will see later that discrete variants of specifically this form of compatibility condition are crucial for proofs of structure preservation. We refer to \eqref{eq:pcond} as the primal compatibility condition, while \eqref{eq:dcond} is referred to as the dual compatibility condition. The primal and dual compatibility relations are equivalent for $C^1(\mathcal{U})$ flux functions $\bm{f}^i$ in light of the product rule identity:
\begin{equation*}
    \ppto{\bm{\eta}^T\bm{f}^i}{\bm{u}}^T = \bm{\eta}(\bm{u})^T \popt{\bm{f}^i}{\bm{u}} + \bm{f}^i(\bm{u})^T\popt{\bm{\eta}}{\bm{u}}.
\end{equation*}
Substituting either the primal or dual compatibility condition into the above immediately results in satisfaction of the other condition with $\psi^i = \bm{\eta}^T\bm{f}^i - \mathcal{F}^i$ and $\mathcal{F}^i = \bm{\eta}^T\bm{f}^i - \psi^i$, respectively. The entropy conservation law can be derived from the dual compatibility condition using the product rule:
\begin{align}
    0 = \bm{\eta}(\bm{u})^T \left(\pupt{\bm{u}} + \sum_{i=1}^d \ppto{\bm{f}^i(\bm{u})}{x_i}\right)  &= \pupt{s(\bm{u})} + \sum_{i=1}^d \ppto{\bm{\eta}(\bm{u})^T\bm{f}^i(\bm{u})}{x_i} - \sum_{i=1}^d\bm{f}^i(\bm{u})^T \popt{\bm{\eta}(\bm{u})}{x_i} \label{eq:dualderiv}  \\
    &\overset{\eqref{eq:pote}}{=} \pupt{s(\bm{u})} + \sum_{i=1}^d\ppto{\mathcal{F}^i(\bm{u})}{x_i} + \sum_{i=1}^d \left(\popt{\psi^i}{\bm{u}}^T - \bm{f}^i(\bm{u})^T \popt{\bm{\eta}}{\bm{u}}\right)\popt{\bm{u}}{x_i}  \nonumber \\
    &\overset{\eqref{eq:dcond}}{=} \pupt{s(\bm{u})} + \sum_{i=1}^d\ppto{\mathcal{F}^i(\bm{u})}{x_i}. \nonumber
\end{align}
In this way, the entropy flux potential $\psi^i$ is just the residual that has to be canceled out by the terms appearing outside of the divergence in the third line to obtain $\mathcal{F}^i$ as an entropy flux. This provides a natural way to extend the tools from entropy analysis to general secondary quantities.
%\overset{\eqref{eq:numqflux}}{=} 
\subsection{General secondary quantities and generalized compatibility}
Besides conserved secondary quantities like entropy functions, many physical systems also exhibit secondary quantities adhering to more general balance laws containing conservative \emph{and} non-conservative terms. The approach and tools developed in the analysis of entropy functions of conservation laws remain relatively unexplored in this more general non-conservative setting. Existing literature has considered several of such non-conserved secondary quantities, for example kinetic energy \cite{ranochaentropy, jamesonformulation} and pressure \cite{ranochapreventing, bernadeskinetic, demichelenovel} in compressible flow applications, but not through the lens of entropy analysis \cite{tadmorentropy, tadmornumerical}. In this article, we aim to explore how these approaches developed for convex entropies can be generalized to the analysis of more general non-conserved secondary quantities, especially for the construction of numerical schemes.

The more general evolution equations associated with secondary quantities can often be formulated as consisting of a conservative and a non-conservative part. Specifically, we will assume a secondary quantity $q : \mathcal{U} \rightarrow \mathbb{R}$, which we allow to be \emph{non-convex}, evolves according to:
\begin{equation}
    \pupt{q(\bm{u})} + \sum_{i=1}^d \ppto{\mathcal{F}^i_q(\bm{u})}{x_i} + \sum_{i=1}^d c_q^i(\bm{u})\ppto{\mathcal{G}^i_q(\bm{u})}{x_i} = 0.
    \label{eq:clq}
\end{equation}
This precise form is assumed due to its prevalence in fluid dynamics applications (e.g.\ kinetic energy \cite{ranochaentropy}, pressure \cite{bernadeskinetic, toutantgeneral} satisfy this). Here, $\mathcal{F}_q^i : \mathcal{U} \rightarrow \mathbb{R}$ is the conservative $q$-flux and the non-conservative part of the evolution is governed by $c_q^i, \mathcal{G}_q^i : \mathcal{U} \rightarrow \mathbb{R}$ with $i=1,...,d$. We will refer to $\mathcal{G}_q^i$ as the non-conservative $q$-flux and to $c_q^i$ as the non-conservative $q$-flux coefficient. As it often models a form of work performed on or by the system, we will refer to the product of the non-conservative $q$-flux derivative and coefficient as it appears in \eqref{eq:clq} as the \emph{work term}. With \eqref{eq:clq}, we can generalize the primal compatibility condition \eqref{eq:pcond} of conserved secondary quantities to the case of general secondary quantities. Assuming $q$ is twice-continuously differentiable and $\mathcal{F}_q^i, \mathcal{G}_q^i$ are continuously differentiable, additionally denoting the gradient $\bm{\xi} : \mathcal{U} \rightarrow \mathbb{R}^n, \bm{u} \mapsto \popt{q}{\bm{u}}(\bm{u})$, $C^1$ solutions to \eqref{eq:cl} adhere to \eqref{eq:clq} if for all $i$ the following compatibility condition is satisfied:
\begin{equation}
    \bm{\xi}(\bm{u})^T \popt{\bm{f}^i}{\bm{u}} = \popt{\mathcal{F}_q^i}{\bm{u}}^T + c_q^i(\bm{u})\popt{\mathcal{G}_q^i}{\bm{u}}^T.
    \label{eq:pcondq}
\end{equation}
This is a \emph{generalized primal compatibility condition}. It follows from equivalent manipulations as for convex entropies that:
\begin{align}
    0 = \bm{\xi}(\bm{u})^T\left(\pupt{\bm{u}} + \sum_{i=1}^d \ppto{\bm{f}^i(\bm{u})}{x_i}\right) 
    &= \bm{\xi}(\bm{u})^T\pupt{\bm{u}} + \sum_{i=1}^d \bm{\xi}(\bm{u})^T \popt{\bm{f}^i(\bm{u})}{\bm{u}}  \popt{\bm{u}}{x_i}  \label{eq:diffprimderiv} \\
     &= \pupt{q(\bm{u})} + \sum_{i=1}^d \ppto{\mathcal{F}_q^i(\bm{u})}{x_i} + \sum_{i=1}^dc_q^i(\bm{u})\ppto{\mathcal{G}^i_q(\bm{u})}{x_i}. \nonumber
\end{align}
This motivates us to generalize the definition of an entropy pair to a more general secondary structure.
\begin{dfn}[Secondary structures]
    A secondary structure $(q,\bm{\mathcal{F}}_q,\bm{\mathcal{G}}_q,\bm{c}_q)$ of a conservation law \eqref{eq:cl} is an ordered set consisting of a $C^2$ secondary quantity $q : \mathcal{U} \rightarrow \mathbb{R}$, a vector of $C^1$ conservative fluxes $\mathcal{F}_q^i:\mathcal{U} \rightarrow \mathbb{R}$, a vector of $C^1$ non-conservative fluxes $\mathcal{G}_q^i:\mathcal{U} \rightarrow \mathbb{R}$ and a vector of non-conservative flux coefficients $c_q^i:\mathcal{U} \rightarrow \mathbb{R}$ that satisfy the set of PDEs \eqref{eq:pcondq}. The secondary structure is conservative if \eqref{eq:pcondq} is satisfied with $\bm{\mathcal{G}}_q,\bm{c}_q = 0$.
\end{dfn}
An entropy pair of \eqref{eq:cl} thus forms a conservative secondary structure. A generalized dual compatibility condition can also be constructed, for this we define the conservative flux residual function:
\begin{equation}
    \psi_q^i(\bm{u}) := \bm{\xi}(\bm{u})^T\bm{f}^i(\bm{u}) - \mathcal{F}_q^i(\bm{u}),
    \label{eq:potq}
\end{equation}
which takes the same form as the entropy flux potential \eqref{eq:pote}. The generalized primal compatibility condition can be rephrased in terms of this residual function to obtain the following \emph{generalized dual compatibility condition}:
\begin{equation}
    \popt{\psi_q^i}{\bm{u}}^T - \bm{f}^i(\bm{u})^T \popt{\bm{\xi}}{\bm{u}} = c_q^i(\bm{u}) \popt{\mathcal{G}_q^i}{\bm{u}}^T. 
    \label{eq:dcondq}
\end{equation}
Equivalence between the generalized primal and dual compatibility condition follows for $C^1$ flux functions $\bm{f}^i$ using the same product rule identity as for entropy. The secondary evolution equation \eqref{eq:clq} can be obtained from the generalized dual compatibility and the same product rule approach as for entropy in \eqref{eq:dualderiv}:
\begin{align*}
    0 = \bm{\xi}(\bm{u})^T \left(\pupt{\bm{u}} + \sum_{i=1}^d \ppto{\bm{f}^i(\bm{u})}{x_i}\right)  &= \pupt{q(\bm{u})} + \sum_{i=1}^d \ppto{\bm{\xi}(\bm{u})^T \bm{f}^i(\bm{u})}{x_i} - \sum_{i=1}^d\bm{f}^i(\bm{u})^T \popt{\bm{\xi}(\bm{u})}{x_i} \\
    &\overset{\eqref{eq:potq}}{=} \pupt{q(\bm{u})} + \sum_{i=1}^d\ppto{\mathcal{F}_q^i(\bm{u})}{x_i} + \sum_{i=1}^d \left(\popt{\psi_q^i}{\bm{u}}^T - \bm{f}^i(\bm{u})^T \popt{\bm{\xi}}{\bm{u}}\right)\popt{\bm{u}}{x_i} \\
    &\overset{\eqref{eq:dcondq}}{=} \pupt{q(\bm{u})} + \sum_{i=1}^d\ppto{\mathcal{F}_q^i(\bm{u})}{x_i} +  \sum_{i=1}^dc_q^i(\bm{u}) \ppto{\mathcal{G}_q^i(\bm{u})}{x_i}.
\end{align*}
In the generalized case, the flux residual $\popt{\psi_q^i}{\bm{u}}$ should not cancel out the non-conservative remainder of the product rule application $\bm{f}^i(\bm{u})^T \popt{\bm{\xi}}{\bm{u}}$, but instead should correct it to the work term we expect. This is precisely the statement of the dual compatibility condition, equation \eqref{eq:dcondq}.

\begin{rmk}[General secondary balance law]
    We have stated the work term in \eqref{eq:clq} with only one term per coordinate direction, as this is how it often appears in fluid dynamical contexts. However, the most general form of the work term that can appear in the evolution equation of a function $q(\bm{u})$ of the solution $\bm{u}$ of a conservation law \eqref{eq:cl} is:
    \begin{equation*}
        \sum_{i=1}^d\sum_{j=1}^nc_q^{i,j}(\bm{u})\ppto{\mathcal{G}_q^{i,j}(\bm{u})}{x_i}.
    \end{equation*}
    It is straightforward to show by an application of the product rule that any sufficiently smooth function $q$ of the solution $\bm{u}$ of a conservation law \eqref{eq:cl} satisfies a balance law with a work term as above and a conservative flux $\mathcal{F}_q^i(\bm{u}) = \bm{\xi}(\bm{u})^T \bm{f}^i(\bm{u})$. All results in this article can be extended to this case by replacing the single work term with a sum.
\end{rmk}

\subsection{Null-consistency}\label{sec:nullcons}
\begin{figure}[]
    \centering
    \includegraphics[width=0.6\linewidth]{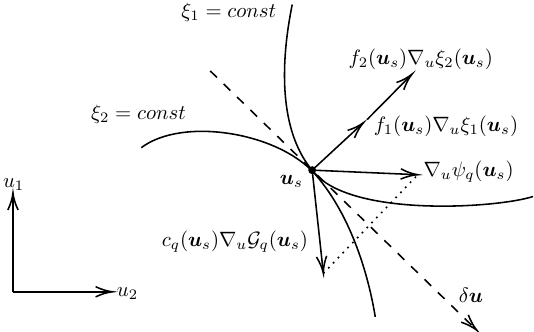}
    \caption{Illustration of the null-consistency property at a point $\bm{u}_s \in \mathcal{U}$ where $\mathcal{H}_q(\bm{u}_s) := \ppoppt{q}{\bm{u}}(\bm{u}_s)$ is singular for $n=2$ and $d=1$. Note how $f_1(\bm{u}_s) \nabla_u \xi_1(\bm{u}_s) + f_2(\bm{u}_s) \nabla_u \xi_2(\bm{u}_s) = \nabla_u \psi_q(\bm{u}_s) - c_q(\bm{u}_s)\nabla_u \mathcal{G}_q(\bm{u}_s)$ to satisfy the dual compatibility condition.}
    \label{fig:solvacond}
\end{figure} 
The generalized dual compatibility condition \eqref{eq:dcondq} is a system of PDEs \emph{on} $\mathcal{U} \subseteq \mathbb{R}^n$ relating a secondary structure $q,\psi_q^i, \mathcal{G}_q^i,c_q^i : \mathcal{U} \rightarrow \mathbb{R}$ to a conservation law flux $\bm{f}^i : \mathcal{U} \rightarrow \mathbb{R}^n$. We now examine a property of solutions to these PDEs that appears when $q$ is allowed to be non-convex. We call this property null-consistency. When $q$ is non-convex, its Hessian matrix: 
\begin{equation*}
    \mathcal{H}_q(\bm{u}) := \ppoppt{q}{\bm{u}}(\bm{u}) = \popt{\bm{\xi}}{\bm{u}}(\bm{u}),
\end{equation*}
is no longer guaranteed to be non-singular, in contrast to the case of convex entropy functions \cite{warneckegodunov}. At any point $\bm{u}_s \in \mathcal{U}$ where $\mathcal{H}_q(\bm{u}_s)$ is singular, null-consistency constrains how the functions $\psi_q^i, \mathcal{G}_q^i$ may vary in state space at the point $\bm{u}_s$. This property will be important to consider, especially due to the fact that it occurs in multiple discrete forms as a \emph{necessary} condition for the \emph{existence} of structure-preserving discretizations, as we will see in later sections where we analyze numerical schemes.

In more detail, the null-consistency property can be derived as follows. We momentarily use the classical vector calculus notation $\nabla_u := \frac{\partial}{\partial \bm{u}}$ for the gradient on $\mathcal{U}$ and use $\left<\cdot,\cdot \right> : \mathbb{R}^n \times \mathbb{R}^n \rightarrow \mathbb{R}$ and $||\cdot|| : \mathbb{R}^n \rightarrow\mathbb{R}_+$ for the Euclidean inner-product and norm, respectively. Consider a point $\bm{u}_s \in \mathcal{U}$ where $\mathcal{H}_q(\bm{u}_s)$ is singular. Let $\delta \bm{u} \in \ker (\mathcal{H}_q(\bm{u}_s))$ where $\delta \bm{u} \neq 0$. Multiplying the generalized dual compatibility \eqref{eq:dcondq} from the right with $\delta \bm{u}$, we get:
\begin{equation*}
    \left<\nabla_u \psi_q^i(\bm{u}_s), \delta \bm{u}\right> - \bm{f}^i(\bm{u})^T\mathcal{H}_q(\bm{u}_s) \delta \bm{u} = \left<c_q^i(\bm{u}_s)\nabla_u \mathcal{G}_q^i(\bm{u}_s), \delta \bm{u}\right>.
\end{equation*}
Since $\delta \bm{u} \in \ker(\mathcal{H}_q(\bm{u}_s))$, we have $\mathcal{H}_q(\bm{u}_s) \delta \bm{u} = 0$ and the expression simplifies to:
\begin{equation}
    \left<\nabla_u \psi_q^i(\bm{u}_s), \delta \bm{u}\right>  = c_q^i(\bm{u}_s) \left<\nabla_u \mathcal{G}_q^i(\bm{u}_s), \delta \bm{u}\right> \quad \forall \delta \bm{u} \in \ker(\mathcal{H}_q(\bm{u}_s)). \label{eq:solvacond}
\end{equation}
We can scale both sides of the equation with $||\delta \bm{u}||^{-1}$ to normalize $\delta \bm{u}$. Then \eqref{eq:solvacond} states that, at some point $\bm{u}_s \in \mathcal{U}$ where the kernel of $\mathcal{H}_q(\bm{u}_s)$ is non-trivial, \emph{the directional derivatives of $\psi_q^i$ and $\mathcal{G}_q^i$ in direction $\delta \bm{u}$ must become proportional with factor $c_q^i(\bm{u}_s)$}. We call this property \emph{null-consistency}. Note that these directional derivatives are taken in $\mathcal{U}$-space.

In \autoref{fig:solvacond}, an example of null-consistency is sketched for $n=2$, setting $d=1$ to forgo the need to indicate a spatial direction with superscripts. The rows of $\mathcal{H}_q(\bm{u}_s)$ are given by the gradients $\nabla_u\xi_1(\bm{u}_s), \nabla_u\xi_2(\bm{u}_s)$ on $\mathcal{U}$-space. Assuming $\delta \bm{u} \in \ker(\mathcal{H}_q(\bm{u}_s))$, we have that the gradients $\nabla_u\xi_1(\bm{u}_s), \nabla_u\xi_2(\bm{u}_s)$ are orthogonal to $\delta \bm{u}$ so that $f_1(\bm{u}_s) \nabla_u \xi_1(\bm{u}_s)$ and $f_2(\bm{u}_s) \nabla_u \xi_2(\bm{u}_s)$ are collinear. If $\nabla_u \psi_q(\bm{u}_s)$ is not orthogonal to $\delta \bm{u}$, the vector identity \eqref{eq:dcondq} can only hold if the component of $c_q(\bm{u}_s)\nabla_u \mathcal{G}_q(\bm{u}_s)$ along $\delta \bm{u}$ is equal to that of $\nabla_u \psi_q(\bm{u}_s)$ i.e.\ $\langle\nabla_u \psi_q(\bm{u}_s), \delta \bm{u}\rangle = \langle c_q(\bm{u}_s)\nabla_u \mathcal{G}_q(\bm{u}_s),\delta \bm{u}\rangle$. Since these are the components of gradients, it thus requires that the change of $\psi_q$ and $\mathcal{G}_q$ at $\bm{u}_s \in \mathcal{U}$ along $\delta \bm{u}$ becomes proportional through $c_q(\bm{u}_s)$. If $\nabla_u \xi_1(\bm{u}_s),\nabla_u \xi_2(\bm{u}_s) \neq 0$, $\delta \bm{u}$ must then also be tangent to the level curves of $\xi_1, \xi_2$ at $\bm{u}_s$.

\section{Structure-preserving numerical schemes}\label{sec:spns}
In this section, we develop a discrete framework that ensures a finite-volume method for a conservation law also satisfies a compatible discrete secondary balance law. This framework is a discrete analogue of the theoretical framework developed in the previous section. Additionally, it extends an existing discrete framework to ensure discrete entropy conservation laws \cite{tadmorentropy, tadmornumerical}, which we briefly review first.

\subsection{Entropy-conserving discretizations}\label{sec:ecdiscs}
We will discretize the conservation law \eqref{eq:cl} on a spatial grid, indexed by multi-indices $(i_1,...,i_d) \in \mathbb{N}^d$, with a finite-volume scheme:
\begin{equation}
    \dudt{\bm{u}_{\bm{I}}} + \sum_{j=1}^d \frac{1}{\Delta x_j}\left(\bm{f}_{\bm{I}+\bm{1}_j/2}^j - \bm{f}_{\bm{I}-\bm{1}_j/2}^j \right) = 0,
    \label{eq:disc}
\end{equation}
where $\bm{u}_{\bm{I}} : [0,t_f] \rightarrow \mathbb{R}^n$ is the solution in grid cell $\bm{I} := (i_1, ..., i_d)$ and $\Delta x_j > 0$ is the cell-width in coordinate direction $j$, which we assume uniform in direction $j$ to ease notation. We refer to neighboring grid cells with respect to grid cell $\bm{I}$ by adding a multi-index of zeros with a $1$ at position $j$, which we denote $\bm{1}_j$. For example, the index of the neighboring grid point in the first coordinate direction is $\bm{I}+\bm{1}_1 = (i_1+1,...,i_d)$. Finally, $\bm{f}^j_{\bm{I}+\bm{1}_j/2}$ is short for $\bm{f}_h(\bm{u}_{\bm{I}},\bm{u}_{\bm{I}+\bm{1}_j})$, where $\bm{f}_h : \mathcal{U}\times \mathcal{U} \rightarrow \mathbb{R}^n$ is a sufficiently smooth numerical flux function. We consider numerical flux functions which are consistent and symmetric in the following sense:
\begin{enumerate}
    \item consistency: $\bm{f}_h^j(\bm{u},\bm{u}) = \bm{f}^j(\bm{u}), \quad \forall \bm{u} \in \mathcal{U}$,
    \item symmetry: $\bm{f}_h^j(\bm{u},\bm{v}) = \bm{f}_h^j(\bm{v},\bm{u}), \quad \forall \bm{u},\bm{v} \in \mathcal{U}$.
\end{enumerate}
The framework due to Tadmor provides a third algebraic condition for a numerical flux in order to be entropy-conserving \cite{tadmorentropy, tadmornumerical}. We are interested in generalizing this framework for entropy-conserving schemes due to Tadmor to general secondary quantities. Therefore we first review this framework. We speak of discrete entropy conservation when the local entropy value $s_{\bm{I}} := s(\bm{u}_{\bm{I}})$ evolves according to a discrete conservation law of the same form as \eqref{eq:disc} and consistently discretizes \eqref{eq:cle}. The local entropy evolution must thus be of the following form:
\begin{equation}
    \dudt{s_{\bm{I}}} +\sum_{j=1}^d \frac{1}{\Delta x_j}\left(\mathcal{F}_{s,\bm{I}+\bm{1}_j/2}^j - \mathcal{F}_{s,\bm{I}-\bm{1}_j/2}^j \right) = 0,
    \label{eq:numcle}
\end{equation}
where $\mathcal{F}_{s,\bm{I}+\bm{1}_j/2}^j := \mathcal{F}_{s,h}^j(\bm{u}_{\bm{I}},\bm{u}_{\bm{I}+\bm{1}_j})$ and $\mathcal{F}_{s,h} : \mathcal{U} \times \mathcal{U} \rightarrow \mathbb{R}$ is a consistent and symmetric numerical approximation to the entropy flux $\mathcal{F}_s$ between the cells. Let $\bm{\eta}_{\bm{I}} := \bm{\eta}(\bm{u}_{\bm{I}})$ be a vector of local entropy variable values. To obtain \eqref{eq:numcle} it is then required to satisfy a discrete analogue of the primal compatibility condition:
\begin{equation}
    \bm{\eta}_{\bm{I}}^T\left(\bm{f}_{\bm{I}+\bm{1}_j/2}^j - \bm{f}_{\bm{I}-\bm{1}_j/2}^j \right) = \mathcal{F}_{s,\bm{I}+\bm{1}_j/2}^j - \mathcal{F}_{s,\bm{I}-\bm{1}_j/2}^j.
    \label{eq:numpconde}
\end{equation}
%This can be seen by the following discrete manipulations:
%\begin{align*}
%    0 &= \bm{\eta}_{\bm{I}}^T\left( \dudt{\bm{u}_{\bm{I}}} + \sum_{j=1}^d \frac{1}{\Delta x_j}\left(\bm{f}_{\bm{I}+\bm{1}_j/2}^j - \bm{f}_{\bm{I}-\bm{1}_j/2}^j \right)\right) \\
%    &= \bm{\eta}_{\bm{I}}^T\dudt{\bm{u}_{\bm{I}}} + \sum_{j=1}^d \frac{1}{\Delta x_j} \bm{\eta}_{\bm{I}}^T\left(\bm{f}_{\bm{I}+\bm{1}_j/2}^j - \bm{f}_{\bm{I}-\bm{1}_j/2}^j \right) \\
%    &= \dudt{s_{\bm{I}}} +\sum_{j=1}^d \frac{1}{\Delta x_j}\left(\mathcal{F}_{s,\bm{I}+\bm{1}_j/2}^j - \mathcal{F}_{s,\bm{I}-\bm{1}_j/2}^j \right).
%    \label{eq:entropyevo} 
%\end{align*} where in case the arguments originate from a grid function $\bm{u}_{\bm{I}}$ we will sometimes also use the short notation $\overline{g}_{\bm{I}+\bm{1}_j/2} := \overline{g}(\bm{u}_{\bm{I}},\bm{u}_{\bm{I}+\bm{1}_j})$. 
The discrete primal compatibility condition \eqref{eq:numpconde} is difficult to work with as the flux appears twice and three grid points are involved. Let $\overline{(\cdot)}$ and $\Delta (\cdot)$ applied to any function $g$ denote: 
\begin{equation}
    \overline{g}(\bm{u},\bm{v}) := \frac{g(\bm{v}) + g(\bm{u})}{2}, \quad \Delta g(\bm{u},\bm{v}) := g(\bm{v}) - g(\bm{u}),
    \label{eq:meananddiff}
\end{equation}
respectively. This issue can then be solved by defining the numerical flux as:
\begin{equation}
    \mathcal{F}_{s,h}(\bm{u}_{\bm{I}},\bm{u}_{\bm{I}+\bm{1}_j}) := \overline{\bm{\eta}}(\bm{u}_{\bm{I}}, \bm{u}_{\bm{I}+\bm{1}_j}) ^T\bm{f}_h^j(\bm{u}_{\bm{I}},\bm{u}_{\bm{I}+\bm{1}_j}) - \overline{\psi}_{s}^j(\bm{u}_{\bm{I}},\bm{u}_{\bm{I}+\bm{1}_j}),
    \label{eq:numeflux}
\end{equation}
which is clearly consistent if $\bm{f}_h^j$ is consistent, as can be seen from \eqref{eq:pote}. With the numerical entropy flux defined as in \eqref{eq:numeflux}, we can instead use a discrete dual compatibility condition:
\begin{equation}
    \bm{f}_h^j(\bm{u}_{\bm{I}},\bm{u}_{\bm{I}+\bm{1}_j})^T \Delta \bm{\eta}(\bm{u}_{\bm{I}}, \bm{u}_{\bm{I}+\bm{1}_j}) = \Delta \psi_s^j(\bm{u}_{\bm{I}}, \bm{u}_{\bm{I}+\bm{1}_j}).
    \label{eq:prototadmor}
\end{equation}
To show that this condition equivalently leads to \eqref{eq:numcle}, we give a proof in \autoref{app:ddual}. This condition is due to Tadmor and beares his name \cite{tadmornumerical}. For reference, we provide the following formal definition. Let $\bm{\psi}_s := [\psi_s^1,...,\psi_s^d]^T$. 
\begin{dfn}[Tadmor condition]
    \label{def:tadmor}
    A symmetric and consistent numerical flux function $\bm{f}_h^j : \mathcal{U} \times \mathcal{U} \rightarrow \mathbb{R}^n$ for a conservation law \eqref{eq:cl} with an entropy pair $(s,\bm{\mathcal{F}}_s)$ with entropy variables $\bm{\eta}$ and entropy flux potential $\bm{\psi}_s$ is said to satisfy the Tadmor condition if:
    \begin{equation}
        \Delta \psi_s^j(\bm{u},\bm{v})=\bm{f}_h^j(\bm{u},\bm{v})^T\Delta\bm{\eta}(\bm{u},\bm{v}),
        \label{eq:tadcond}
    \end{equation}
    for all pairs $\bm{u},\bm{v} \in \mathcal{U}$.
\end{dfn}
Note that for any $\bm{u},\bm{v} \in \mathcal{U}$ the set of vectors $\mathcal{P}_s^j(\bm{u},\bm{v}) := \{\bm{w}\in\mathbb{R}^n : \Delta \psi_s^j(\bm{u},\bm{v}) - \bm{w}^T\Delta\bm{\eta}(\bm{u},\bm{v}) = 0\}$ is non-empty, as the case $\Delta\bm{\eta}(\bm{u},\bm{v}) = 0$ implies $\bm{u}=\bm{v}$ and thus $\Delta \psi_s^j(\bm{u},\bm{v}) = 0$, due to injectivity of the entropy variables. Non-emptiness of $\mathcal{P}_s^j(\bm{u},\bm{v})$ for any $\bm{u},\bm{v} \in \mathcal{U}$ is a necessary condition for the existence of a numerical flux satisfying the Tadmor condition.

\subsection{Secondary-structure-consistent discretization}
Our goal is to establish similar conditions on the numerical flux functions $\bm{f}_h^j$ used in \eqref{eq:disc} so that, given a general secondary structure, the discretization \eqref{eq:disc} also implies a consistent discretization of the general balance law \eqref{eq:clq} for $q$. This discrete balance law should be conservative in the same sense as \eqref{eq:disc} if the secondary structure is conservative and otherwise has a clearly defined conservative and non-conservative part. Letting $q_{\bm{I}} := q(\bm{u}_{\bm{I}})$, we will look for discrete secondary balance laws of the form:
\begin{align}
    \dudt{q_{\bm{I}}} &+\sum_{j=1}^d \frac{1}{\Delta x_j}\left(\mathcal{F}_{q,\bm{I}+\bm{1}_j/2}^j - \mathcal{F}_{q,\bm{I}-\bm{1}_j/2}^j\right) + \sum_{j=1}^d \frac{1}{2}\left(\frac{\left[ c_q^j \Delta \mathcal{G}_q^j \right]_{\bm{I}+\bm{1}_j/2}}{\Delta x_j} + \frac{\left[ c_q^j \Delta \mathcal{G}_q^j \right]_{\bm{I}-\bm{1}_j/2}}{\Delta x_j} \right) = 0.\label{eq:discq} 
\end{align}
This form is a consistent discretization of \eqref{eq:clq} and, as we will show, follows naturally from an extension of the approach taken by Tadmor for entropy \cite{tadmorentropy, tadmornumerical}. In particular, $\mathcal{F}_{q,\bm{I}+\bm{1}_j/2}^j := \mathcal{F}_{q,h}^j(\bm{u}_{\bm{I}},\bm{u}_{\bm{I}+\bm{1}_j})$ where $\mathcal{F}_{q,h} : \mathcal{U} \times \mathcal{U} \rightarrow \mathbb{R}$ is a consistent and symmetric numerical approximation to the conservative flux $\mathcal{F}_q$ of \eqref{eq:clq} between grid cells. Furthermore, $\left[ c_q^j \Delta \mathcal{G}_q^j \right]_{\bm{I}+\bm{1}_j/2} := \left[ c_q^j \Delta \mathcal{G}_q^j \right](\bm{u}_{\bm{I}},\bm{u}_{\bm{I}+\bm{1}_j})$ where $\left[ c_q^j \Delta \mathcal{G}_q^j \right] : \mathcal{U} \times \mathcal{U} \rightarrow \mathbb{R}$ is an anti-symmetric function that, when divided by the corresponding cell-width $\Delta x_j$, approximates the non-conservative term in \eqref{eq:clq} consistently between the cells. Here, anti-symmetric is meant in the sense that, for all $\bm{u},\bm{v}\in\mathcal{U}$:
\begin{equation*}
    \left[ c_q^j \Delta \mathcal{G}_q^j \right](\bm{u},\bm{v}) = - \left[ c_q^j \Delta \mathcal{G}_q^j \right](\bm{v},\bm{u}).
\end{equation*}
When deriving a generalization of \eqref{eq:prototadmor}, it will become clear why this anti-symmetry is necessary. As $\left[ c_q^j \Delta \mathcal{G}_q^j \right]$ often describes a form of physical work performed on or by the system, we will refer to $\left[ c_q^j \Delta \mathcal{G}_q^j \right]$ as the \textit{numerical work term}. For reasons that will become clear later, we do not fully specify the work term; we only specify its behavior. 

We will proceed in a similar fashion as for entropy to derive a generalization of the Tadmor condition \eqref{eq:tadcond}. Let $\bm{\xi}_{\bm{I}} := \bm{\xi}(\bm{u}_{\bm{I}})$, then a discrete generalized primal compatibility condition can be stated as follows:
\begin{align}
    &\bm{\xi}_{\bm{I}}^T\left(\bm{f}_{\bm{I}+\bm{1}_j/2}^j - \bm{f}_{\bm{I}-\bm{1}_j/2}^j \right)  \label{eq:numpcondq} \\
    & \qquad = \mathcal{F}_{q,\bm{I}+\bm{1}_j/2}^j - \mathcal{F}_{q,\bm{I}-\bm{1}_j/2}^j + \frac{1}{2}\left[ c_q^j \Delta \mathcal{G}_q^j \right]_{\bm{I}+\bm{1}_j/2} + \frac{1}{2}\left[ c_q^j \Delta \mathcal{G}_q^j \right]_{\bm{I}-\bm{1}_j/2}. \nonumber
\end{align}
The manipulations to obtain \eqref{eq:discq} from the discrete generalized primal compatibility condition are straightforward and mimic exactly the procedure at the differential level \eqref{eq:diffprimderiv}:
\begin{align*}
    0 &= \bm{\xi}_{\bm{I}}^T\left( \dudt{\bm{u}_{\bm{I}}} + \sum_{j=1}^d \frac{1}{\Delta x_j}\left(\bm{f}_{\bm{I}+\bm{1}_j/2}^j - \bm{f}_{\bm{I}-\bm{1}_j/2}^j \right)\right) \\
    &= \bm{\xi}_{\bm{I}}^T\dudt{\bm{u}_{\bm{I}}} + \sum_{j=1}^d \frac{1}{\Delta x_j} \bm{\xi}_{\bm{I}}^T\left(\bm{f}_{\bm{I}+\bm{1}_j/2}^j - \bm{f}_{\bm{I}-\bm{1}_j/2}^j \right) \\
    &= \dudt{q_{\bm{I}}} +\sum_{j=1}^d \frac{1}{\Delta x_j}\left(\mathcal{F}_{q,\bm{I}+\bm{1}_j/2}^j - \mathcal{F}_{q,\bm{I}-\bm{1}_j/2}^j\right) + \sum_{j=1}^d \frac{1}{2}\left(\frac{\left[ c_q^j \Delta \mathcal{G}_q^j \right]_{\bm{I}+\bm{1}_j/2}}{\Delta x_j} + \frac{\left[ c_q^j \Delta \mathcal{G}_q^j \right]_{\bm{I}-\bm{1}_j/2}}{\Delta x_j} \right). 
\end{align*}
However, the discrete generalized primal compatibility condition is cumbersome to work with and again requires the consideration of three grid points. As for entropy, the generalized dual compatibility condition provides a solution. Defining a conservative numerical $q$-flux as:
\begin{equation}
    \mathcal{F}_{q,h}(\bm{u}_{\bm{I}},\bm{u}_{\bm{I}+\bm{1}_j}) :=\overline{\bm{\xi}}(\bm{u}_{\bm{I}},\bm{u}_{\bm{I}+\bm{1}_j})^T\bm{f}_h^j(\bm{u}_{\bm{I}},\bm{u}_{\bm{I}+\bm{1}_j}) - \overline{\psi}_q^j(\bm{u}_{\bm{I}},\bm{u}_{\bm{I}+\bm{1}_j}),
    \label{eq:numqflux}
\end{equation}
we can define the following \emph{discrete generalized dual compatibility condition}:
\begin{equation}
    \Delta \psi_q^j(\bm{u}_{\bm{I}},\bm{u}_{\bm{I}+\bm{1}_j}) - \bm{f}_h^j(\bm{u}_{\bm{I}},\bm{u}_{\bm{I}+\bm{1}_j})^T \Delta \bm{\xi}(\bm{u}_{\bm{I}},\bm{u}_{\bm{I}+\bm{1}_j}) = \left[ c_q^j \Delta \mathcal{G}_q^j \right](\bm{u}_{\bm{I}},\bm{u}_{\bm{I}+\bm{1}_j}).
    \label{eq:numdcondq}
\end{equation}
Here, we see that the anti-symmetry of $\left[ c_q^j \Delta \mathcal{G}_q^j \right]$ is necessary. Specifically, both $\Delta \psi_q^j$ and $\Delta \bm{\xi}$ are anti-symmetric functions, while the flux $\bm{f}_h^j$ is assumed symmetric, hence the left side of \eqref{eq:numdcondq} is anti-symmetric. Equality then necessarily forces the numerical work term $\left[ c_q^j \Delta \mathcal{G}_q^j \right]$ to be anti-symmetric. Before we show that \eqref{eq:numdcondq} leads to \eqref{eq:discq}, we introduce some short notation and useful identities. When the arguments of $g$ in \eqref{eq:meananddiff} originate from the grid function $\bm{u}_{\bm{I}}$, it will be useful to define the following even more compact notation:
\begin{align*}
    \overline{g}_{\bm{I}+\bm{1}_j/2} := \overline{g}(\bm{u}_{\bm{I}}, \bm{u}_{\bm{I}+\bm{1}_j}), \quad \Delta g_{\bm{I}+\bm{1}_j/2} &:= \Delta g(\bm{u}_{\bm{I}}, \bm{u}_{\bm{I}+\bm{1}_j}).
\end{align*}
Two identities also hold between the $\overline{(\cdot)}$ and $\Delta(\cdot)$ operators, which we denote in the compact grid function notation, as that is the setting in which we will mainly use them:
\begin{gather}
    g_{\bm{I}} = \overline{g}_{\bm{I}\pm\bm{1}_j/2} \mp \frac{1}{2}\Delta g_{\bm{I}\pm\bm{1}_j/2}, \label{eq:usefulid} \\
    \frac{1}{2}\Delta g_{\bm{I}+\bm{1}_j/2}+\frac{1}{2}\Delta g_{\bm{I}-\bm{1}_j/2} = \overline{g}_{\bm{I}+\bm{1}_j/2} - \overline{g}_{\bm{I}-\bm{1}_j/2}.
    \label{eq:commdifavg}
\end{gather}
We are now ready to provide the following derivation to show that \eqref{eq:numdcondq} leads to \eqref{eq:discq}, which is one of our main contributions in this article:
\begin{align*}
    0 &= \bm{\xi}_{\bm{I}}^T\left( \dudt{\bm{u}_{\bm{I}}} + \sum_{j=1}^d \frac{1}{\Delta x_j}\left(\bm{f}_{\bm{I}+\bm{1}_j/2}^j - \bm{f}_{\bm{I}-\bm{1}_j/2}^j \right)\right) \\
    &= \bm{\xi}_{\bm{I}}^T\dudt{\bm{u}_{\bm{I}}} + \sum_{j=1}^d \frac{1}{\Delta x_j} \bm{\xi}_{\bm{I}}^T\left(\bm{f}_{\bm{I}+\bm{1}_j/2}^j - \bm{f}_{\bm{I}-\bm{1}_j/2}^j \right) \\
    &\overset{\eqref{eq:usefulid}}{=} \dudt{q_{\bm{I}}} + \sum_{j=1}^d \frac{1}{\Delta x_j} \left(\overline{\bm{\xi}}_{\bm{I}+\bm{1}_j/2}^T\bm{f}_{\bm{I}+\bm{1}_j/2}^j - \overline{\bm{\xi}}_{\bm{I}-\bm{1}_j/2}^T\bm{f}_{\bm{I}-\bm{1}_j/2}^j\right) \\
    &\qquad \qquad \qquad - \sum_{j=1}^d \frac{1}{\Delta x_j} \frac{1}{2}\left(\left(\bm{f}_{\bm{I}+\bm{1}_j/2}^j\right)^T\Delta \bm{\xi}_{\bm{I}+\bm{1}_j/2} + \left(\bm{f}_{\bm{I} - \bm{1}_j/2}^j\right)^T\Delta \bm{\xi}_{\bm{I}-\bm{1}_j/2}  \right) \\
    &\overset{\eqref{eq:numqflux}}{=} \dudt{q_{\bm{I}}} + \sum_{j=1}^d \frac{1}{\Delta x_j} \left(\mathcal{F}_{q,\bm{I}+\bm{1}_j/2}^j - \mathcal{F}_{q,\bm{I}-\bm{1}_j/2}^j \right) + \sum_{j=1}^d \frac{1}{\Delta x_j}\left(\overline{\psi}_{q, \bm{I}+\bm{1}_j/2}^j - \overline{\psi}_{q,\bm{I}-\bm{1}_j/2}^j\right) \\
    &\qquad \qquad \qquad - \sum_{j=1}^d \frac{1}{\Delta x_j} \frac{1}{2}\left(\left(\bm{f}_{\bm{I}+\bm{1}_j/2}^j\right)^T\Delta \bm{\xi}_{\bm{I}+\bm{1}_j/2} + \left(\bm{f}_{\bm{I}-\bm{1}_j/2}^j\right)^T\Delta \bm{\xi}_{\bm{I}-\bm{1}_j/2}  \right) \\
    &\overset{\eqref{eq:commdifavg}}{=} \dudt{q_{\bm{I}}} + \sum_{j=1}^d \frac{1}{\Delta x_j} \left(\mathcal{F}_{q,\bm{I}+\bm{1}_j/2}^j - \mathcal{F}_{q,\bm{I}-\bm{1}_j/2}^j \right) + \sum_{j=1}^d \frac{1}{\Delta x_j}\frac{1}{2}\left(\Delta{\psi}_{q, \bm{I}+\bm{1}_j/2}^j + \Delta{\psi}_{q,\bm{I}-\bm{1}_j/2}^j\right) \\
    &\qquad \qquad \qquad - \sum_{j=1}^d \frac{1}{\Delta x_j} \frac{1}{2}\left(\left(\bm{f}_{\bm{I}+\bm{1}_j/2}^j\right)^T\Delta \bm{\xi}_{\bm{I}+\bm{1}_j/2} + \left(\bm{f}_{\bm{I}-\bm{1}_j/2}^j\right)^T\Delta \bm{\xi}_{\bm{I}-\bm{1}_j/2}  \right) \\
    &\overset{\eqref{eq:numdcondq}}{=} \dudt{q_{\bm{I}}} + \sum_{j=1}^d \frac{1}{\Delta x_j} \left(\mathcal{F}_{q,\bm{I}+\bm{1}_j/2}^j - \mathcal{F}_{q,\bm{I}-\bm{1}_j/2}^j \right) + \sum_{j=1}^d \frac{1}{2}\left(\frac{\left[ c_q^j \Delta \mathcal{G}_q^j \right]_{\bm{I}+\bm{1}_j/2}}{\Delta x_j} + \frac{\left[ c_q^j \Delta \mathcal{G}_q^j \right]_{\bm{I}-\bm{1}_j/2}}{\Delta x_j} \right), %\\ & \qquad \qquad \qquad 
\end{align*}
where in the third equality the identity \eqref{eq:usefulid} was used, in the fourth equality definition \eqref{eq:numqflux}, in the fifth equality identity \eqref{eq:commdifavg} and in the final equality the discrete generalized dual compatibility condition \eqref{eq:numdcondq}. The derivation again closely mimics the differential setting, where the differences in the residual function $\Delta \psi_q^j$ now are used to correct the $\left(\bm{f}_{\bm{I}+\bm{1}_j/2}^j\right)^T\Delta \bm{\xi}_{\bm{I}+\bm{1}_j/2}$ terms to the correct numerical work terms instead of canceling them out. Unlike for entropy, however, we are not yet ready to define a generalization of the Tadmor condition at this point. We must still deal with possible degeneracies that can occur in \eqref{eq:numdcondq} due to the lack of a convexity requirement on $q$. 

\subsection{Generalized Tadmor condition and weak discrete null-consistency}\label{sec:discnull}
Inspecting the discrete dual compatibility condition \eqref{eq:numdcondq}, we can note the following. Evaluating \eqref{eq:numdcondq} at some pair $\bm{u},\bm{v}\in \mathcal{U}$ gives a relation between the numerical work term $[c_q^j\Delta \mathcal{G}_q^j]$, change in $\psi_q^h$ and $\bm{\xi}$ between $\bm{u}, \bm{v}\in\mathbb{R}^n$ and the form of the numerical flux $\bm{f}_h^j(\bm{u},\bm{v})$. In this relation the form of the numerical flux $\bm{f}_h^j$ and of the numerical work term $\left[c_q^j \Delta \mathcal{G}_q^j \right]$ are degrees of freedom that can be determined to satisfy the discrete dual compatibility in the respective $\bm{v}-\bm{u}$ direction. 

However, the fact that for general secondary structures $q$ is not necessarily a convex function, puts constraints on the form of the numerical work term $\left[ c_q^j \Delta \mathcal{G}_q^j \right]$. Namely, where the entropy variables $\bm{\eta}$ are injective due to the convexity of the entropy function $s$ as mentioned in \autoref{sec:entanalysis}, $\bm{\xi}$ can fail to be injective. This is the case as we assume there is no such convexity constraint on $q$. Consequently, there can exist pairs of points $\bm{u},\bm{v}\in \mathcal{U}$ so that $\Delta \bm{\xi}(\bm{u},\bm{v}) = 0$ while $\bm{u}\neq\bm{v}$. In such a case, the contribution of the numerical flux to the discrete dual compatibility condition \eqref{eq:numdcondq} drops out like in \eqref{eq:solvacond}. When this occurs, equality should still hold in \eqref{eq:numdcondq} if we want our discretization \eqref{eq:disc} to satisfy \eqref{eq:discq}. Therefore, when $\Delta \bm{\xi}_{\bm{I}+\bm{1}_j/2} = 0$, we must require:
\begin{equation}
    \Delta \psi_{q,\bm{I}+\bm{1}_j/2}^j = \left[ c_q^j \Delta \mathcal{G}_q^j \right]_{\bm{I}+\bm{1}_j/2}.
    \label{eq:dxizero}
\end{equation}
This is a discrete analogue to the null-consistency property \eqref{eq:solvacond} of general secondary structures for the non-conservative numerical work term. To obtain a scheme that is always consistent with \eqref{eq:discq}, we \emph{must} require this discrete null-consistency as an additional constraint. As we will see later, it will also be needed to consider a stronger form of discrete null-consistency, hence we refer to this form as \textit{weak discrete null-consistency}. We provide the following more precise definition.
\begin{dfn}[Weak discrete null-consistency]
    \label{def:discnull}
    An anti-symmetric and consistent numerical work term $\left[ c_q^j \Delta \mathcal{G}_q^j\right] : \mathcal{U}\times\mathcal{U}\rightarrow \mathbb{R}$ for a discretization of a conservation law \eqref{eq:cl} with a secondary structure $(q,\bm{\mathcal{F}}_q,\bm{\mathcal{G}}_q,\bm{c}_q)$ with gradient $\bm{\xi}$ and residual function $\bm{\psi}_q$ is said to satisfy the weak discrete null-consistency property if the following implication holds:
    \begin{equation}
        \Delta \bm{\xi}(\bm{u},\bm{v}) = 0 \implies \Delta \psi_q^j(\bm{u},\bm{v}) = \left[ c_q^j \Delta \mathcal{G}_q^j\right](\bm{u},\bm{v}), 
        \label{eq:weaknullcons}
    \end{equation}
    for all $\bm{u},\bm{v} \in \mathcal{U}$.
\end{dfn}

When a discretely null-consistent numerical work term $\left[c_q^j \Delta \mathcal{G}_q^j\right]$ is found, then for any $\bm{u},\bm{v}\in \mathcal{U}$ the set $\mathcal{P}_q^j(\bm{u},\bm{v}) := \{\bm{w}\in\mathbb{R}^n : \Delta \psi_q^j(\bm{u},\bm{v}) - \bm{w}^T\Delta\bm{\xi}(\bm{u},\bm{v}) = \left[c_q^j \Delta \mathcal{G}_q^j\right](\bm{u},\bm{v})\}$ is never empty, as when $\Delta \bm{\xi}(\bm{u},\bm{v})=0$ we have $\mathcal{P}_q^j(\bm{u},\bm{v}) = \mathbb{R}^n$ and otherwise $\mathcal{P}_q^j(\bm{u},\bm{v})$ is simply a hyper-plane. The non-emptiness of $\mathcal{P}_q^j(\bm{u},\bm{v})$ for all $\bm{u},\bm{v}\in \mathcal{U}$ is a necessary condition for the existence of a flux $\bm{f}_h^j$ that always satisfies \eqref{eq:numdcondq}. We are now ready to generalize the Tadmor condition to general secondary structures.
\begin{dfn}[Generalized Tadmor condition]
    \label{def:gentad}
    A symmetric and consistent numerical flux function $\bm{f}_h^j : \mathcal{U} \times \mathcal{U} \rightarrow \mathbb{R}^n$ for a conservation law \eqref{eq:cl} with a secondary structure $(q,\bm{\mathcal{F}}_q,\bm{\mathcal{G}}_q,\bm{c}_q)$ with $q$-gradient $\bm{\xi}$ and residual function $\bm{\psi}_q$ is said to satisfy the generalized Tadmor condition if:
    \begin{equation}
        \Delta\psi_q^j(\bm{u},\bm{v}) -\bm{f}_h^j(\bm{u},\bm{v})^T\Delta \bm{\xi}(\bm{u},\bm{v}) = \left[ c_q^j \Delta \mathcal{G}_q^j\right](\bm{u},\bm{v}), \quad \forall \bm{u},\bm{v} \in \mathcal{U},
        \label{eq:gentadcond}
    \end{equation}
    where $\left[ c_q^j \Delta \mathcal{G}_q^j\right] : \mathcal{U} \times \mathcal{U} \rightarrow \mathbb{R}$ is an anti-symmetric and consistent numerical work term satisfying weak discrete null-consistency \eqref{eq:weaknullcons}.
\end{dfn}
This condition reduces to the standard Tadmor condition, \autoref{def:tadmor}, if the secondary structure is an entropy pair. We have proven the following theorem.
\begin{thm}[Discrete balance law]\label{thm:balance}
    Consider a conservation law of the form \eqref{eq:cl} with a secondary structure $(q,\bm{\mathcal{F}}_q,\bm{\mathcal{G}}_q,\bm{c}_q)$. Let a numerical scheme of the form \eqref{eq:disc} be given to approximate \eqref{eq:cl}. Denote the symmetric and consistent numerical flux functions $\bm{f}_h^j : \mathcal{U} \times \mathcal{U} \rightarrow \mathbb{R}^n$, $j=1,...,d$. The scheme satisfies the additional discrete balance law \eqref{eq:discq} with given anti-symmetric and consistent numerical work terms $\left[ c_q^j \Delta \mathcal{G}_q^j\right] : \mathcal{U}\times\mathcal{U}\rightarrow \mathbb{R}$, $j=1,...,d$, satisfying the weak discrete null-consistency property \eqref{eq:weaknullcons}, if the numerical flux functions $\bm{f}_h^j$ satisfy the associated generalized Tadmor condition \eqref{eq:gentadcond}.
\end{thm}

\begin{rmk}[Satisfying discrete null-consistency]
    Although at the differential level, null-consistency is automatically satisfied, the discrete analogue \eqref{eq:weaknullcons} seems to impose a strong condition on the numerical work term. This is the case as \eqref{eq:weaknullcons} requires non-local\footnote{i.e.\ between distinct points $\bm{u},\bm{v}\in\mathcal{U}$} relations to be satisfied by the functions $\bm{\xi}$, $\psi_q^j$, and the functions involved in constructing $[c_q^j\Delta \mathcal{G}_q^j]$. However, these functions are only locally related to each other using the dual compatibility condition \eqref{eq:dcondq}. It is for this reason that we have, so far, not fully specified $\left[c_q^j \Delta \mathcal{G}_q^j \right]$ as this provides maximal flexibility in satisfying \eqref{eq:weaknullcons}. It is interesting to note that for the classical Tadmor condition for entropy \eqref{eq:tadcond}, it is possible to obtain such a non-local relation from a local relation between $\psi_s^j$ and $\bm{\eta}$. Specifically, the entropy-conserving midpoint flux \cite[Equation 4.6a]{tadmornumerical} turns the local dual compatibility condition for entropy functions into a non-local Tadmor condition using a path-independent integral. So far, however, we have not been able to apply a similar approach to constructing weakly null-consistent numerical work terms.
\end{rmk}

\subsection{The method of jump expansions using discrete gradients}\label{sec:jumpexps}
Having derived a generalized Tadmor condition, we now propose and analyze a method for finding solutions given a secondary structure. The usual method for solving Tadmor's entropy conservation condition is the method of jump expansions \cite{ranochathesis, fjordholmenergy, chandrashekarkinetic, ranochacomparison}. The method is systematic and applicable to any entropy-conserving system of conservation laws. For these reasons, we will generalize this approach to the generalized Tadmor condition and analyze some of its properties in this setting. 

In our applications, however, a problem immediately arises. The method of jump expansions requires explicit knowledge of all involved functions $\bm{\psi}_s$ and $\bm{\eta}$. In our application of interest, compressible flow of real gases, these functions $\bm{\psi}_q$ and $\bm{\xi}$ and the additional functions $\bm{c}_q$ and $\bm{\mathcal{G}}_q$ may only be defined up to an arbitrary function. For us, this arbitrary function is specified by a choice of thermodynamic model, i.e.\ an equation of state, as will be explained in \autoref{sec:fluxeszz}. To solve this, we will introduce discrete gradient operators \cite{gonzaleztime, itohhamiltonian, mclachlangeometric}, a discretization method that enables us to apply jump expansions to secondary structures defined only up to an arbitrary function.

The method of jump expansions finds a numerical flux by deriving an $n \times n$ linear system for which a solution satisfies the scalar Tadmor condition \eqref{eq:tadcond} for any fixed $\bm{u},\bm{v} \in \mathcal{U}$. The linear system is constructed by expanding the jumps involved in the Tadmor condition with respect to a common set of auxiliary variables $\bm{z} : \mathcal{U} \rightarrow \mathbb{R}^n$ obtained from a change-of-variables (sufficiently smooth and bijective with smooth inverse). Discrete gradient operators can be very helpful in these expansions. We will first focus on constructing these expansions, then analyze the properties of the resulting linear system for general secondary structures. This analysis is mainly based on functions defined on $\mathcal{U}$, without any spatial considerations related to the physical domain. Hence, for better readability, we will suppress the superscript $(\cdot)^j$ denoting physical coordinate directions throughout this subsection. 

\subsubsection{Jump expansions}
For the method of jump expansions, it is crucial that the Tadmor condition consists of a linear combination of $\Delta (\cdot)$ expressions with coefficients consisting of the numerical flux components. To ensure the same approach works in our generalized setting we assume that the numerical work term takes the form $[c_q \Delta \mathcal{G}_q](\bm{u},\bm{v}) = c_q(\bm{u},\bm{v})\Delta \mathcal{G}_q(\bm{u},\bm{v})$, where $c_q(\bm{u},\bm{v})$ is any symmetric and consistent average of $c_q$. However, more generally, the approach also works for a sum of multiple jumps. If the numerical work term is not a sum of $\Delta (\cdot)$ terms the analysis and method described in what follows does not apply. 

For the functions involved in the generalized Tadmor condition, jump expansions generally take the following form:
\begin{equation*}
    \Delta g(\bm{u},\bm{v}) = \sum_{i=1}^n a_{g,i} (\bm{u},\bm{v})\Delta z_i(\bm{u},\bm{v}), \\
\end{equation*}
where $g \in \{\psi_q, \xi_1, ..., \xi_n, \mathcal{G}_q\}$ and $a_{g,i} : \mathcal{U} \times \mathcal{U} \rightarrow \mathbb{R}$ are expansion coefficients, constructed by suitable applications of the mean value theorem (cf. \cite{ranochathesis}). This approach of constructing expansion coefficients is particularly effective when jumps in univariate functions of $z_i$ for some $i$ can be isolated during the expansion process. Moreover, this approach is not possible when functions are not explicitly or only partially defined, as in our case. By partially defined we mean for example that $g(\bm{z}) = 3z_1^2 + 2z_2 + g_u(\bm{z})$ where $g_u$ is some arbitrary user-supplied function. With some algebra the jumps $\Delta (3z_1^2)$ and $\Delta (2z_2)$ are easily expanded in terms of $\Delta z_1$ and $\Delta z_2$, however the user-supplied function $g_u$ cannot be expanded. 

To still apply jump expansions in this partially defined and/or multivariate setting, we propose using discrete gradient operators \cite{gonzaleztime, itohhamiltonian, mclachlangeometric}, which are gradient discretization methods satisfying the following definition \cite{gonzaleztime, mclachlangeometric}. Let $\mathcal{P}$ be an open subset of $\mathbb{R}^m$.
\begin{dfn}[Discrete gradient functions]
    \label{def:discgrad}
    A discrete gradient for a differentiable function $g : \mathcal{P} \rightarrow \mathbb{R}$ is a continuous mapping $\widetilde{\nabla} g : \mathcal{P} \times \mathcal{P} \rightarrow \mathbb{R}^m$ with the following properties.
    \begin{enumerate}
        \item Discrete gradient, $\widetilde{\nabla} g(\bm{p}_1, \bm{p}_2)^T (\bm{p}_2-\bm{p}_1) =\Delta g(\bm{p}_1, \bm{p}_2) \quad   \forall \bm{p}_1,\bm{p}_2 \in \mathcal{P}$.
        \item Consistency, $\widetilde{\nabla} g(\bm{p},\bm{p}) = \nabla g(\bm{p}) \quad \forall \bm{p} \in \mathcal{P}$.
    \end{enumerate}
\end{dfn}
Note that, according to the discrete gradient property in \autoref{def:discgrad}, the components of discrete gradients can precisely be used as expansion coefficients in the jump expansions. For our convenience we denote $\widetilde{\nabla}_{\bm{z}} g(\bm{u},\bm{v}) := \widetilde{\nabla}g(\bm{z}(\bm{u}),\bm{z}(\bm{v}))$ when the gradient is with respect to $\bm{z}$. We can thus use discrete gradients with respect to $\bm{z}$ to construct our expansions:
\begin{align}
    \Delta \psi_q(\bm{u},\bm{v}) &= \widetilde{\nabla}_{\bm{z}}\psi_q(\bm{u},\bm{v})^T\Delta \bm{z}(\bm{u},\bm{v}), \nonumber \\
    \Delta \bm{\xi}(\bm{u},\bm{v}) &= \widetilde{\nabla}_{\bm{z}}\bm{\xi}(\bm{u},\bm{v})^T\Delta \bm{z}(\bm{u},\bm{v}), \label{eq:dgexps}\\
    \Delta \mathcal{G}_q(\bm{u},\bm{v}) &= \widetilde{\nabla}_{\bm{z}}\mathcal{G}_q(\bm{u},\bm{v})^T\Delta \bm{z}(\bm{u},\bm{v}). \nonumber
\end{align}
Here $\widetilde{\nabla}_{\bm{z}}\bm{\xi}(\bm{u},\bm{v}) \in \mathbb{R}^{n\times n}$ contains the discrete gradient of $\xi_i$ as its $i$-th column. Discrete gradient operators are operators that, given any arbitrary sufficiently smooth function $g$, produce a valid discrete gradient function $\widetilde{\nabla}g$. We will see that this property lets us define, e.g., entropy-conserving fluxes for general equations of state. For $n>1$, discrete gradient operators are not unique \cite{mclachlangeometric}, so we list a number of classical and/or practically useful examples in \autoref{app:dgs}. 

\subsubsection{Linear system}
We will now focus on the construction and properties of the linear system. By substituting the discrete gradient expansions \eqref{eq:dgexps} into the generalized Tadmor condition \eqref{eq:gentadcond}, we obtain the following expression:
\begin{equation*}
    \left(\widetilde{\nabla}_{\bm{z}}\psi_q(\bm{u},\bm{v}) - c_q(\bm{u},\bm{v})\widetilde{\nabla}_{\bm{z}}\mathcal{G}_q(\bm{u},\bm{v}) -   \widetilde{\nabla}_{\bm{z}}\bm{\xi}(\bm{u},\bm{v})\bm{f}_h(\bm{u},\bm{\bm{v}})\right)^T \Delta \bm{z}(\bm{u},\bm{v}) = 0.
\end{equation*}
A linear system that can be solved for a numerical flux can then be obtained by requiring:
\begin{equation}
    \widetilde{\nabla}_{\bm{z}}\psi_q(\bm{u},\bm{v}) - c_q(\bm{u},\bm{v})\widetilde{\nabla}_{\bm{z}}\mathcal{G}_q(\bm{u},\bm{v}) -   \widetilde{\nabla}_{\bm{z}}\bm{\xi}(\bm{u},\bm{v})\bm{f}_h(\bm{u},\bm{\bm{v}}) = 0. \label{eq:jumpexpsys}
\end{equation}
This system and, if they exist, its solutions, depend on the choice of auxiliary variables $\bm{z}(\bm{u})$ through the gradients and the specific form of the discrete gradient functions. Thus, different numerical fluxes can potentially be obtained for different expansion variables $\bm{z}$ and discrete gradient functions \cite{chandrashekarkinetic, ranochacomparison}. In the case of entropy pairs, by consistency of the discrete gradient, we know that at $\bm{v}=\bm{u}$ a unique solution exists (the differential flux $\bm{f}(\bm{u})$) since $\nabla_{\bm{u}} \bm{z}(\bm{u})$ and $\nabla_{\bm{u}}\bm{\eta}(\bm{u})$ are invertible. By continuity of the determinant and discrete gradients, this argument extends to $\bm{v}$ for which $||\bm{u}-\bm{v}||$ is sufficiently small. Such an argument cannot be made for general secondary structures as $\nabla_{\bm{u}} \bm{\xi}(\bm{u})$ can generally not be assumed to be non-singular. We should therefore not expect to recover unique fluxes from a single generalized Tadmor condition, but rather to obtain general solution sets. From these solution sets, numerical fluxes can be chosen that satisfy other desirable properties or design criteria.

\subsubsection{Strong discrete null-consistency}
For fixed $\bm{u},\bm{v} \in \mathcal{U}$, the system in \eqref{eq:jumpexpsys} has at least one solution if $\widetilde{\nabla}_{\bm{z}}\psi_q(\bm{u},\bm{v}) - c_q(\bm{u},\bm{v})\widetilde{\nabla}_{\bm{z}}\mathcal{G}_q(\bm{u},\bm{v}) \in \text{col}(\widetilde{\nabla}_{\bm{z}}\bm{\xi}(\bm{u},\bm{v}))$. This condition is, in fact, a stronger form of discrete null-consistency as is given in definition \eqref{def:discnull}. To see this, we note that a vector is an element of $\text{col}(\widetilde{\nabla}_{\bm{z}}\bm{\xi}(\bm{u},\bm{v}))$ if and only if it is orthogonal to any element of $\text{ker}(\widetilde{\nabla}_{\bm{z}}\bm{\xi}(\bm{u},\bm{v})^T)$. Thus for existence of a solution at $\bm{u},\bm{v}\in\mathcal{U}$, we require that for all $\delta \bm{z} \in \ker(\widetilde{\nabla}_{\bm{z}}\bm{\xi}(\bm{u},\bm{v})^T)$:
\begin{equation}
    \widetilde{\nabla}_{\bm{z}} \psi_q(\bm{u},\bm{v})^T\delta \bm{z} = c_q(\bm{u},\bm{v})\widetilde{\nabla}_{\bm{z}} \mathcal{G}_q(\bm{u},\bm{v})^T \delta \bm{z},
    \label{eq:strongdiscnull}
\end{equation}
must hold. From this relation, \eqref{eq:solvacond} can immediately be recognized. In contrast to the scalar condition for weak discrete null-consistency, this is a full discrete vector analogue of \eqref{eq:solvacond}. This \textit{strong discrete null-consistency} property implies weak discrete null-consistency. Assume $\bm{u},\bm{v}\in\mathcal{U}$ are such that $\Delta \bm{\xi}(\bm{u},\bm{v}) = 0$, then, by the discrete gradient property of discrete gradient functions:
\begin{equation*}
    0 =\Delta \bm{\xi}(\bm{u},\bm{v}) = \widetilde{\nabla}_{\bm{z}}\bm{\xi}(\bm{u},\bm{v})^T \Delta \bm{z}(\bm{u},\bm{v}),
\end{equation*}
so $\Delta \bm{z}(\bm{u},\bm{v}) \in \ker(\widetilde{\nabla}_{\bm{z}}\bm{\xi}(\bm{u},\bm{v})^T)$. Hence, by strong discrete null-consistency:
\begin{align*}
    \Delta \psi_q(\bm{u},\bm{v}) &= \widetilde{\nabla}_{\bm{z}} \psi_q(\bm{u},\bm{v})^T\Delta \bm{z}(\bm{u},\bm{v}) \\
    &= c_q(\bm{u},\bm{v})\widetilde{\nabla}_{\bm{z}} \mathcal{G}_q(\bm{u},\bm{v})^T \Delta \bm{z}(\bm{u},\bm{v}) = c_q(\bm{u},\bm{v}) \Delta \mathcal{G}_q(\bm{u},\bm{v}),
\end{align*}
and so the implication of weak null-consistency holds. The extra constraints that need to be satisfied in addition to weak discrete null-consistency are a consequence of the fact that we are now solving a linear system \eqref{eq:jumpexpsys} instead of the scalar generalized Tadmor condition \eqref{eq:gentadcond}. 

\subsubsection{Flux symmetry and consistency}
If they exist, solutions to \eqref{eq:jumpexpsys} satisfy a certain notion of symmetry. In particular, if symmetric discrete gradients (i.e.\ $\widetilde{\nabla} g(\bm{x}, \bm{y}) = \widetilde{\nabla} g(\bm{y}, \bm{x})$) are used in \eqref{eq:jumpexpsys}, then the set of solutions to \eqref{eq:jumpexpsys} is invariant under exchange of $\bm{u}$ and $\bm{v}$. This holds since the linear system itself is invariant to such an exchange. We also have a notion of consistency. Take $\bm{u}=\bm{v}$, then by consistency of the discrete gradient functions:
\begin{align*}
    0 &= \widetilde{\nabla}_{\bm{z}}\psi_q(\bm{u},\bm{u}) - c_q(\bm{u},\bm{u})\widetilde{\nabla}_{\bm{z}}\mathcal{G}_q(\bm{u},\bm{u}) -   \widetilde{\nabla}_{\bm{z}}\bm{\xi}(\bm{u},\bm{u})\bm{f}_h(\bm{u},\bm{u}) \\
    &= \nabla_{\bm{z}}\psi_q(\bm{u}) - c_q(\bm{u}) \nabla_{\bm{z}} \mathcal{G}_q(\bm{u}) - \nabla_{\bm{z}} \bm{\xi}(\bm{u}) \bm{f}_h(\bm{u},\bm{u}) \\
    &= \nabla_{\bm{z}} \bm{u}(\bm{u})\left(\nabla_{\bm{u}}\psi_q(\bm{u}) - c_q(\bm{u}) \nabla_{\bm{u}} \mathcal{G}_q(\bm{u}) - \nabla_{\bm{u}} \bm{\xi}(\bm{u}) \bm{f}_h(\bm{u},\bm{u})\right) \\
    &= \nabla_{\bm{z}}\bm{u}(\bm{u})\left(\nabla_{\bm{u}} \bm{\xi}(\bm{u})(\bm{f}(\bm{u})- \bm{f}_h(\bm{u},\bm{u}))\right),
\end{align*}
since $\nabla_{\bm{z}} \bm{u}(\bm{u})$ is invertible, this hold if and only if $\nabla_{\bm{u}} \bm{\xi}(\bm{u})(\bm{f}(\bm{u})- \bm{f}_h(\bm{u},\bm{u})) = 0$. Because $\nabla_{\bm{u}} \bm{\xi}$ is symmetric, this implies that the flux $\bm{f}_h$ is consistent at $\bm{u}$ within the subspace defined by $\text{col}(\nabla_{\bm{u}} \bm{\xi}(\bm{u}))$, since:
\begin{equation*}
    \nabla_{\bm{u}} \bm{\xi}(\bm{u})^T(\bm{f}(\bm{u})- \bm{f}_h(\bm{u},\bm{u})) = 0,
\end{equation*}
and thus the components of $\bm{f}_h(\bm{u},\bm{u})$ and $\bm{f}(\bm{u})$ along all $\nabla_{\bm{u}} \xi_i(\bm{u})$ agree. However, the generalized Tadmor condition does not ensure consistency outside of $\text{col}(\nabla_{\bm{u}} \bm{\xi}(\bm{u}))$. Instead, consistency outside of $\text{col}(\nabla_{\bm{u}} \bm{\xi}(\bm{u}))$ must be ensured by proper design choices; we will see an example of this in the following section. Once the proper design choices outside of $\text{col}(\nabla_{\bm{u}} \bm{\xi}(\bm{u}))$ are made, satisfaction of the generalized Tadmor condition then ensures the rest.
\section{Application: numerical fluxes for real gases}\label{sec:fluxeszz}
In this section, we apply the discrete framework developed in the previous section to construct an entropy- and kinetic-energy-preserving numerical flux function, given in equation \eqref{eq:numflux}, for our application of interest: the compressible flow of real gases with arbitrary equations of state. We also briefly introduce the necessary concepts from real-gas dynamics. 

\subsection{Real-gas Euler equations}
The compressible flow of a fluid can be accurately modeled by the Euler equations when irreversible processes such as viscous stresses and heat transfer can be neglected. The Euler equations are given as:
\begin{equation}
    \ppto{\begin{bmatrix}
        \rho \\
        \bm{m} \\
        E
    \end{bmatrix}}{t} + \sum_{i=1}^d \ppto{\begin{bmatrix}
        m_i \\
        v_i(\bm{u})  \bm{m} + p(\bm{u})\bm{\delta}_i \\
        v_i(\bm{u}) (E + p(\bm{u}))
    \end{bmatrix}}{x_i} = 0,
    \label{eq:euler}
\end{equation}
where $\rho : \mathbb{R}^d \times [0,t_f] \rightarrow \mathbb{R}_+$ is the mass density, $\bm{m} : \mathbb{R}^d \times [0,t_f] \rightarrow \mathbb{R}^d$ the momentum density and $E : \mathbb{R}^d \times [0,t_f] \rightarrow \mathbb{R}$ the total-energy density. Together, these form the conservative variables of the Euler equations which we gather in a vector $\bm{u}$ of conservative variables. Furthermore, $v_i : \mathcal{U} \rightarrow \mathbb{R}, \bm{u} \mapsto m_i / \rho$ is the fluid velocity\footnote{For the remainder of the article $\bm{v}$ will be used for the velocity, in contrast to its usage as some point in $\mathcal{U}$ in earlier sections.} in direction $i$ and $p : \mathcal{U} \rightarrow \mathbb{R}_+$ is the thermodynamic pressure as a function of the conservative variables. 

To close the system of conservation laws in \eqref{eq:euler} a specification of the pressure $p(\bm{u})$ is required. Such a specification is provided by a thermodynamic equation of state (EoS). A thermodynamic EoS for a single component fluid relates two so-called natural variables $\bm{\tau} \in \mathcal{T} \subseteq \mathbb{R}^2$ (e.g.\ density and temperature) to any other thermodynamic quantity. Hence, an EoS provides us with a pressure in terms of natural variables $p(\bm{\tau})$. It is then necessary to evaluate the (often implicitly defined) map $\bm{\tau} : \mathcal{U}\rightarrow \mathcal{T}$ from conservative variables $\bm{u}$ to natural variables $\bm{\tau}$, which lets us, by some abuse of notation, define $p(\bm{u}) := p(\bm{\tau}(\bm{u}))$. We will interchangeably describe thermodynamic quantities as a function of conservative and thermodynamic variables. Further practical information about our choices of EoS is provided in \autoref{app:thermo}.

Secondary quantities of the Euler equations \eqref{eq:euler} of interest to our applications are the fluid's kinetic-energy density $k : \mathcal{U} \rightarrow \mathbb{R}_+$ and the classical mathematical entropy $s$ of \eqref{eq:euler} given by the negative physical entropy density. Specifically, we have:
\begin{equation}
    k(\bm{u}) := \frac{1}{2}\frac{||\bm{m}||^2}{\rho}, \quad s(\bm{u}) := -\rho \sigma(\bm{u}), \label{eq:kinent}
\end{equation}
where $\sigma : \mathcal{U} \rightarrow \mathbb{R}$ is the specific entropy (i.e., entropy per unit mass) computed using any EoS of choice (see \eqref{eq:entropythermo}). Kinetic energy satisfies the following evolution equation:
\begin{equation}
    \pupt{k(\bm{u})} + \sum_{i=1}^d \ppto{k(\bm{u})v_i(\bm{u})}{x_i} + \sum_{i=1}^d v_i(\bm{u}) \ppto{p(\bm{u})}{x_i} = 0,
    \label{eq:kinencl}
\end{equation}
this is in the form of \eqref{eq:clq} with $q(\bm{u}) = k(\bm{u})$. Entropy satisfies the conservation law:
\begin{equation}
    \pupt{s(\bm{u})} + \sum_{i=1}^d \ppto{s(\bm{u}) v_i(\bm{u})}{x_i} = 0,
    \label{eq:entcl}
\end{equation}
assuming the solution function $\bm{u}$ is $C^1$.

\subsection{Jump expansions for kinetic energy}
We will start by solving the generalized Tadmor condition \eqref{eq:gentadcond} for kinetic energy \eqref{eq:kinent}. It is useful to start here, as the Hessian matrix of kinetic energy $\mathcal{H}_k(\bm{u}) = \ppoppt{k}{\bm{u}}(\bm{u})$ is always singular. In light of the discussion in \autoref{sec:jumpexps}, we thus expect the solution of the method of jump expansions to be in the form of a set defined by a flux condition instead of a single flux. This condition can then be substituted in the classical Tadmor condition for entropy, for which we do expect a unique solution due to convexity. The gradient $\bm{\kappa} : \mathcal{U} \rightarrow \mathbb{R}^n, \bm{u} \mapsto \popt{k}{\bm{u}}(\bm{u})$ and residual function $\psi_k^i$ \eqref{eq:potq} are given as:
\begin{align}
    \psi_k^i(\bm{u}) = p(\bm{u})v_i(\bm{u}), \quad \kappa_\rho(\bm{u}) &= - \frac{1}{2} \frac{||\bm{m}||_2^2}{\rho^2}, \quad \kappa_{\bm{m}}(\bm{u}) = \frac{\bm{m}}{\rho}, \quad\kappa_E(\bm{u}) = 0,
    \label{eq:kgrad}
\end{align}
where the subscripts of $\kappa$ denote with respect to which conservative variable the partial derivative has been taken. As a numerical work term, we will take a straightforward approximation:
\begin{equation}
    [c_k^i\Delta\mathcal{G}_k^i] = \overline{v}_i \Delta p, 
    \label{eq:kinennumwork}
\end{equation}
where we suppress the explicit dependence on an arbitrary pair $(\bm{u}_1,\bm{u}_2) \in \mathcal{U} \times\mathcal{U}$ to ease notation. We will continue doing this in the remainder of this section. As explained in \autoref{sec:discnull}, it is important to verify whether the chosen numerical work term satisfies weak discretely null-consistency \eqref{eq:weaknullcons}. Using the definition \eqref{eq:kgrad}, note that $\Delta \bm{\kappa}_{\bm{m}} = 0$ implies $\Delta \bm{v} = 0$ so $\bm{v}(\bm{u}_1)=\bm{v}(\bm{u}_2) = \widehat{\bm{v}}$. Then, by substituting the definition of $\psi_k^i$ we observe $\Delta \psi_k^i = \Delta (v_i p) = \widehat{v}_i \Delta p$ since $v_i$ is constant. For the same reason, it holds that $\widehat{v}_i = \overline{v}_i$ in \eqref{eq:kinennumwork}; therefore, indeed the implication $\Delta \bm{\kappa} = 0 \implies \Delta \psi_k^i = \overline{v}_i\Delta p$ holds, meaning the numerical work term \eqref{eq:kinennumwork} satisfies weak discrete null-consistency.

The generalized Tadmor condition can now be solved to find a numerical flux condition. We will often use the discrete product rule identity:
\begin{equation}
    \Delta (ab) = \overline{b}\Delta a + \overline{a}\Delta b.
    \label{eq:discprod}
\end{equation}
We will expand in terms of the auxiliary variables $\bm{z}=[\rho, \bm{v}, p]$, as they make all functions in the generalized Tadmor condition polynomial in $\bm{z}$. This makes it particularly simple to solve the generalized Tadmor condition using \eqref{eq:discprod}. Beginning with $\Delta \psi_k^i$, we get:
\begin{equation*}
    \Delta \psi_k^i = \Delta (v_ip) = \overline{p}\Delta v_i + \overline{v}_i \Delta p,
\end{equation*}
where the product rule was used. We continue with $\kappa_{\rho} = -||\bm{v}||^2/2$:
\begin{equation*}
    \Delta \kappa_{\rho} = - \frac{1}{2} \Delta \left(||\bm{v}||^2 \right) = -\overline{\bm{v}}^T\Delta \bm{v}.
\end{equation*}
This is just a component-wise application of the discrete product rule \eqref{eq:discprod}. For $\Delta \kappa_{\bm{m}} = \Delta \bm{v}$, $\Delta \kappa_E = 0$ and the work term $\overline{v}_i\Delta p$ no expansions are necessary. Gathering all the expansions in \eqref{eq:gentadcond} leaves us with:
\begin{equation*}
    \left(
    \begin{bmatrix}
        0 \\ \overline{p} \bm{\delta}_i \\ \overline{v}_i
    \end{bmatrix} - 
    \begin{bmatrix}
        0 \\ \bm{0} \\ \overline{v}_i
    \end{bmatrix} 
    - \begin{bmatrix}
         0 & 0 & 0\\
         -\overline{\bm{v}} & I & 0 \\
         0 & 0 & 0
    \end{bmatrix}
    \begin{bmatrix}
        f_\rho^i \\ f_{\bm{m}}^i \\ f_E^i
    \end{bmatrix}
    \right)^T \begin{bmatrix}
        \Delta \rho \\ \Delta \bm{v} \\ \Delta p
    \end{bmatrix} = 0.
\end{equation*}
It can immediately be seen that the strong discrete null-consistency \eqref{eq:strongdiscnull} is satisfied. A general solution is given by any numerical flux satisfying this relation:
\begin{equation}
    f_{\rho}^i \overline{\bm{v}} + \overline{p} \bm{\delta}_i =  f_{\bm{m}}^i.
    \label{eq:rhomrel}
\end{equation}
The relation \eqref{eq:rhomrel} is the same as the kinetic-energy condition found in \cite{ranochaentropy}. This shows that our new framework can capture existing approaches to derive structure-preserving fluxes. Interestingly, a different discrete form of the kinetic-energy equation was assumed in \cite[equation 10]{ranochaentropy}, as well as a different ansatz for its discretization. These forms are, in fact, equivalent, as they are necessarily compatible, discrete, split-forms of each other.

\subsection{Entropy-conserving and kinetic-energy-consistent flux for real gases}
We will now derive an entropy-conserving numerical flux that incorporates the kinetic-energy consistency condition derived earlier. For an arbitrary EoS the entropy variables $\bm{\eta}$ and residual function $\psi_s^i$ are given by:
\begin{equation}
    \psi_s^i(\bm{u}) = \frac{v_i(\bm{u}) p(\bm{u})}{T(\bm{u})},\quad \eta_\rho(\bm{u}) = \frac{g(\bm{u}) - \frac{1}{2}\frac{||\bm{m}||_2^2}{\rho^2}}{T(\bm{u})}, \quad \eta_{\bm{m}}(\bm{u}) = \frac{\bm{m}}{\rho T(\bm{u})}, \quad \eta_E(\bm{u}) = -\frac{1}{T(\bm{u})},
    \label{eq:sgrad}
\end{equation}
where $T : \mathcal{U} \rightarrow \mathbb{R}_+$ is the thermodynamic temperature and $g : \mathcal{U} \rightarrow \mathbb{R}$ is the specific Gibbs energy \eqref{eq:gibbs}. Since entropy is conserved for $C^1$ solutions, no numerical work term is necessary. Moreover, since the entropy variables are injective, the Tadmor condition is automatically discretely null-consistent. 

Following the approach taken in the ideal-gas case in \cite{chandrashekarkinetic}, we will perform jump expansions with the auxiliary variables $\bm{z} = [\rho, \bm{v}, \beta]$, where $\beta = 1/T$ denotes the inverted temperature. To obtain a flux that is also kinetic-energy-consistent on top of entropy conservative, we will directly substitute the relation \eqref{eq:rhomrel} into the Tadmor condition for entropy. Instead of expanding on a term-by-term basis, as done for kinetic energy, we will see fit to take an integrated approach now due to the greater complexity of the involved terms. In contrast to kinetic energy, this will allow us to combine terms arising from jumps in different quantities during the expansion process, thereby reducing some of the complexity. Written out in full, the Tadmor condition \eqref{eq:tadcond} is:
\begin{equation*}
    \Delta \left(v_i p \beta\right) - f_\rho^i \Delta \left(\left(g - \frac{1}{2} ||\bm{v}||^2\right)\beta\right) - \left(f_{\rho}^i \overline{\bm{v}} + \overline{p}\bm{\delta}_i\right)^T\Delta \left(\beta \bm{v}\right) + f_E^i \Delta \beta = 0.
\end{equation*}
We can use the product rule \eqref{eq:discprod} to write:
\begin{align*}
    \Delta \left(-\frac{1}{2} ||\bm{v}||^2\beta \right) &= -\frac{1}{2}\overline{||\bm{v}||^2} \Delta \beta - \overline{\beta} \Delta \left(\frac{1}{2}||\bm{v}||^2\right), \\
\end{align*}
and:
\begin{align*}
    \overline{\bm{v}}^T\Delta \left(\beta\bm{v}\right) &= ||\overline{\bm{v}}||^2 \Delta \beta + \overline{\beta} \hbox{ }\overline{\bm{v}}^T\Delta \bm{v} \\
    &= ||\overline{\bm{v}}||^2 \Delta \beta+ \overline{\beta} \Delta \left( \frac{1}{2}||\bm{v}||^2\right),
\end{align*}
which allows us to simplify the Tadmor condition to:
\begin{equation*}
    \Delta \left(v_i p\beta\right) - f_\rho^i \Delta \left(g\beta\right) - f_{\rho}^i \left(||\overline{\bm{v}}||^2 -\frac{1}{2}\overline{||\bm{v}||^2} \right) \Delta \beta   - \overline{p} \Delta \left(v_i \beta\right) + f_E^i \Delta \beta = 0.
\end{equation*}
The terms involving pressure in the Tadmor condition can be rewritten as follows:
\begin{equation*}
    \Delta (v_i p \beta) = \overline{pv_i} \Delta \beta + \overline{\beta} \Delta (p v_i),
\end{equation*}
and:
\begin{align*}
    \overline{p} \Delta (v_i \beta) &= \overline{p} \hbox{ } \overline{v}_i\Delta \beta + \overline{p} \hbox{ } \overline{\beta} \Delta v_i \\
    &=\overline{p} \hbox{ } \overline{v}_i\Delta \beta + \overline{\beta} \Delta (p v_i) - \overline{v}_i \overline{\beta} \Delta p \\
    &= 2 \overline{p} \hbox{ } \overline{v}_i\Delta \beta + \overline{\beta} \Delta (p v_i) - \overline{v}_i \Delta (p\beta).
\end{align*}
Combining these manipulations further simplifies the Tadmor condition to:
\begin{gather*}
    \overline{v}_i \Delta (p\beta) - f_\rho^i \Delta \left(g\beta\right) - 
    \left(f_\rho^i\left(||\overline{\bm{v}}||^2 -\frac{1}{2}\overline{||\bm{v}||^2} \right) + (2 \overline{v}_i\overline{p}-\overline{v_ip} ) -f_E^i \right)\Delta \beta = 0. 
\end{gather*}
We are left with two jumps $\Delta (p \beta), \Delta (g\beta)$ that are only defined up to an arbitrary choice of EoS. As both of the terms are thermodynamic in nature, they can also be described using natural EoS variables $\bm{\tau}$. A choice of natural variables that fits our set of auxiliary variables $\bm{z}$ is $\bm{\tau} = [\rho, \beta]$. \textit{A novel application of discrete gradients is now the following}. With this choice of variables, we can apply discrete gradient operators to our arbitrary EoS to still expand the jump $\Delta (p \beta), \Delta (g\beta)$ even though we lack explicit algebraic forms. We then obtain the following expansions:
\begin{align*}
    \Delta \left(p\beta\right) &= \widetilde{\nabla}_\rho (p\beta) \Delta \rho + \widetilde{\nabla}_\beta (p\beta) \Delta \beta, \\
    \Delta \left(g\beta\right) &= \widetilde{\nabla}_\rho (g\beta) \Delta \rho + \widetilde{\nabla}_\beta (g\beta) \Delta \beta.
\end{align*}
Substituting these discrete gradient expansions into the Tadmor condition and rewriting gives the following equation:
\begin{align*}
    &\left(\begin{bmatrix}
        \overline{v}_i\widetilde{\nabla}_\rho (p\beta) \\ \overline{v}_i\widetilde{\nabla}_\beta (p\beta) - (2 \overline{v}_i\overline{p}-\overline{v_ip} )
    \end{bmatrix} -\begin{bmatrix}
        \widetilde{\nabla}_\rho (g\beta) & 0 \\
        \widetilde{\nabla}_\beta(g\beta) + \left(||\overline{\bm{v}}||^2 -\frac{1}{2}\overline{||\bm{v}||^2} \right) & -1 & 
    \end{bmatrix} 
    \begin{bmatrix}
        f_\rho^i \\ f_E^i
    \end{bmatrix}\right)^T \begin{bmatrix}
        \Delta \rho \\ \Delta \beta
    \end{bmatrix} = 0,
\end{align*}
for the jump expansion method. This is a reduced form resulting from the drop in the number of free parameters due to substituting the kinetic-energy condition into the Tadmor condition for entropy. It is sufficient for solutions $f_{\rho}^i, f_E^i$ to exist if:
\begin{equation}
    \det \begin{bmatrix}
        \widetilde{\nabla}_\rho (g\beta) & 0 \\
        \widetilde{\nabla}_\beta(g\beta) + \left(||\overline{\bm{v}}||^2 -\frac{1}{2}\overline{||\bm{v}||^2} \right) & -1 & 
    \end{bmatrix} = - \widetilde{\nabla}_\rho (g\beta) \neq 0,
    \label{eq:determinant}
\end{equation}
for all $\bm{u}_1,\bm{u}_2 \in \mathcal{U}$. We will show later that $\widetilde{\nabla}_\rho (g\beta) > 0$ for some choices of discrete gradient and under physically realistic conditions on the real-gas EoS \cite{holystthermodynamics}. Solving for a flux yields the final form of the mass-density flux component, along with the momentum flux components from the kinetic-energy condition \eqref{eq:rhomrel} and the total-energy-density flux component:
\begin{empheq}[box=\fbox]{align}
    f_\rho^i &= \widetilde{\rho} \hbox{ }\overline{v}_i, \label{eq:numflux} \\
    f_{\bm{m}}^i &= f_{\rho}^i \hbox{ }\overline{\bm{v}} + \overline{p} \bm{\delta}_i, \nonumber \\
    f_E^i &= f_\rho^i \hbox{ } \widetilde{e}  +f_\rho^i\left(||\overline{\bm{v}}||^2 -\frac{1}{2}\overline{||\bm{v}||^2} \right) + (2 \overline{v}_i\overline{p}-\overline{v_ip} ), \nonumber
\end{empheq}
where the density mean satisfies:
\begin{equation}
    \widetilde{\rho} = \frac{\widetilde{\nabla}_\rho (p\beta)}{\widetilde{\nabla}_\rho (g\beta)},
    \label{eq:densitymean}
\end{equation}
and the specific-internal-energy mean satisfies:
\begin{equation}
    \widetilde{e} = \widetilde{\nabla}_\beta(g\beta) - \frac{\widetilde{\nabla}_\beta (p\beta)}{\widetilde{\rho}}.
    \label{eq:internalenergymean}
\end{equation}
Note that the density mean has no singularities if $\widetilde{\nabla}_\rho (g\beta) > 0$, which results from the determinant condition \eqref{eq:determinant}. When assuming $\widetilde{\nabla}_\rho (p\beta) > 0$, and thus that $\widetilde{\rho} > 0$, the internal-energy mean also has no singularities. 

We will show that for some choices of discrete gradient we always have $\widetilde{\nabla}_\rho (p\beta),\widetilde{\nabla}_\rho (g\beta) > 0$ and thus $\widetilde{\rho} > 0$ under appropriate conditions on the EoS. This resulting flux \eqref{eq:numflux} is \emph{consistent with the kinetic-energy equation} \eqref{eq:kinencl} with numerical work term \eqref{eq:kinennumwork} according to \autoref{thm:balance} and \emph{conserves entropy for any EoS} in the sense of \eqref{eq:numcle}. In the literature, similar fluxes for the ideal-gas EoS are often referred to as a kinetic-energy- and entropy-preserving (KEEP). Here, preservation of kinetic energy is intended to mean that only pressure work can produce it. Since our flux \eqref{eq:numflux} is KEEP for any EoS we will refer to our flux and any other flux with this property as a GEoS-KEEP flux, where GEoS refers to general equation of state. To refer specifically to our novel GEoS-KEEP flux as in \eqref{eq:numflux}, we will use the name KEEP-DG, where DG refers to discrete gradient.

\subsection{Flux analysis}
This section presents several statements capturing useful properties of the derived kinetic-energy-consistent and entropy-conserving flux. Let $\bm{f}_h^i = [f_{\rho}^i, (f_{\bm{m}}^i)^T, f_E^i]^T$ denote the flux \eqref{eq:numflux}. The proofs of these statements are given in \autoref{app:propprfs}.
\begin{prop}[Consistency]
    The flux $\bm{f}_h^i$ is consistent with the flux of the Euler equations \eqref{eq:euler} for any discrete gradient and EoS and all $i=1,...,d$.
\end{prop}
\begin{prop}[Symmetry]
    Let $\widetilde{\nabla}(p\beta), \widetilde{\nabla}(g\beta)$ be symmetric discrete gradients. The flux $\bm{f}_h^i$  is symmetric for all $i=1,...,d$.
\end{prop}
\begin{prop}[Density mean positivity]\label{prop:pos}
    Let the symmetrized Itoh-Abe discrete gradient \eqref{eq:siadg} be used for the terms $\widetilde{\nabla}(p\beta), \widetilde{\nabla}(g\beta)$ in the flux $\bm{f}_h^i$, the density mean \eqref{eq:densitymean} satisfies $\widetilde{\rho}(\bm{u}_1,\bm{u}_2) > 0$ for some $\bm{u}_1,\bm{u}_2\in\mathcal{U}$ if for all $\hat{\bm{\tau}} \in \widetilde{\mathcal{T}}(\bm{u}_1,\bm{u}_2) = \{(\rho, T)\in \mathcal{T}: \rho \in [\rho_1,\rho_2] \subset \mathbb{R}_+, T\in \{T_1,T_2\}\subset \mathbb{R}_+ \}$ the Helmholtz energy $\mathcal{A}$ is convex in $1/\rho$ and $\mathcal{A} \in C^2(\widetilde{\mathcal{T}}(\bm{u}_1,\bm{u}_2))$.
\end{prop}
The assumptions for this proposition are essentially always satisfied if the two states $\bm{u}_1,\bm{u}_2\in\mathcal{U}$ are in the same thermodynamic phase and the EoS is thermodynamically stable \cite{holystthermodynamics} (i.e.\ it represents a stable thermodynamic equilibrium). 
\begin{prop}[Ideal gas behavior]
    Using the ideal-gas EoS \cite{holystthermodynamics} and the symmetrized Itoh-Abe discrete gradient, the flux $\bm{f}_h^i$ agrees with the Ranocha flux \cite{ranochapreventing, ranochathesis, ranochaentropy}.
\end{prop}

\section{Application: numerical experiments}\label{sec:numexps}
\subsection{Setup and implementation}
A series of numerical experiments is performed on transcritical flows to test the proposed KEEP-DG numerical flux \eqref{eq:numflux} and to validate the theoretical results obtained from the proposed generalized Tadmor condition framework. We will first perform numerical experiments for a simple one-dimensional transcritical density wave (DW) \cite{aielloentropy}, for which there is an exact solution. This will be followed by a more complex compressible turbulent flow given by a three-dimensional transcritical Taylor-Green vortex (TGV) \cite{wanghighorder, debonissolutions, sciacovelliassessment, boldinicubens}. The case of the DW will serve to validate some basic numerical properties of the flux \eqref{eq:numflux} and to provide an initial assessment of its conservation properties under real-gas conditions. The TGV will be used to assess the properties of our KEEP-DG flux \eqref{eq:numflux} under more realistic real-gas and compressible turbulent flow configurations. 

All test cases consider transcritical real-gas conditions. A flow is termed transcritical when its pressure and temperature are above its critical point, and its isobaric heat capacity is close to a local maximum \cite{guardonenonideal}. Such states are characterized by strong thermodynamic gradients as shown in \autoref{fig:eos_comparison}. In \autoref{fig:eos_comparison}, the ideal-gas case (IG), which possesses no transcritical regime, is shown as a reference. A number of EoS will be considered, namely the Van der Waals (VdW) EoS \cite{holystthermodynamics}, Peng-Robinson (PR) EoS \cite{pengnew}, and the  Kunz-Wagner (KW) EoS \cite{kunzgerg}. Although our flux can be applied to any single-component flow, we will focus on carbon dioxide ($\mathrm{CO_2}$) due to its importance in practical applications \cite{draskicsteady}. The applied $\mathrm{CO_2}$-specific material properties in the EoS are given in \autoref{tab:eosparams} and, except for the number of molecular degrees of freedom $\varsigma$, are all taken from the REFPROP 10.0 database \cite{lemmonnist}. For reference, we also report the VdW and PR EoS in \autoref{app:eos}. 

All experiments are carried out with our multi-GPU, open-source, JAX-based \cite{jaxgithub} code for the simulation of compressible flows of real gases called HelmEOS2 \cite{kleinHelmEOS2}. HelmEOS2 uses a single user-provided expression for an arbitrary specific Helmholtz energy to derive a full library of thermodynamic quantities using JAX's automatic differentiation capabilities. High-performance multi-GPU capabilities are obtained through the use of JAX Array objects in combination with JAX's \emph{sharded memory} functionality \cite{jaxgithub}.

\begin{table}[]
\centering
\begin{tabular}{llll}
\hline
Symbol   & Value & Unit & Description           \\ \hline
$M$      & $44.0098 \cdot 10^{-3}$  & $\kg \cdot \mol^{-1}$   & Molar mass \\
$\rho_c$ & 467.5997                 & $\kg \cdot \meter^{-3}$ & Mass density at critical point \\
$T_c$    & 304.1282                 & $\kelv$                 & Temperature at critical point \\
$p_c$    & $7.3773 \cdot 10^6$      & $\pascal$               & Pressure at critical point \\
$\varsigma$    & 5                        & --                      & Number of molecular degrees of freedom \\
         &                          &                         & (vibrational neglected for VdW \& PR) \\
$\theta$ & 0.22394                  & --                      & Acentric factor \\ \hline
\end{tabular}
\caption{Material properties of $\mathrm{CO_2}$ used throughout the numerical experiments. All properties but $\varsigma$ are obtained from the REFPROP 10.0 database \cite{lemmonnist}.}
\label{tab:eosparams}
\end{table}

\begin{figure}[!htbp]
   \centering
   \includegraphics[width=\linewidth]{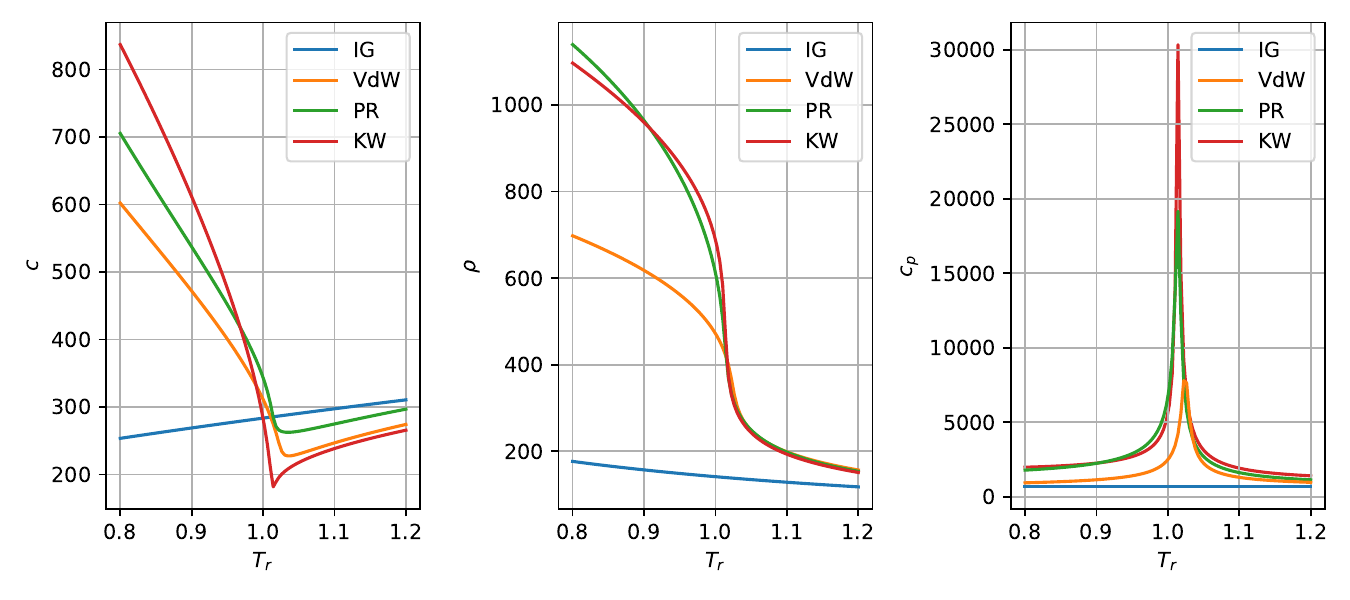}
   \caption{Comparison of the speed of sound $c$ $[\meter \cdot \second^{-1}]$, reduced density $\rho_r = \rho / \rho_c$ and isobaric heat capacity $c_p$ $[\joule \cdot \kg^{-1} \cdot \kelv^{-1}]$ as a function of reduced temperature $T_r = T / T_c$, using different EoS (ideal gas (IG), Van der Waals (VdW), Peng-Robinson (PR), Kunz-Wagner (KW)) over an isobar $p=1.1 \cdot p_c$ $[\pascal]$.}
   \label{fig:eos_comparison}
\end{figure}

%Describe what gas and EoS, what code on what machine etc, high-level description of code (autodiff, jax etc.) paragraph + paper repo 

\subsection{Transcritical density wave}\label{sec:DW}
Our first numerical experiment is the one-dimensional transcritical density wave (DW) \cite{aielloentropy}. As an analytical solution is available, this is a useful case to validate the order of accuracy of the scheme \eqref{eq:disc} when used in combination with our novel KEEP-DG flux \eqref{eq:numflux}. The structure of the DW solution also allows us to study a variety of different conservation properties in addition to entropy conservation. Namely, the total kinetic energy is also conserved for DW. As we will show later, this is a consequence of no pressure work being done in the flow. Moreover, we can study the so-called pressure equilibrium preservation (PEP) property \cite{ranochapreventing, bernadeskinetic, demichelenovel} in the presence of real-gas dynamics as the velocity and pressure fields are constant (a sufficient condition for pressure equilibria). The DW experiment will be carried out on a one-dimensional periodic domain $\Omega = [0, 1]$; $x$ is the spatial coordinate with unit m. The initial condition is the following:
\begin{align}
    \rho_0(x) &= \rho_c \left[0.839 + 0.1 \sin(2 \pi x) \right], \label{eq:dwinit}\\
    v_0(x) &= 10, \nonumber \\
    p_0(x) &= 1.758 p_c. \nonumber 
\end{align}
It is straightforward to show that the advection of the initial density profile $\rho_0(x)$ with velocity $10 \hbox{ } \meter \cdot \second^{-1}$ at constant pressure and velocity is a solution to \eqref{eq:euler} for any EoS. We use the VdW EoS for simplicity. The coefficients in \eqref{eq:dwinit} are selected so that the flow is highly transcritical under the VdW EoS to emphasize real-gas effects. All experiments are carried out until a final time $t_f = 0.5 \hbox{ } \second$, corresponding to five periods of the solution. As a discrete gradient, we will use the second-order accurate symmetrized Itoh-Abe operator \eqref{eq:siadg} \cite{itohhamiltonian} throughout this experiment. We will use the classical fourth-order Runge-Kutta (RK4) \cite{hairergeometric} integrator for time integration. This experiment was performed using HelmEOS2 \cite{kleinHelmEOS2} on a simple laptop workstation with an Intel i7-12700H CPU.

We start by examining the scheme's convergence order using an empirical convergence analysis. We will carry out a convergence analysis on meshes of $N_x \in \{33, 65, 129, 257 \}$ cells. We use odd values of $N_x$ to avoid a symmetric placement of grid points about local extrema of the initial condition. For the DW initial condition on an even mesh, an extremum can lie exactly midway between two neighboring points at equal distance, which may produce locally isochoric or isothermal discrete states (i.e., vanishing $\Delta T$, $\Delta \rho$). In this situation, several discrete gradient evaluations in the numerical flux \eqref{eq:numflux} must be replaced by analytically computed gradients (see \autoref{app:dgs}) via the switch \eqref{eq:siadgswitch}. Because activation of this switch changes the discrete operator in a mesh-dependent way and can bias the measured convergence rate, we disable it for the convergence tests. Let $\bm{x}_h \in \mathbb{R}^{N_x}$ be a vector of cell centers and $\rho(\bm{x}_h,t_f) \in \mathbb{R}^{N_x}$ be a vector of function values obtained from evaluating the exact solution $\rho(x,t)$ on all cell centers $\bm{x}_h$ at time $t_f$. Furthermore, let $\bm{\rho}_h(t_f) \in \mathbb{R}^{N_x}$ be a vector containing the numerical solution in all cell centers $\bm{x}_h$ at time $t_f$. In a similar fashion, define $m(\bm{x}_h, t_f), \bm{m}_h(t_f) \in \mathbb{R}^{N_x}$ and $E(\bm{x}_h,t_f), \bm{E}_h(t_f) \in \mathbb{R}^{N_x}$ for the exact and numerically approximated momentum and total-energy densities, respectively. We will measure the errors in mass density with: 
% (this is another option to list, you can change it back to how you had it before )
\begin{equation*}
    \varepsilon_{\rho} := \frac{||\bm{\rho}_h(t_f) -  \rho(\bm{x}_h,t_f)||_{\infty}}{||\rho(\bm{x}_h,t_f)||_{\infty}},
\end{equation*}
in momentum density with: 
% measuring relative error in mass density;
\begin{equation*}
    \varepsilon_{m} := \frac{||\bm{m}_h(t_f) -  m(\bm{x}_h,t_f)||_{\infty}}{||m(\bm{x}_h,t_f)||_{\infty}},
\end{equation*}
and in total-energy density with: 
% measuring relative error in momentum density, and:
\begin{equation*}
    \varepsilon_{E} := \frac{||\bm{E}_h(t_f) -  E(\bm{x}_h,t_f)||_{\infty}}{||E(\bm{x}_h,t_f)||_{\infty}},
\end{equation*}
respectively. 
% measuring relative error in total energy density. 
Here, $||\bm{w}||_{\infty} := \max_{i}|w_i|$ denotes the vector $\infty$-norm of some vector $\bm{w}\in\mathbb{R}^{N_x}$. We are primarily interested in the contribution of the spatial discretization to the error at time $t_f$; therefore, we will integrate over a large number of timesteps $N_t \in \mathbb{N}$ to mitigate time integration contributions to the error. For $N_x = 33$ we will set $N_t = 10^6$, resulting in a $\mathrm{CFL} \approx 0.005$ for the VdW EoS. We will double $N_t$ at each subsequent mesh in the convergence analysis, thus keeping the $\mathrm{CFL}$ number roughly constant. The results of the convergence analysis are provided in \autoref{fig:convanal}. The expected order of convergence of $\mathcal{O}(\Delta x^2)$ for central three-point schemes is observed when comparing $\varepsilon_{\rho}, \varepsilon_{m}, \varepsilon_{E}$ with the reference value $C\Delta x^2$ for some $C > 0$.

\begin{figure}
\begin{minipage}[c]{0.48\linewidth}
\includegraphics[width=\linewidth]{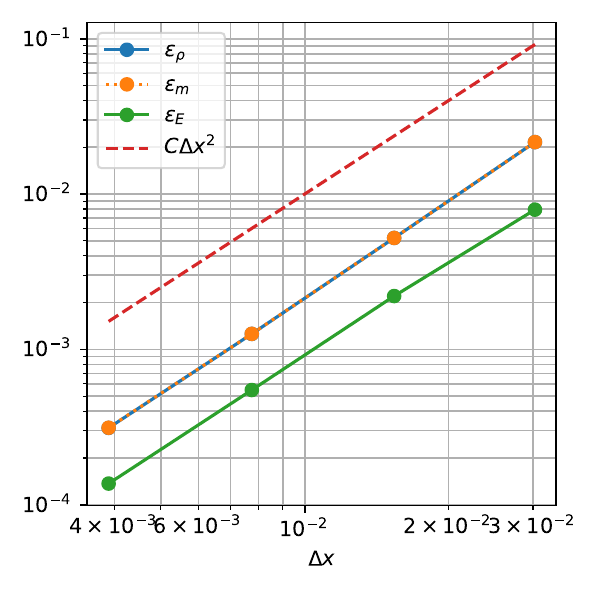}
\caption{Results of the convergence analysis, the relative errors in mass density $\varepsilon_{\rho}$, momentum density $\varepsilon_m$, and total-energy density $\varepsilon_E$ are shown and compared against a reference. All errors are seen to be of order $\mathcal{O}(\Delta x^2)$.}
\label{fig:convanal}
\end{minipage}
\hfill
\begin{minipage}[c]{0.48\linewidth}
\includegraphics[width=\linewidth]{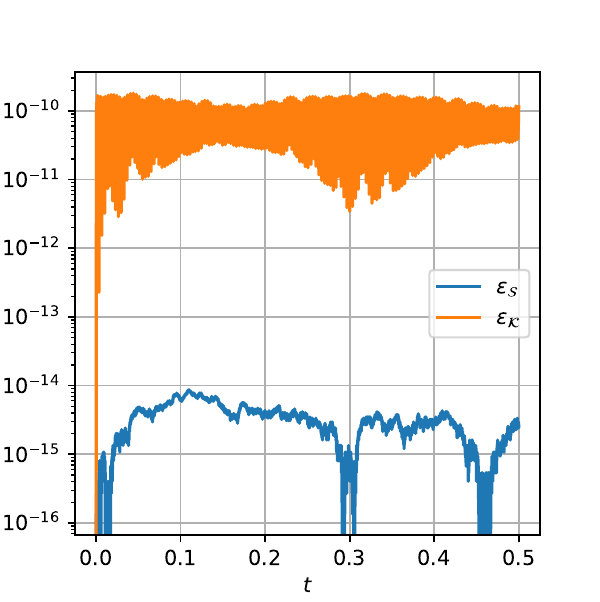}
\caption{Results of the conservation analysis, the relative conservation error of total entropy $\varepsilon_{\mathcal{S}}$, and total kinetic energy $\varepsilon_{\mathcal{K}}$ are shown. Total entropy is conserved to machine precision, while total kinetic energy is conserved to within a very small error.}
\label{fig:consanal}
\end{minipage}%
\end{figure}

\begin{figure}
    \centering
    \includegraphics[width=\linewidth]{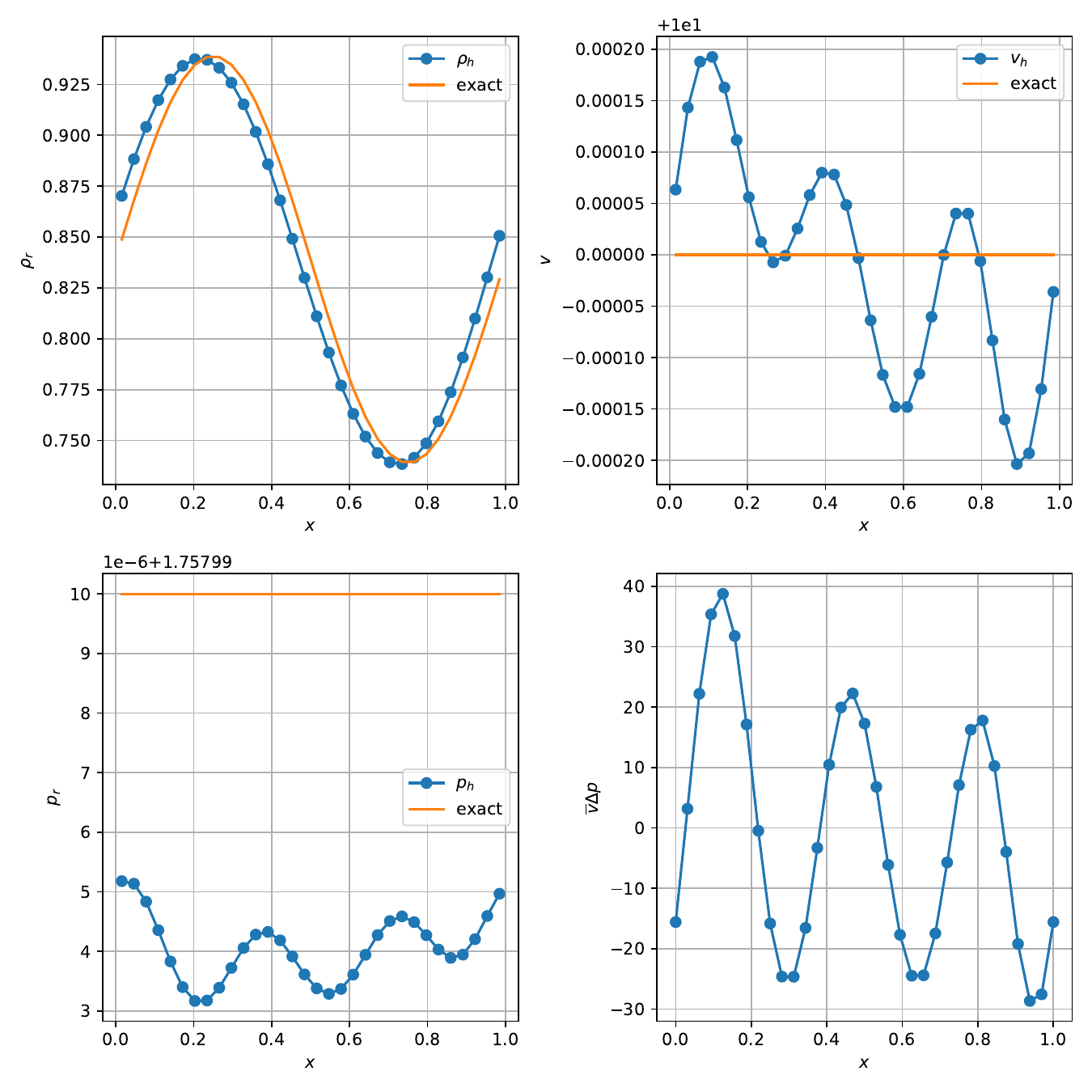}
    \caption{The numerical solution of DW at $t=t_f$ in terms of the primitive variables $\bm{\rho}_h, \bm{v}_h,\bm{p}_h$, alongside the numerical work term \eqref{eq:kinennumwork} at cell interfaces (bottom-right). Deviations from exact spatially constant solutions can be observed in the numerical solution (top-right and bottom-left), leading to spurious pressure work (bottom-right).}
    \label{fig:dwavesol}
\end{figure}

We continue with an initial analysis of the scheme's conservation properties. To emphasize possible errors, we will perform this experiment on a particularly coarse mesh of $N_x = 32$, turning back on the switching algorithm \eqref{eq:siadgswitch} for the discrete gradients in our numerical flux \eqref{eq:numflux}. Again, we want to mitigate time integration contributions to the conservation errors, so we will integrate over $N_t = 10^6$ time steps. This results in roughly the same $\mathrm{CFL}$ number, $\mathrm{CFL} \approx 0.005$, as in the previous experiment. The exact solution of DW satisfies both conservation of total entropy $\mathcal{S} : [0, t_f] \rightarrow \mathbb{R}$ and kinetic energy $\mathcal{K} : [0, t_f] \rightarrow \mathbb{R}$, respectively defined as:
\begin{align*}
    \mathcal{S}(t) &:= \int_{\Omega} s(\bm{u}(x,t)) d\Omega, \\
    \mathcal{K}(t) &:= \int_{\Omega} k(\bm{u}(x,t)) d\Omega,
\end{align*}
with $s,k$ as in \eqref{eq:kinent}. The conservation of $\mathcal{K}$ follows from integrating \eqref{eq:kinencl} over the domain $\Omega$ for spatially constant velocity and pressure fields for which the non-conservative term in \eqref{eq:kinencl} disappears. On a uniform mesh, with cell sizes $\Delta x = 1/32 \hbox{ } \meter$, the numerical total entropy $\mathcal{S}_h : [0, t_f] \rightarrow \mathbb{R}$ and numerical total kinetic energy $\mathcal{K}_h : [0, t_f] \rightarrow \mathbb{R}$ are defined as:
\begin{align*}
    \mathcal{S}_h(t) &:= \sum_{i=1}^{N_x} s(\bm{u}_i(t)) \Delta x, \\
    \mathcal{K}_h(t) &:= \sum_{i=1}^{N_x} k(\bm{u}_i(t)) \Delta x,
\end{align*}
where $[\rho_i(t), m_i(t), E_i(t)]^T = \bm{u}_i(t) \in \mathbb{R}^3$ is the numerical solution vector in grid cell $i$. 

Since entropy is a conserved secondary structure, summing \eqref{eq:numcle} over the periodic mesh shows that $\mathcal{S}_h$ is discretely conserved. Let $\bm{v}_h(t)\in\mathbb{R}^{N_x}$ be a vector containing the velocity values at time $t \in [0, t_f]$ in all cell centers, computed from the numerical solution vector $\bm{u}_i(t)$ in each respective cell. Similarly, let $\bm{p}_h(t) \in \mathbb{R}^{N_x}$ be a vector of pressure values in cell centers. Then conservation also holds for $\mathcal{K}_h$, after summing \eqref{eq:discq} (for $q_{\bm{I}}=k_{\bm{I}}$) over the periodic mesh and substituting constant $\bm{v}_h, \bm{p}_h$ vectors in the numerical work terms \eqref{eq:kinennumwork}. We will analyze the conservation properties using the following error measures:
\begin{equation}
    \varepsilon_{\mathcal{S}}(t) := \frac{|\mathcal{S}_h(t) - \mathcal{S}_h(0)|}{|\mathcal{S}_h(0)|}, 
    \label{eq:enterr}
\end{equation}
measuring the deviation from total entropy conservation, and:
\begin{equation*}
    \varepsilon_{\mathcal{K}}(t) := \frac{|\mathcal{K}_h(t) - \mathcal{K}_h(0)|}{|\mathcal{K}_h(0)|},
\end{equation*}
measuring the deviation from total kinetic-energy conservation. Both errors are shown as a function of time in \autoref{fig:consanal}. Machine precision errors are reached throughout the simulation for $\varepsilon_{\mathcal{S}}$ as expected in the presence of negligible time integration errors. However, $\varepsilon_{\mathcal{K}}$ is not at machine precision. This result has been noted in earlier studies of kinetic-energy conservation \cite{gassnersplit, ranochaentropy}. It is a consequence of the fact that the numerical work terms in the discrete balance law \eqref{eq:discq} (for $q_{\bm{I}}=k_{\bm{I}}$) are only conditionally zero (for constant $\bm{v}_h, \bm{p}_h$ in the case of $\mathcal{K}_h$) after summing over the periodic mesh and numerical errors can still influence the numerical production of $\mathcal{K}_h$. This cannot occur in the case of entropy, as the numerical work terms for entropy are identical to zero for entropy-conserving fluxes. To illustrate this, \autoref{fig:dwavesol} shows the numerical solution at time $t_f$ in terms of $\bm{\rho}_h, \bm{v}_h,\bm{p}_h$ and the work term in \autoref{fig:dwavesol}. Indeed, due to numerical error accumulation, both $\bm{p}_h$ and $\bm{v}_h$ are not spatially constant, and thus spurious pressure work is being done. We note that this issue is entirely distinct from discrete null-consistency. Where discrete null-consistency ensures the existence of a discrete kinetic-energy balance law under locally constant $\bm{\kappa}$, the current issue is concerned with the numerical accuracy of a readily existing numerical work term \eqref{eq:kinennumwork} under perturbations from locally constant $\bm{v}_h, \bm{p}_h$ due to numerical errors. \autoref{fig:dwavesol} also shows that the scheme does not satisfy the PEP property \cite{bernadeskinetic, ranochapreventing, demichelenovel} exactly. Furthermore, we note that the error is dominated by dispersion effects when time-integration errors are negligible, indicating the non-dissipative nature of the spatial discretization scheme.

\subsection{Inviscid real-gas Taylor-Green vortex}\label{sec:tgv}
We now proceed with the compressible turbulent supercritical TGV test case. We will carry out a number of inviscid and viscous experiments, starting with the inviscid case in this subsection. These experiments are meant, respectively, to test the conservation properties and stability of the scheme against under-resolution, and its physical fidelity for resolved transcritical turbulence. In particular, in inviscid form, there is no lower bound on the size of the turbulent length-scales that can form in this flow due to the lack of a turbulent dissipation mechanism. In fact, spatial scales become progressively smaller over time. Therefore, for sufficiently long times, the flow will become under-resolved at any mesh resolution. It has been noted in \cite{gassneraccuracy} that even in the presence of some under-resolution, practically useful results can be obtained from turbulent computations under certain conditions. It is therefore interesting to assess the stability of our KEEP-DG flux \eqref{eq:numflux} in the presence of strong under-resolution in turbulent flows of real gases. 

We will perform a comparison with other fluxes that are suitable to be used with arbitrary EoS and are noted for their entropy-conserving capabilities, specifically the recently proposed \textit{singular} GEoS-KEEP flux \cite{aielloentropy}, which we will refer to as KEEP-S with S standing for singular, and the KEEP-Q flux of \cite{tamakicomprehensive, kuyakinetic}. For KEEP-S, singular is meant in the sense that divisions by zero can occur. Specifically, this singularity takes place when neighboring cells have the same temperature. It should also be noted that KEEP-Q is not strictly entropy-conserving in the sense of Tadmor \cite{tadmorentropy, tadmornumerical} and therefore not a GEoS-KEEP flux, but has still shown good entropy conservation properties \cite{tamakicomprehensive}. The singularity of the KEEP-S flux in \cite{aielloentropy} is dealt with through a switching mechanism, where the flux is replaced by another non-singular flux when nearly singular. The KEEP-S flux of \cite{aielloentropy} has been shown to be very stable under non-isothermal flow configurations and to satisfy the GEoS-KEEP property; however, when combined with non-entropy-conserving schemes to deal with the singularity, the resulting scheme loses strict entropy conservation. We will show that the switching mechanism and the resulting loss of strict entropy conservation of the singular KEEP-S flux \cite{aielloentropy} can still lead to instability in flow configurations that require a lot of switch activations. We will show that our novel flux prevents these issues. 

A configuration of the Taylor-Green vortex \cite{wanghighorder, debonissolutions} that is particularly challenging for the singular KEEP-S flux of \cite{aielloentropy} is the isothermal TGV, as isothermal conditions are precisely when the flux in \cite{aielloentropy} becomes singular. We define the isothermal TGV on a triple-periodic box domain $\Omega = [-\pi L, \pi L]^3$. The initial condition on $\Omega$ is given as:
\begin{align}
    v_{0,1}(\bm{x}) &= V_0 \sin \left(\frac{x_1}{L}\right) \cos \left(\frac{x_2}{L}\right) \cos \left(\frac{x_3}{L}\right), \label{eq:tgv_initial} \\
    v_{0,2}(\bm{x}) &= -V_0 \cos \left(\frac{x_1}{L}\right) \sin \left(\frac{x_2}{L}\right) \cos\left(\frac{x_3}{L}\right), \nonumber\\
    v_{0,3}(\bm{x}) &= 0, \nonumber \\
    p_0(\bm{x}) &= p_0 + \frac{\rho_0 V_0^2}{16} \left( \cos \left(\frac{2x_1}{L}\right) + \cos \left(\frac{2x_2}{L}\right)\right)\left( \cos \left(\frac{2x_3}{L}\right) + 2\right), \nonumber\\
    T_0(\bm{x}) &= T_0. \nonumber
\end{align}
We choose the length scale $L \in \mathbb{R}_+$ and characteristic velocity $V_0 \in \mathbb{R}_+$, pressure $p_0  \in \mathbb{R}_+$, density $\rho_0  \in \mathbb{R}_+$ and temperature $T_0  \in \mathbb{R}_+$ to be in line with the inviscid numerical conservation experiment carried out in \cite{tamakicomprehensive}. We will thus set\footnote{Although precise physical units do not matter much in this scale-invariant inviscid setting.} $L = 1 \hbox{ } \meter$ and choose consistent thermodynamic quantities so that $\mathrm{Ma}= 0.4$, where:
\begin{equation*}
    \mathrm{Ma} := \frac{V_0}{c(\rho_0, T_0)},
\end{equation*}
is the Mach number and $c : \mathcal{T} \rightarrow \mathbb{R}_+$ is the speed-of-sound \eqref{eq:speedofsound}. We will use the Peng-Robinson (PR) EoS \cite{pengnew} for this experiment. The characteristic quantities can then be computed to satisfy our $\mathrm{Ma}= 0.4$-constraint by using the values in \autoref{tab:iTGVparams}. Below the critical point $(p<p_c, T<T_c)$ the PR EoS can produce unphysical values. The particular choice of characteristic density and temperature $\rho_0,T_0$ has been made to prevent potentially strong numerical oscillations, produced by the high degree of under-resolution, from making the solution enter this regime. 
%sufficiently far away from the thermodynamic two-phase region. below the critical point $(p<p_c, T<T_c)$ where many EoS produce unphysical values To prevent the numerical solution in a computational cell from falling . Many EoS provide unphysical values in this region without additional phase-change models and would make the Euler equations \eqref{eq:euler} inherently unstable. As phase-change models fall outside the scope of this article, we choose our reference state $(\rho_0,T_0)$ sufficiently far from this region to mitigate this potential mode of failure. This reference state is not transcritical, but is supercritical.

\begin{table}[h!]
\centering
\begin{tabular}{lll}
\hline
Symbol     & Value            & Unit                                  \\ \hline
$\rho_0$   & $0.3 \rho_c$   & $\kg \cdot \meter^{-3}$               \\
$T_0$     & $1.4 T_c$       & $\kelv$                               \\
$p_0$     & $p(\rho_0,T_0)$        & $\pascal$                      \\
$V_0$     & $0.4 c(\rho_0, T_0)$   & $\meter \cdot \second^{-1}$    \\
$L$        & 1              & $\meter$                              \\ \hline
\end{tabular}
\caption{Characteristic values used for the inviscid TGV initial condition \eqref{eq:tgv_initial}. Here, $\rho_c,T_c$ are as in \autoref{tab:eosparams}. The $\mathrm{Ma}=0.4$-constraint is satisfied by taking $V_0 = 0.4 c(\rho_0, T_0)$.}
\label{tab:iTGVparams}
\end{table}

We will carry out the experiment on a uniform Cartesian mesh with $N_x^3$ cells, where $N_x = 64$ is the number of cells in each coordinate direction. The turbulent convective time scale is defined as $t_c = L/V_0$. We will integrate up to a final time $t_f = 50 t_c$, so that the flow becomes highly under-resolved. As a time integrator, we will use the classical fourth-order Runge-Kutta (RK4) method \cite{hairergeometric}. To mitigate contributions to the error due to time integration, we will use $N_t = 50 \cdot 6400$ timesteps. This leads to $\mathrm{CFL} \approx 0.005$ throughout the simulation. As a discrete gradient, our KEEP-DG flux \eqref{eq:numflux} will use the symmetrized Itoh-Abe discrete gradient operator \eqref{eq:siadg} \cite{itohhamiltonian}. We will measure entropy conservation through the error $\varepsilon_{\mathcal{S}}$ as defined in \eqref{eq:enterr}, where in the three-dimensional case, the definition of the discrete total entropy is:
\begin{equation*}
    \mathcal{S}_h(t) := \sum_{\bm{I}\in\mathcal{I}} s(\bm{u}_{\bm{I}}) \Delta x_1\Delta x_2\Delta x_3.
\end{equation*}
Here $\mathcal{I} \subset \mathbb{N}^3$ is an index set containing all grid cell multi-indices $\bm{I} = (i_1,i_2,i_3)$. The definition of the discrete total kinetic energy is:
\begin{equation*}
    \mathcal{K}_h(t) := \sum_{\bm{I}\in\mathcal{I}} k(\bm{u}_{\bm{I}}) \Delta x_1\Delta x_2\Delta x_3.
\end{equation*}
For the simulation with the singular KEEP-S flux \cite{aielloentropy}, we will also track the percentage of cells that have at least one cell face with a replacement flux active over time. As a replacement flux, we will use the KEEP-Q flux \cite{tamakicomprehensive}. We will use the switching strategy as proposed in \cite{aielloentropy}. We will switch when the absolute difference between temperatures between two cells is less than $10^{-3} \hbox{ } \kelv$. This experiment is performed using HelmEOS2 \cite{kleinHelmEOS2} on a quarter GPU node of Snellius (the Dutch national supercomputer). A quarter GPU node of Snellius consists of a single Nvidia H100 GPU with $94$ Gigabytes of RAM.

\begin{figure}
  \makebox[\textwidth][c]{\includegraphics[width=\textwidth]{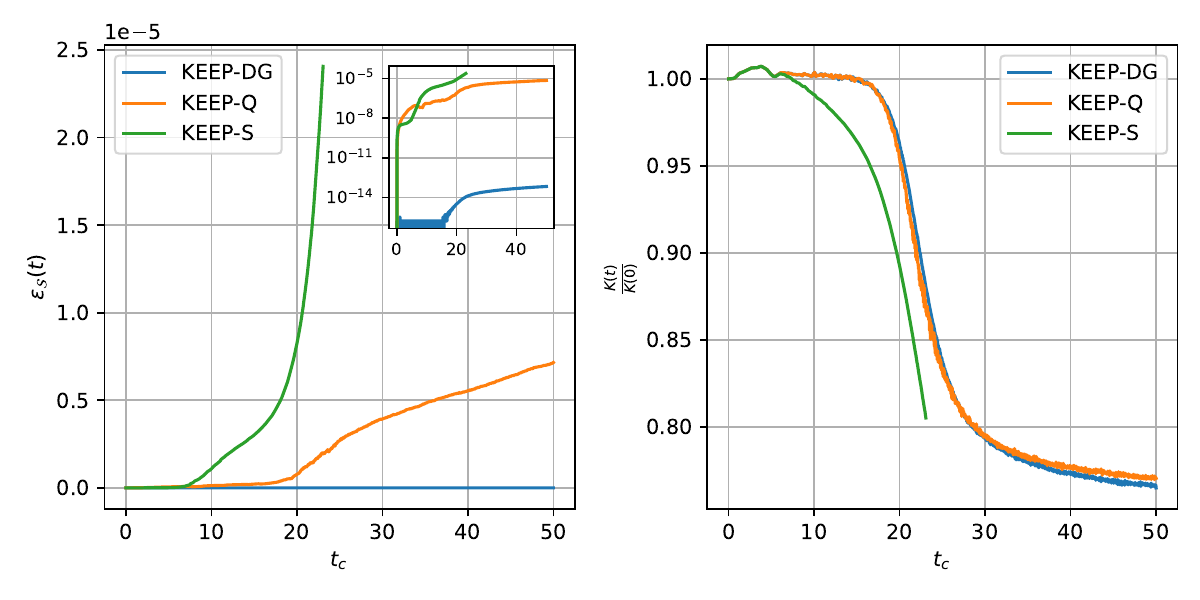}}%
  \caption{Results of the conservation analysis. Left is a comparison of the relative total entropy conservation error $\varepsilon_{\mathcal{S}}(t)$ as in \eqref{eq:enterr} between different fluxes. The error is shown both in linear and semi-log axes. Right is a comparison of the relative total kinetic-energy evolution. In both figures, time is measured in number of convective time scales $t_c = L /V_0$.}
  \label{fig:blowuptgv}
\end{figure}

\begin{figure}
    \centering
    \includegraphics[width=0.6\linewidth]{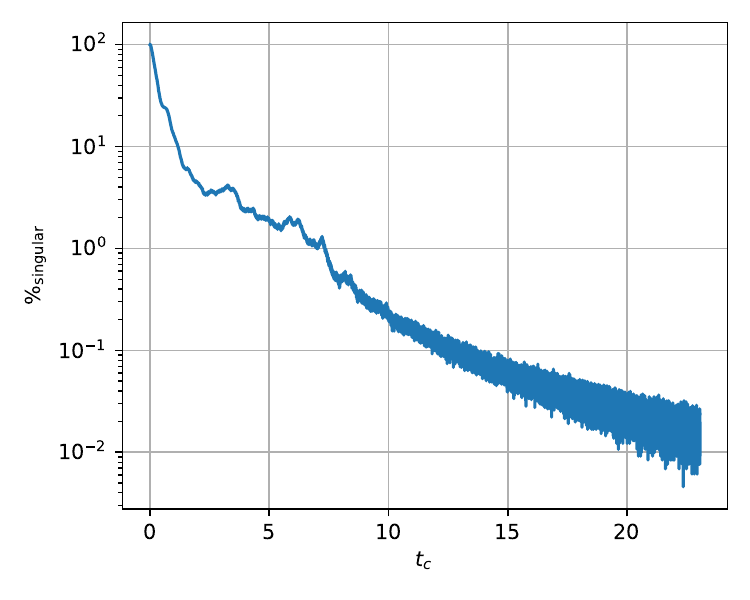}
    \caption{The percentage of computational cells where the switch of the singular KEEP-S flux of \cite{aielloentropy} is active as a function of time measured in convective time scales $t_c$.}
    \label{fig:perctgv}
\end{figure}

The results of the experiment are shown in \autoref{fig:blowuptgv} and \autoref{fig:perctgv}. Shortly after $t = 20t_c$, a crash occurs in the singular KEEP-S flux as can be seen from the blow-up of the entropy in \autoref{fig:blowuptgv}. In \autoref{fig:perctgv}, it can be seen that a significant number of cells require switching up to $t\approx 5t_c$, and the percentage never reaches zero. This drives the simulation to become unstable. The KEEP-Q flux \cite{kuyahigh, tamakicomprehensive} remains stable, but the entropy error $\varepsilon_{\mathcal{S}}$ starts to increase significantly in the second half of the simulation. Our KEEP-DG flux conserves entropy up to machine-precision accuracy throughout the first half of the simulation. It experiences a small increase in error around the same time as the KEEP-Q at $t = 20 t_c$. However, the total size of the entropy error remains close to machine precision and orders of magnitude lower than the other two fluxes. As both our KEEP-DG flux and the KEEP-Q flux experience an increase in entropy error at the same time, it is likely that time-integration errors start playing a role at the level of under-resolution attained at $t = 20 t_c$.

\subsection{Viscous transcritical Taylor-Green vortex}\label{sec:vtgv}
To demonstrate the performance of our KEEP-DG flux for resolved turbulence, we will also perform a viscous transcritical TGV experiment \cite{wanghighorder, debonissolutions} using the full compressible Navier-Stokes equations:
\begin{align}
    \ppto{\begin{bmatrix}
        \rho \\
        \bm{m} \\
        E
    \end{bmatrix}}{t} + \sum_{i=1}^d &\ppto{\begin{bmatrix}
        m_i \\
        v_i(\bm{u})  \bm{m} + p(\bm{u})\bm{\delta}_i \\
        v_i(\bm{u}) (E + p(\bm{u}))
    \end{bmatrix}}{x_i} 
     = \sum_{i=1}^d \ppto{
    \begin{bmatrix}
        0 \\
        \bm{\tau}^i(\bm{u}) \\
        \bm{\tau}^i(\bm{u}) \cdot \bm{v}(\bm{u}) - q_i(\bm{u})
    \end{bmatrix}}{x_i} \label{eq:NS} . %&\qquad \qquad
\end{align}
Here, $\bm{\tau} $ and $ \bm{q}$ are differential operators that give the stress tensor and heat flux vector, respectively. The system is closed by introducing so-called constitutive laws for the stress tensor and the heat flux, defining the particular form of the differential operators, in addition to an EoS. To demonstrate that our KEEP-DG flux provides accurate results and is useful in many practical situations, we will use state-of-the-art industrial-standard constitutive laws and EoS for transcritical $\mathrm{CO_2}$. As EoS we will use the high-accuracy Kunz-Wagner (KW) EoS \cite{kunzgerg}. We assume $\mathrm{CO_2}$ is Newtonian so that the stress tensor can be modeled as:
\begin{equation*}
    \bm{\tau}(\bm{u}) = 2\mu(\bm{u}) \left[\frac{1}{2}(\nabla \bm{v} + \nabla \bm{v}^T) - \frac{1}{3}(\nabla \cdot \bm{v})I \right],
\end{equation*}
and that the heat flux follows Fourier's law:
\begin{equation}
    \bm{q}(\bm{u}) = -\kappa(\bm{u}) \nabla T. \label{eq:fourier}
\end{equation}
The transport coefficients $\mu : \mathcal{U} \rightarrow \mathbb{R}$ and $\kappa : \mathcal{U} \rightarrow \mathbb{R}$ are the dynamic viscosity and heat conductivity, respectively. The dynamic viscosity and heat conductivity are computed using state-of-the-art models for $\mathrm{CO_2}$, given by \cite{Laeseckereference} and \cite{huberreference}, respectively. In our definition of the stress tensor constitutive law, we have made Stokes' hypothesis and set the bulk viscosity \cite{larssonturbulence, nematidirect} to zero. We summarize all models used to close the compressible Navier-Stokes equations \eqref{eq:NS} in \autoref{tab:co2models}. Their discretization is outlined in \autoref{app:viscdisc}. 

A canonical test case for high-order methods \cite{wanghighorder} is the viscous TGV as in \eqref{eq:tgv_initial} with $\mathrm{Re}= 1600$ and $\mathrm{Ma}= 0.1$, where:
\begin{equation*}
    \mathrm{Re} := \frac{\rho_0 V_0 L}{\mu(\rho_0,T_0)},
\end{equation*}
is the Reynolds number. The canonical case applies an ideal-gas EoS and constant transport properties $\mu, \kappa$. Here, we will extend this case to the highly transcritical real-gas regime. As this case is nearly incompressible, we will also run a more compressible test case with $\mathrm{Ma}=0.3$ to demonstrate that our KEEP-DG flux can handle compressibility effects. Our choices of characteristic variables are shown in \autoref{tab:tgparams}. They are specifically chosen so that the state is highly transcritical, as determined by the KW EoS.

\begin{table}
\centering
\begin{tabular}{lll}
\hline
Quantity          & Reference                                           & Critical \\ 
                  &                                                     & enhancement \\ \hline
Equation of state & Kunz, Klimeck,                                      & -                    \\
                  & Wagner \& Jaeschke \cite{kunzgerg}                   &                       \\
Dynamic viscosity & Laesecke \& Muzny \cite{Laeseckereference}          & No                   \\
Bulk viscosity    & -                                                   & -                    \\
Heat conductivity & Huber, Sykioti,                                      & Yes               \\ 
                  & Assael \& Perkins  \cite{huberreference}                  &                       \\
\hline
\end{tabular}
\caption{Overview of physical models used for the simulation of the Taylor-Green vortex with transcritical $\mathrm{CO_2}$. Many models allow for an enhancement of accuracy around the critical point. If such an enhancement was available in the reference, we indicated whether we used it in our simulation. As done in \cite{kunzgerg}, we used \cite{jaeschkeideal} for the ideal-gas part of the equation of state. The bulk viscosity was set to zero according to Stokes' hypothesis \cite{larssonturbulence}. }
\label{tab:co2models}
\end{table}

\begin{table}
\centering
\begin{tabular}{lll}
\hline
Symbol & Value                      & Unit                         \\ \hline
$\rho_0$   & $1.198 \cdot \rho_c$   & $\kg \cdot \meter^{-3}$      \\
$T_0$     & $1.1 \cdot T_c $        & $\kelv$                      \\
$p_0$     & $p(\rho_0,T_0)$         & $\pascal$                    \\
$V_0$     & $\mathrm{Ma} \cdot c(\rho_0, T_0)$                      & $\meter \cdot \second^{-1}$  \\
$L$        & $1600 \cdot \mu(\rho_0,T_0) / (\rho_0 V_0)$   & $\meter$                     \\ \hline
\end{tabular}
\caption{Choice of characteristic quantities for the simulation of the Taylor-Green vortex with transcritical $\mathrm{CO_2}$, $\mathrm{Ma} \in \{0.1,0.3\}$. This initial condition is highly transcritical using the Kunz-Wagner EoS \cite{kunzgerg}. Here, $\rho_c,T_c$ are as in \autoref{tab:eosparams}.}
\label{tab:tgparams}
\end{table}

We will use a uniform Cartesian mesh of $N_x^3$ cells with $N_x = 512$ in each coordinate direction. The turbulent convective time scale is $t_c = L / V_0$. We will integrate in time until $t_f = 20t_c$ and use $N_t = 20 \cdot 2048$ time steps. Using the theory of incompressible turbulence \cite{nieuwstadturbulence} as an approximation, this is sufficient to resolve the smallest turbulent spatial and temporal scales. The choice of time step size leads to $\mathrm{CFL} \approx 0.43$ throughout the simulation. The high-accuracy KW EoS results in significant memory consumption, so we use Wray's third-order low-storage Runge-Kutta method \cite{wrayminimal} for the time integration to minimize RAM usage. For the Euler term of \eqref{eq:NS} we will use our KEEP-DG flux \eqref{eq:numflux}, using the symmetrized Itoh-Abe discrete gradient operator \eqref{eq:siadg} \cite{itohhamiltonian}. For the viscous stresses and heat fluxes, we use standard second-order accurate central schemes for spatially varying transport coefficients \cite{wesselingprinciples}. The quite large mesh in combination with the expensive KW EoS motivates us to use the multi-GPU capabilities of HelmEOS2 \cite{kleinHelmEOS2}. This experiment was therefore performed on a full GPU node of Snellius, consisting of four Nvidia H100 GPUs. To both validate our multi-GPU implementation and assess grid convergence, we have run an ideal-gas validation case on the same mesh, using \cite{debonissolutions} as a reference. The results, as discussed in more detail in \autoref{app:igvalid}, indicate that the implementation is indeed correct and, although not exactly grid-converged, quite close to the reference solution.

\begin{figure}[h!]
  \makebox[\textwidth][c]{\includegraphics[width=\textwidth]{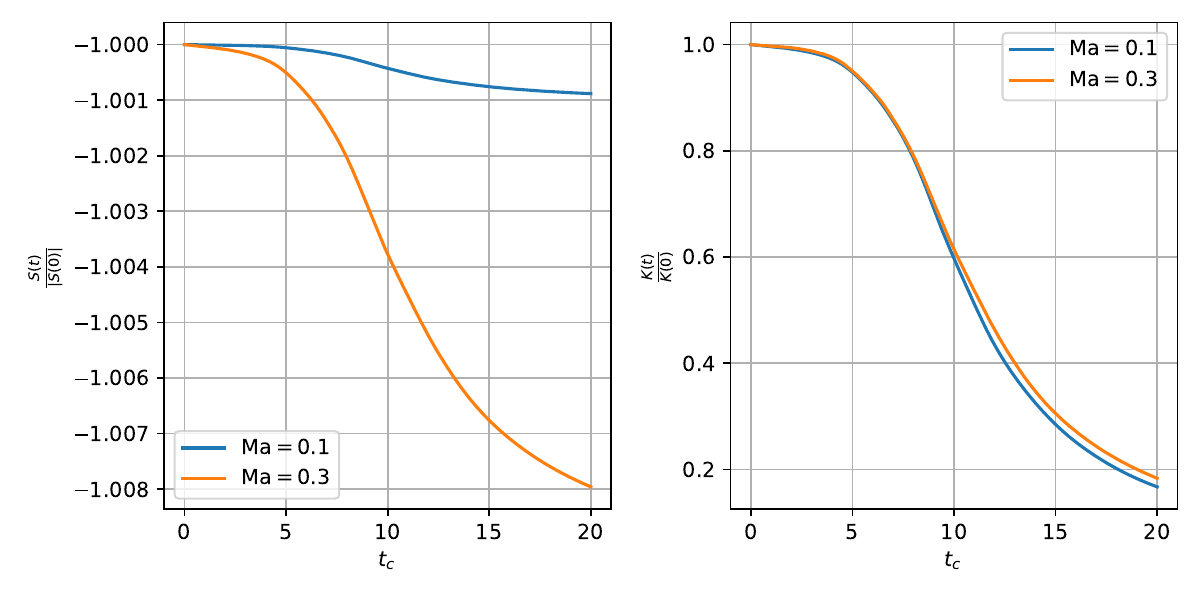}}%
  \caption{Evolution of the total entropy (left) and kinetic energy (right) during the viscous transcritical TGV experiment for different Mach numbers. Time is measured in terms of convective time scales $t_c$.}
  \label{fig:stkt}
\end{figure}

\begin{figure}[h!]
  \makebox[\textwidth][c]{\includegraphics[width=\textwidth]{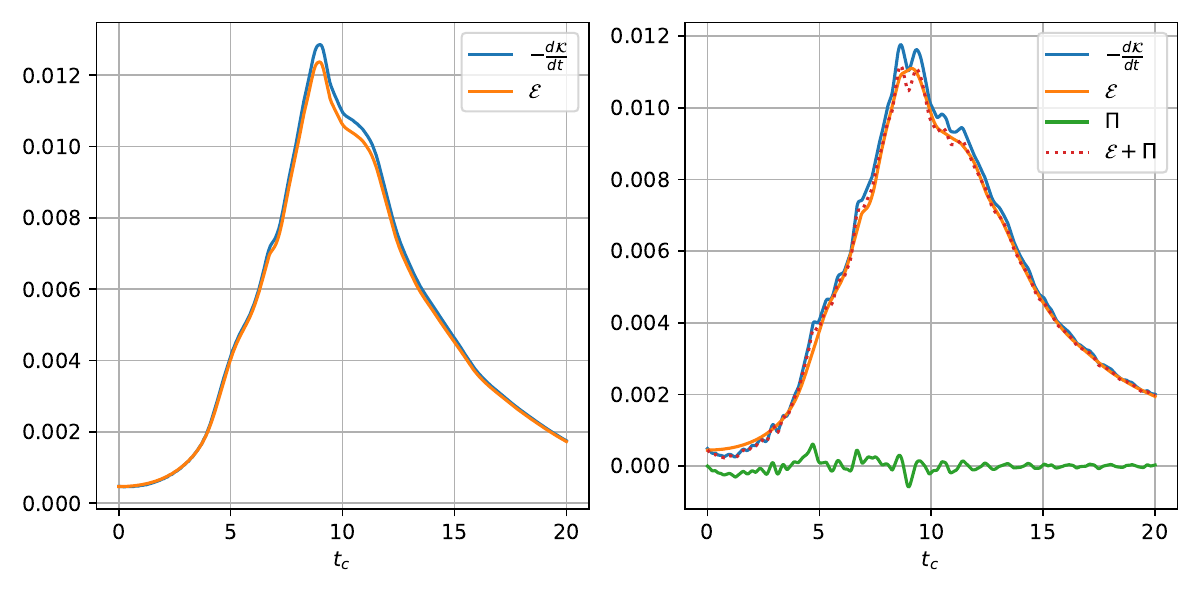}}%
  \caption{Evolution of the kinetic-energy dissipation rate $\dudt{\mathcal{K}}$ and various terms in the total kinetic-energy balance \eqref{eq:kinenbudget} for $\mathrm{Ma}=0.1$ (left) and $\mathrm{Ma}=0.3$ (right). All quantities are normalized by multiplication with the factor $\frac{t_c}{\rho_0 V_0^2 |\Omega|}$ as further elaborated on in \autoref{app:igvalid}. Time is measured in terms of convective time scales $t_c$.}
  \label{fig:dkdtv}
\end{figure}

It can be shown that on periodic domains, for fluids with variable dynamic viscosity, the total kinetic energy satisfies the following balance:
\begin{equation}
    \dudt{\mathcal{K}}= \Pi - \underbrace{(\mathcal{D} + \mathcal{E} + \mathcal{M})}_{\geq 0},
    \label{eq:kinenbudget}
\end{equation}
where:
\begin{align*}
    \Pi &= \int_{\Omega} p(\bm{u}) \nabla\cdot\bm{v} \hbox{ } d\Omega, \\
    \mathcal{D} &= \int_{\Omega} \frac43\mu(\bm{u}) (\nabla \cdot \bm{v})^2 \hbox{ } d\Omega, \\
    \mathcal{E} &= \int_{\Omega} \mu(\bm{u}) ||\nabla \times \bm{v}||^2 \hbox{ } d\Omega, 
\end{align*}
and $\mathcal{M}$ is a (small) variable-viscosity correction term that we do not report (cf. \cite{wuvorticity}) and neglect. We will assess the physical fidelity of our numerical scheme by its ability to reproduce this balance discretely. Note that in the nearly incompressible case $\mathrm{Ma}= 0.1$, both $\Pi$ and $\mathcal{D}$ are small. For $\mathrm{Ma}=0.1$ we thus expect $-\frac{d\mathcal{K}}{dt} \approx \mathcal{E}$. In \autoref{fig:stkt}, we show the total entropy and total kinetic-energy evolution for both Mach number cases as a function of time. We use an $8$-th order temporal central difference scheme on these graphs to obtain $\frac{d\mathcal{K}}{dt}$, which we plot in \autoref{fig:dkdtv}. Due to the lack of grid convergence, there is no exact match, but $\mathcal{E}$ and $-\frac{d\mathcal{K}}{dt}$ are very close for $\mathrm{Ma}=0.1$. For $\mathrm{Ma}=0.3$, the evolution of $\Pi+\mathcal{E}$ closely follows the evolution of $-\frac{d\mathcal{K}}{dt}$. This highlights the dissipation-free nature of our KEEP-DG flux, and that scheme can accurately model the turbulent flow of a transcritical fluid with a modest grid. To give further insight into the local behavior of the solution, we also plot the normalized vorticity magnitude $\omega$ defined as:
\begin{equation}
    \omega := ||\nabla \times \bm{v}||\frac{L}{V_0}, \label{eq:normvort}
\end{equation}
at time $t = 10 t_c$, just after the peak dissipation of kinetic energy takes place. In \autoref{fig:vort03} we plot volume renders of $\omega$ for $\mathrm{Ma}= 0.3$. We note that the $\omega$-field for $\mathrm{Ma}=0.1$ is very similar. Realising that the dissipation is dominated by $\mathcal{E}$, comparing the definitions of $\omega$ and $\mathcal{E}$ shows that large local values of $\omega$ indicate where the most kinetic energy is being dissipated. 

%\begin{figure}
%  \makebox[\textwidth][c]{\includegraphics[width=1.3\textwidth]{figures/taylor_green/vort_01.png}}%
%  \caption{Volume render of the normalized vorticity magnitude $\omega$ \eqref{eq:normvort} at $t=10t_c$ and $\mathrm{Ma}=0.1$. In this volume render, regions of low vorticity ($\omega<2$) are made transparent, and regions of high vorticity ($7<\omega$) are made increasingly transparent as $\omega$ decreases. The region with $2<\omega<7$ is made nearly opaque. (Render produced with ParaView \cite{ParaView})}
%  \label{fig:vort01}
%\end{figure}

\begin{figure}
  \makebox[\textwidth][c]{\includegraphics[width=1.3\textwidth]{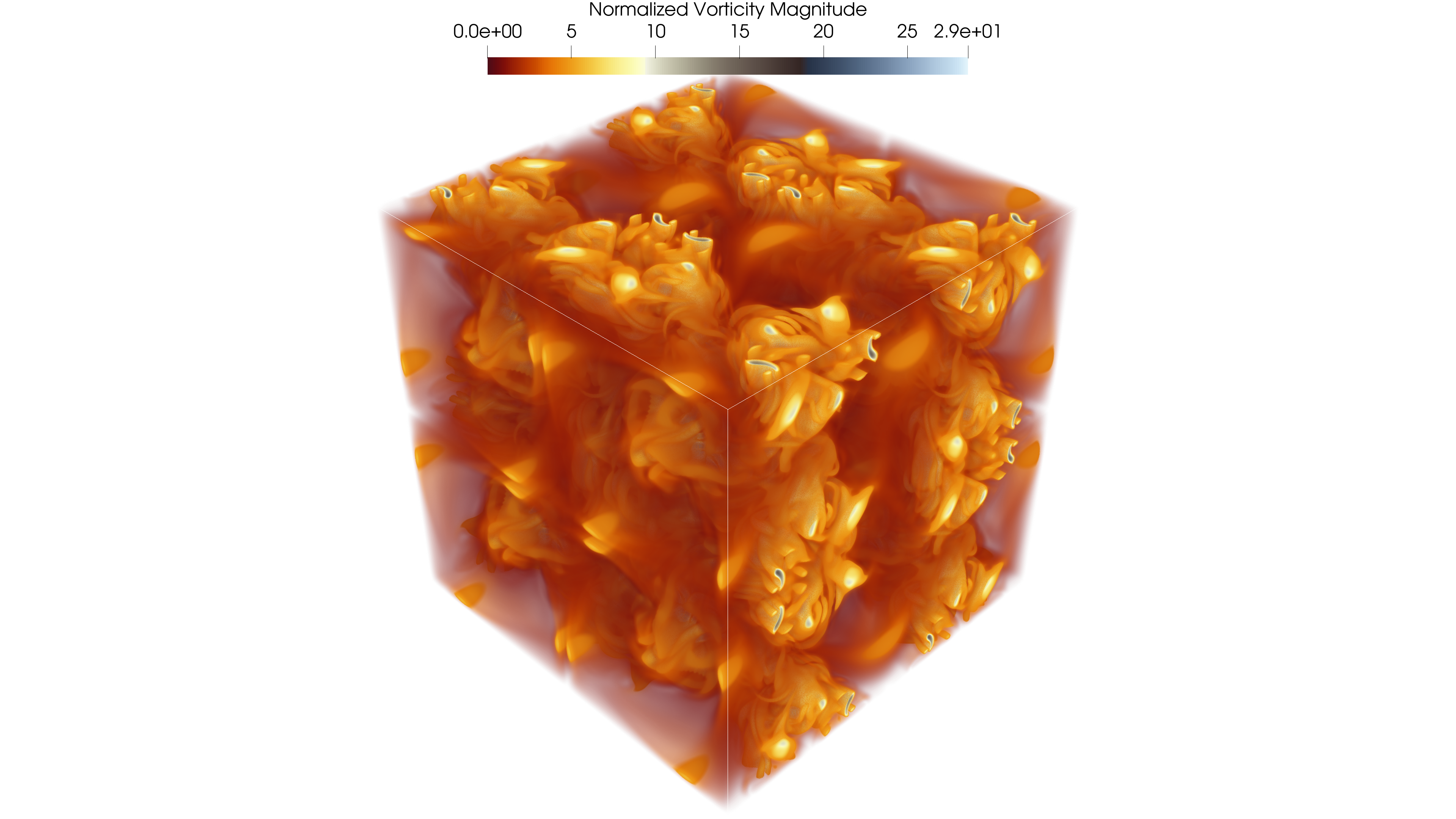}}%
  \caption{Volume render of the normalized vorticity magnitude $\omega$ \eqref{eq:normvort} at $t=10t_c$ and $\mathrm{Ma}=0.3$. In this volume render, regions of low vorticity ($\omega<2$) are made transparent, and regions of high vorticity ($7<\omega$) are made increasingly transparent as $\omega$ decreases. The region with $2<\omega<7$ is made nearly opaque. (Render produced with ParaView \cite{ParaView})}
  \label{fig:vort03}
\end{figure}
\section{Conclusion}
We have proposed a framework that generalizes the tools usually applied to the analysis of convex entropy functions \cite{tadmorentropy, LeVequefinite, dafermoshyperbolic} to non-convex, non-conserved secondary functions of smooth solutions of conservation laws. Within this framework, we showed that general balance laws of secondary quantities are satisfied by solutions of conservation laws whenever the different terms in those balance laws satisfy primal or dual compatibility conditions. These compatibility conditions naturally extend those of entropy pairs and thereby place general secondary quantities into a unified theoretical setting. At the same time, the analysis highlights structural differences with the entropy case: the balance laws of general secondary quantities may involve non-conservative work terms, so that the domain integral of a general secondary quantity is not necessarily conserved for smooth solutions even in the absence of boundary effects. Moreover, the framework reveals additional structural properties that are specific to general non-conservative secondary structures. In particular, the generalized compatibility relations imply that the secondary structure necessarily satisfies a null-consistency property whenever the Hessian of the secondary quantity has a non-trivial kernel.

We have further shown that the generalized compatibility conditions of secondary structures admit discrete analogues, in close analogy with the discrete compatibility conditions of convex entropy functions. In particular, the discrete analogue of the generalized dual compatibility condition naturally yields a numerical flux condition that ensures that the discretization satisfies discrete balance laws of secondary quantities. This condition recovers Tadmor's condition when the secondary structure is an entropy pair, showing that the classical entropy-conservative setting is contained as a special case of the present theory. For non-conservative secondary structures, there is a degree of freedom in the choice of the numerical work term. To ensure a well-posed flux condition, this choice of numerical work term must satisfy additional structural properties. Specifically, a weak discrete analogue of the null-consistency property is required to ensure that the flux condition is well posed. In addition, a stronger form of discrete null-consistency is needed for the generalized Tadmor conditions to be solvable by jump expansions. Furthermore, we demonstrated that discrete gradient operators from geometric integration \cite{mclachlangeometric} provide an effective tool for solving the generalized Tadmor conditions, thereby supplying a constructive mechanism for the design of admissible numerical fluxes.

This framework enables the design of stable non-dissipative, kinetic-energy- and entropy-preserving (KEEP) numerical fluxes for supercritical fluids, which require non-ideal general equations of state (GEoS). Using the proposed framework, we derived a GEoS-KEEP numerical flux for the Euler equations \eqref{eq:euler}, referred to as KEEP-DG. The resulting flux \eqref{eq:numflux} removes a singularity present in the recently proposed GEoS-KEEP flux in \cite{aielloentropy}, demonstrating the practical value of the generalized framework. We also showed, through a number of experiments on a transcritical extension of a canonical compressible turbulent test case, that no instabilities occur for the KEEP-DG flux and that it remains conservative to machine precision. These results indicate that KEEP-DG, derived from the generalized Tadmor framework, provides a dissipation-free and robust approach for accurately resolving compressible turbulent flows of transcritical fluids with state-of-the-art equations of state and transport models on modest grids. 

Although we have focused on kinetic energy and entropy in this work, we note that another popular secondary quantity to consider in the numerical simulation of compressible flow is pressure \cite{ranochapreventing, bernadeskinetic, demichelenovel}. It would be especially interesting to see whether the framework can be used to obtain numerical fluxes that satisfy the so-called pressure-equilibrium-preservation (PEP) property for real gases \cite{bernadeskinetic}. This type of structure-preservation property is fundamentally different from what we consider in this work, as it does not aim to satisfy an extra balance law, but only to exactly reproduce certain solutions of the conservation law \eqref{eq:cl}. However, pressure balance laws are still often considered in the analysis of the PEP property \cite{toutantgeneral, bernadeskinetic}, and thus our framework could be useful here. Initial investigations of some numerical work terms for a discrete pressure balance law show that for real gases, weak discrete null-consistency is not trivial and requires extra conditions on the equation of state (EoS). In turn, verifying these conditions requires detailed knowledge of the respective EoS. We leave this as future work.

\section*{CRediT authorship contribution statement}
\textbf{R.B. Klein}: conceptualization, methodology, software, validation, formal analysis, investigation, writing - original draft \\
\textbf{B. Sanderse}: writing - review \& editing, supervision \\
\textbf{P. Costa}: writing - review \& editing, supervision, resources \\
\textbf{R. Pecnik}: writing - review \& editing, supervision \\
\textbf{R.A.W.M. Henkes}: writing - review \& editing, supervision, project administration, funding acquisition \\

\section*{Software and reproducibility statement}
The code used to generate the results of this paper is available at \href{https://github.com/rbklein/HelmEOS2}{github.com/rbklein/HelmEOS2}. It is released under the MIT license. An archived version is available at \href{https://zenodo.org/records/18985427}{zenodo.org/records/18985427}. The data and scripts to generate the figures in \autoref{sec:numexps}, in addition to extra solution snapshots, are available at \href{https://zenodo.org/records/18982907}{zenodo.org/records/18982907} and can be used with the archived version of the code. The results in \autoref{sec:DW} were obtained on a desktop CPU. The results in \autoref{sec:tgv} were obtained using a single Nvidia H100 GPU on the Dutch National Supercomputer Snellius. The results in \autoref{sec:vtgv} and \autoref{app:igvalid} were obtained using four Nvidia H100 GPUs on the Dutch National Supercomputer Snellius.

\section*{Declaration of Generative AI and AI-assisted technologies in the writing process}
During the preparation of this work, the authors used ChatGPT 5.2, 5.4, and Grammarly to improve readability. After using these tools/services, the authors reviewed and edited the content as needed and take full responsibility for the content of the publication.

\section*{Declaration of competing interests}
The authors declare that they have no known competing financial interests or personal relationships that could have appeared to influence the work reported in this paper.

\section*{Data availability}
Data will be made available on request. %result scripts and data available on zenodo

\section*{Acknowledgements}
R.B. Klein gratefully acknowledges the funding for this project obtained from Delft University of Technology. The three-dimensional GPU runs performed on the supercomputer Snellius (based at SURF, The Netherlands) were supported by the NWO Grant no.\ \texttt{2024/ENW/01704792}.

\appendix 
\section{Entropy conservation using the discrete dual compatibility condition}\label{app:ddual}
Let $\bm{v}_{\bm{I}} := \bm{v}(\bm{u}_{\bm{I}})$ be some grid function. We define the following compact notation:
\begin{align*}
    \Delta \bm{v}_{\bm{I}+\bm{1}_j/2} &:= \bm{v}_{\bm{I}+\bm{1}_j} - \bm{v}_{\bm{I}}, \quad \overline{\bm{v}}_{\bm{I}+\bm{1}_j/2} := \frac{\bm{v}_{\bm{I}+\bm{1}_j} + \bm{v}_{\bm{I}}}{2},
\end{align*}
and two identities between them:
\begin{gather}
    \bm{v}_{\bm{I}} = \overline{\bm{v}}_{\bm{I}\pm\bm{1}_j/2} \mp \frac{1}{2}\Delta \bm{v}_{\bm{I}\pm\bm{1}_j/2}, \label{eq:usefulid_app} \\
    \frac{1}{2}\Delta \bm{v}_{\bm{I}+\bm{1}_j/2}+\frac{1}{2}\Delta \bm{v}_{\bm{I}-\bm{1}_j/2} = \overline{\bm{v}}_{\bm{I}+\bm{1}_j/2} - \overline{\bm{v}}_{\bm{I}-\bm{1}_j/2}.
    \label{eq:commdifavg_app}
\end{gather}
Discrete entropy conservation for a scheme with a numerical flux satisfying the discrete dual compatibility condition is then shown as follows:
{\allowdisplaybreaks
\begin{align*}
    0 &= \bm{\eta}_{\bm{I}}^T\left( \dudt{\bm{u}_{\bm{I}}} + \sum_{j=1}^d \frac{1}{\Delta x_j}\left(\bm{f}_{\bm{I}+\bm{1}_j/2}^j - \bm{f}_{\bm{I}-\bm{1}_j/2}^j \right)\right) \\
    &= \bm{\eta}_{\bm{I}}^T\dudt{\bm{u}_{\bm{I}}} + \sum_{j=1}^d \frac{1}{\Delta x_j} \bm{\eta}_{\bm{I}}^T\left(\bm{f}_{\bm{I}+\bm{1}_j/2}^j - \bm{f}_{\bm{I}-\bm{1}_j/2}^j \right) \\
    &\overset{\eqref{eq:usefulid_app}}{=} \dudt{s_{\bm{I}}} + \sum_{j=1}^d \frac{1}{\Delta x_j} \left(\overline{\bm{\eta}}_{\bm{I}+\bm{1}_j/2}^T\bm{f}_{\bm{I}+\bm{1}_j/2}^j - \overline{\bm{\eta}}_{\bm{I}-\bm{1}_j/2}^T\bm{f}_{\bm{I}-\bm{1}_j/2}^j\right) \\
    &\qquad \qquad \qquad - \sum_{j=1}^d \frac{1}{\Delta x_j} \frac{1}{2}\left(\left(\bm{f}_{\bm{I}+\bm{1}_j/2}^j\right)^T\Delta \bm{\eta}_{\bm{I}+\bm{1}_j/2} + \left(\bm{f}_{\bm{I}-\bm{1}_j/2}^j\right)^T\Delta \bm{\eta}_{\bm{I}-\bm{1}_j/2}  \right) \\
    &\overset{\eqref{eq:numeflux}}{=} \dudt{s_{\bm{I}}} + \sum_{j=1}^d \frac{1}{\Delta x_j} \left(\mathcal{F}_{s,\bm{I}+\bm{1}_j/2}^j - \mathcal{F}_{s,\bm{I}-\bm{1}_j/2}^j \right) + \sum_{j=1}^d \frac{1}{\Delta x_j}\left(\overline{\psi}_{s, \bm{I}+\bm{1}_j/2}^j - \overline{\psi}_{s,\bm{I}-\bm{1}_j/2}^j\right) \\
    &\qquad \qquad \qquad - \sum_{j=1}^d \frac{1}{\Delta x_j} \frac{1}{2}\left(\left(\bm{f}_{\bm{I}+\bm{1}_j/2}^j\right)^T\Delta \bm{\eta}_{\bm{I}+\bm{1}_j/2} + \left(\bm{f}_{\bm{I}-\bm{1}_j/2}^j\right)^T\Delta \bm{\eta}_{\bm{I}-\bm{1}_j/2}  \right) \\
    &\overset{\eqref{eq:commdifavg_app}}{=} \dudt{s_{\bm{I}}} + \sum_{j=1}^d \frac{1}{\Delta x_j} \left(\mathcal{F}_{s,\bm{I}+\bm{1}_j/2}^j - \mathcal{F}_{s,\bm{I}-\bm{1}_j/2}^j \right) + \sum_{j=1}^d \frac{1}{\Delta x_j} \frac{1}{2}\left(\Delta{\psi}_{s, \bm{I}+\bm{1}_j/2}^j + \Delta{\psi}_{s,\bm{I}-\bm{1}_j/2}^j\right) \\
    &\qquad \qquad \qquad - \sum_{j=1}^d \frac{1}{\Delta x_j} \frac{1}{2}\left(\left(\bm{f}_{\bm{I}+\bm{1}_j/2}^j\right)^T\Delta \bm{\eta}_{\bm{I}+\bm{1}_j/2} + \left(\bm{f}_{\bm{I}-\bm{1}_j/2}^j\right)^T\Delta \bm{\eta}_{\bm{I}-\bm{1}_j/2}  \right) \\
    &\overset{\eqref{eq:prototadmor}}{=} \dudt{s_{\bm{I}}} + \sum_{j=1}^d \frac{1}{\Delta x_j} \left(\mathcal{F}_{s,\bm{I}+\bm{1}_j/2}^j - \mathcal{F}_{s,\bm{I}-\bm{1}_j/2}^j \right),
\end{align*}}
where in the third equality we used identities \eqref{eq:usefulid_app}, in the fourth equality definition \eqref{eq:numeflux}, in the fifth equality identity \eqref{eq:commdifavg_app} and in the final equality the discrete dual compatibility condition \eqref{eq:prototadmor}. Note that in this way the discrete derivation exactly mimics the derivation in the differential setting. For the discrete dual compatibility condition, which is due to Tadmor and beares his name, we provide a formal definition.
\section{Discrete gradient operators}\label{app:dgs}
We list several well-known discrete gradient operators. Let $\mathcal{P}$ be an open subset of $\mathbb{R}^2$ and $g : \mathcal{P} \rightarrow \mathbb{R}$ a differentiable function. We consider $\mathcal{P} \subseteq \mathbb{R}^2$ to simplify the notation of the Itoh-Abe discrete gradient operator.
\begin{itemize}
    \item Mean value discrete gradient operator (MVDG) \cite{hartenupstream, mclachlangeometric}:
    \begin{equation*}
        \widetilde{\nabla}g(\bm{p}_1,\bm{p}_2) := \int_{0}^1 \nabla g(\bm{p}_1+\lambda(\bm{p}_2-\bm{p}_1)) d\lambda.
    \end{equation*}
    \item Gonzalez discrete gradient operator (GDG) \cite{gonzaleztime}. Assume $\bm{p}_1 \neq \bm{p}_2$: 
    \begin{equation*}
        \widetilde{\nabla}g(\bm{p}_2,\bm{p}_1) := \nabla g\left(\frac{\bm{p}_1+\bm{p}_2}{2}\right) + \frac{\Delta g(\bm{p}_1, \bm{p}_2) - \nabla g\left(\frac{\bm{p}_1+\bm{p}_2}{2}\right)^T(\bm{p}_2-\bm{p}_1)}{||\bm{p}_2-\bm{p}_1||^2}(\bm{p}_2-\bm{p}_1).
    \end{equation*}
    When $\bm{p}_1 = \bm{p}_2 = \bm{p}$:
    \begin{equation*}
        \widetilde{\nabla}g(\bm{p}_2,\bm{p}_1) = \nabla g(\bm{p}).
    \end{equation*}
    \item Symmetrized Itoh-Abe discrete gradient operator (SIADG) \cite{itohhamiltonian}. Let $\bm{p}_1=[x_1,y_1]^T, \bm{p}_2 =[x_2,y_2]^T$ and assume $x_1 \neq x_2$ and $y_1 \neq y_2$: 
    \begin{equation}
        \widetilde{\nabla}g(\bm{p}_1,\bm{p}_2) := \sum_{k=1}^2 \frac12 \text{diag}(\bm{p}_2-\bm{p}_1)^{-1}\begin{bmatrix}
            g(x_2,y_k) - g(x_1,y_k)\\
            g(x_k,y_2) - g(x_k,y_1)
        \end{bmatrix}.
        \label{eq:siadg}
    \end{equation}
    In case $x_1 = x_2 = x$, but $y_1 \neq y_2$, SIADG is defined as:
    \begin{equation}
        \left(\widetilde{\nabla}g(\bm{p}_1,\bm{p}_2)\right)_1 = \frac12 \popt{g}{x}(x,y_1) + \frac12 \popt{g}{x}(x,y_2),
        \label{eq:siadgexact}
    \end{equation}
    and a similar approach is taken when $y_1=y_2$, but $x_1 \neq x_2$, for the second SIADG component. In practical computations, as e.g.\ $x_2 \rightarrow x_1$, finite-precision effects start dominating the approximation error of SIADG far before $x_1 = x_2$. Let $\epsilon \in \mathbb{R}_+$ denote machine precision. It is a classical result \cite{vuiknumerical}, that for central differences this starts to occur at $|x_2-x_1| = \mathcal{O}(\sqrt{\epsilon})$. As a practical SIADG algorithm, we therefore suggest using \eqref{eq:siadg} when e.g.\ $|x_2 - x_1| > \text{tol}_x$ and otherwise switch to the exact gradient evaluation \eqref{eq:siadgexact}. Here, $\text{tol}_x$ is taken as:
    \begin{equation}
        \text{tol}_x = 10 \epsilon + \sqrt{\epsilon}\max \{|x_1|,|x_2| \},
        \label{eq:siadgswitch}
    \end{equation}
    and a similar expression is used for the $y$-component.
\end{itemize}

\section{Helmholtz-energy equation of state}\label{app:thermo}
In our applications we are mainly interested in EoS given by a specific\footnote{per unit mass} Helmholtz energy $\mathcal{A} : \mathcal{T} \rightarrow \mathbb{R}$. Here, $\mathcal{T}\subset \mathbb{R}^2$ is the thermodynamic state space of natural variables, for the Helmholtz energy these are mass density $\rho$ and temperature $T$. In many applications of fluid dynamics the values $\rho, T$ are assumed as pointwise evaluations of their associated fields $\rho, T : \mathbb{R}^d \times [0,t_f] \rightarrow \mathbb{R}_+$, which we denote with the same symbol. We gather the natural variables in a vector $\bm{\tau} \in \mathcal{T}$ to ease notation. The Helmholtz energy is defined as the Legendre transformation with respect to specific entropy of the first law of thermodynamics \cite{holystthermodynamics}:
\begin{align}
    \mathcal{A}(\bm{\tau}) &= e(\bm{\tau}) - T \sigma(\bm{\tau}), \label{eq:helm} \\
    d\mathcal{A} &= \frac{p(\bm{\tau})}{\rho^2} d\rho-\sigma(\bm{\tau}) dT, \label{eq:diffhelm}
\end{align}
where $e : \mathcal{T} \rightarrow \mathbb{R}$ is the specific internal energy, $\sigma : \mathcal{T}\rightarrow \mathbb{R}$ is the specific entropy, and $p : \mathcal{T} \rightarrow \mathbb{R}_+$ is the thermodynamic pressure now expressed as a function of natural variables. Derived thermodynamic quantities can be computed by comparison of the differential of $\mathcal{A}$ against the identity provided in \eqref{eq:diffhelm}:
\begin{align}
    \sigma(\bm{\tau}) &= -\popt{\mathcal{A}}{T}(\bm{\tau}), \label{eq:entropythermo}\\
    p(\bm{\tau}) &= \rho^2 \popt{\mathcal{A}}{\rho}(\bm{\tau}), \label{eq:pressurethermo}\\
    e(\bm{\tau}) &= \mathcal{A}(\bm{\tau}) - T \popt{\mathcal{A}}{T}(\bm{\tau}). \label{eq:internalenergy}
\end{align}
Another derived quantity we will make use of is the specific Gibbs energy $g : \mathcal{T} \rightarrow \mathbb{R}$, whose definition follows from thermodynamic considerations \cite{holystthermodynamics}:
\begin{equation}
    g(\bm{\tau}) = \mathcal{A}(\bm{\tau}) + \rho \popt{\mathcal{A}}{\rho}(\bm{\tau}). \label{eq:gibbs}
\end{equation}
We can then continue to define the required map $\mathcal{U} \rightarrow \mathcal{T}$ by considering the following definition of the total-energy density $E = \rho e + k$ where:
\begin{equation}
    k(\bm{u}) = \frac{1}{2}\frac{||\bm{m}||^2}{\rho},
    \label{eq:kineticenergy}
\end{equation}
is the fluid's kinetic energy. With this definition of the total energy we can express the internal energy explicitly in terms of conservative variables $e(\bm{u}) = (E - k(\bm{u}))/\rho$. Realizing that the mass density in the conservative variables is identical to the mass density in the thermodynamic state space variables, we are left to solve:
\begin{equation*}
    E - k(\bm{u}) = \rho \mathcal{A}(\bm{\tau}) - \rho T \popt{\mathcal{A}}{T}(\bm{\tau}),
\end{equation*}
for $T$ at fixed $\rho$ to finish the map $\mathcal{U}\rightarrow \mathcal{T}$. The previous equation follows from equating $e(\bm{u})=e(\bm{\tau})$. If a solution $T$ exists the physical consideration that $\popt{e}{T}(\bm{\tau}) > 0$ for all $\bm{\tau} \in \mathcal{T}$ and thus that $e(\bm{\tau})$ is injective at fixed $\rho$ lets us conclude $T$ is unique. From here, we will simply assume that a solution indeed exists, as is often done. Finally, we will also use the squared speed of sound $c^2 : \mathcal{T} \rightarrow \mathbb{R}_+$, which in terms of Helmholtz energies satisfies:
\begin{equation}
    c^2(\bm{\tau}) = 2 \rho \popt{\mathcal{A}}{\rho}(\bm{\tau}) + \rho^2 \ppoppt{\mathcal{A}}{\rho}(\bm{\tau}) - \left[\rho \frac{\partial \mathcal{A}}{\partial \rho \partial T}(\bm{\tau}) \right]^2 /  \left[ \ppoppt{\mathcal{A}}{T}(\bm{\tau}) \right],
    \label{eq:speedofsound}
\end{equation}
and follows from computing the derivative of pressure $p$ with respect to $\rho$ while keeping the entropy $\sigma$ constant.

\section{Proposition proofs}\label{app:propprfs}
\subsection{Consistency}
\begin{proof}
    By consistency of the arithmetic mean, the flux \eqref{eq:numflux} can easily be seen to be consistent if the density mean \eqref{eq:densitymean} and internal-energy mean \eqref{eq:internalenergymean} are consistent. Let $\bm{u} \in \mathcal{U}$ and $[\rho,T] = \bm{\tau}(\bm{u})$. By the consistency property of discrete gradients \autoref{def:discgrad} and equations \eqref{eq:pressurethermo}, \eqref{eq:gibbs}: 
    \begin{align*}
        \widetilde{\rho}(\bm{u},\bm{u}) &= \frac{\widetilde{\nabla}_\rho (p\beta)(\bm{u},\bm{u})}{\widetilde{\nabla}_\rho (g\beta)(\bm{u},\bm{u})} \\
        &= \popt{p/T}{\rho}(\rho, T) / \popt{g/T}{\rho}(\rho, T) \\
        &= \left[ \frac{2 \rho}{T} \popt{\mathcal{A}}{\rho}(\rho, T) + \frac{\rho^2}{T}\ppoppt{\mathcal{A}}{\rho}(\rho, T) \right] / \left[\frac{2}{T} \popt{\mathcal{A}}{\rho}(\rho, T) + \frac{\rho}{T}\ppoppt{\mathcal{A}}{\rho}(\rho, T)\right] \\
        &= \rho.
    \end{align*}
    Since $\frac{\partial}{\partial \beta} = -T^2 \frac{\partial}{\partial T}$ and using \eqref{eq:gibbs}, \eqref{eq:pressurethermo} and \eqref{eq:internalenergy}, we have for the internal-energy mean \eqref{eq:internalenergymean}:
    \begin{align*}
        \widetilde{e}(\bm{u},\bm{u}) &= \widetilde{\nabla}_{\beta}(g\beta)(\bm{u},\bm{u}) - \frac{\widetilde{\nabla}_{\beta}(p\beta)(\bm{u},\bm{u})}{\widetilde{\rho}(\bm{u},\bm{u})} \\
        &= -T^2 \popt{g/T}{T}(\rho,T) + \frac{T^2}{\rho}\popt{p/T}{T}(\rho,T) \\
        &= -T^2 \left[-\frac{1}{T^2}\mathcal{A}(\rho,T) + \frac{1}{T}\popt{\mathcal{A}}{T}(\rho,T) - \frac{\rho}{T^2}\popt{\mathcal{A}}{\rho}(\rho,T) + \frac{\rho}{T}\frac{\partial^2\mathcal{A}}{\partial \rho \partial T}(\rho,T) \right] \\
        & \qquad \qquad \qquad+ \frac{T^2}{\rho}\left[-\frac{\rho^2}{T^2}\popt{\mathcal{A}}{\rho}(\rho,T) + \frac{\rho^2}{T}\frac{\partial^2\mathcal{A}}{\partial \rho \partial T}(\rho,T)  \right] \\
        &= \mathcal{A}(\rho,T) -  T\popt{\mathcal{A}}{T}(\rho,T) \\
        &= e(\rho,T).
    \end{align*}
\end{proof}

\subsection{Symmetry}
\begin{proof}
    Since, by assumption, the discrete gradients $\widetilde{\nabla}(p\beta), \widetilde{\nabla}(g\beta)$ are symmetric, the density mean \eqref{eq:densitymean} and internal-energy mean \eqref{eq:internalenergymean} are symmetric. Since the arithmetic mean is symmetric, all terms in the flux \eqref{eq:numflux} are therefore symmetric, and so it is also symmetric.
\end{proof}

\subsection{Density mean positivity}
\begin{proof}
    Suppose $\bm{u}_1,\bm{u}_2 \in \mathcal{U}$ are arbitrary points so that $\mathcal{A}$ is both convex in $1/\rho$ and $C^2$ on $\widetilde{\mathcal{T}}(\bm{u}_1,\bm{u}_2)$. The $\rho$-components of the symmetrized Itoh-Abe discrete gradients $\widetilde{\nabla}(p\beta), \widetilde{\nabla}(g\beta)$ are:
    \begin{align*}
        \widetilde{\nabla}_{\rho}(p\beta)(\bm{\tau}_1,\bm{\tau}_2) &= \sum_{\widetilde{T}\in \{T_1,T_2\}} \frac{1}{2} \frac{p(\rho_2, \widetilde{T}) / \widetilde{T} - p(\rho_1, \widetilde{T}) /\widetilde{T}}{\rho_2 - \rho_1},\\
        \widetilde{\nabla}_{\rho}(g\beta)(\bm{\tau}_1,\bm{\tau}_2) &= \sum_{\widetilde{T}\in \{T_1,T_2\}} \frac{1}{2} \frac{g(\rho_2, \widetilde{T}) / \widetilde{T} - g(\rho_1, \widetilde{T}) / \widetilde{T}}{\rho_2 - \rho_1}.
    \end{align*}
    By equations \eqref{eq:pressurethermo}, \eqref{eq:gibbs} and the assumption that $\mathcal{A} \in C^2(\widetilde{\mathcal{T}}(\bm{u}_1,\bm{u}_2))$, both $p(\cdot, \widetilde{T})/\widetilde{T}$ and $g(\cdot,\widetilde{T})/\widetilde{T}$, where $\widetilde{T} \in \{T_1,T_2\}$, are continuous and differentiable functions on the interval $[\rho_1, \rho_2]$. Thus, the mean value theorem (MVT) applies, and:
    \begin{equation*}
        \text{MVT: } \quad \text{there exist } \widetilde{\rho}_1,\widetilde{\rho}_2 \in [\rho_1,\rho_2] : \widetilde{\nabla}_{\rho}(p\beta)(\bm{\tau}_1,\bm{\tau}_2) = \sum_{i=1}^2 \frac{1}{2} \frac{\partial p / T}{\partial \rho}(\widetilde{\rho}_i, \widetilde{T}_i).
    \end{equation*}
    Note that $(\widetilde{\rho}_i, \widetilde{T}_i) \in \widetilde{\mathcal{T}}(\bm{u}_1,\bm{u}_2)$ for $i=1,2$. Considering \eqref{eq:pressurethermo}, for all $(\rho,T) \in \widetilde{\mathcal{T}}(\bm{u}_1\bm{u}_2)$ we have:
    \begin{equation*}
        \popt{p/T}{\rho} = \frac{2 \rho}{T} \popt{\mathcal{A}}{\rho} + \frac{\rho^2}{T}\ppoppt{\mathcal{A}}{\rho} = \frac{1}{\rho^2T} \ppoppt{\mathcal{A}}{(1/\rho)} > 0,
    \end{equation*}
    due to convexity in $1/\rho$ of $\mathcal{A}$ on $\widetilde{\mathcal{T}}(\bm{u}_1\bm{u}_2)$ and positivity of $(\rho,T) \in \widetilde{\mathcal{T}}(\bm{u}_1\bm{u}_2)$. Thus $\widetilde{\nabla}_{\rho}(p\beta)(\bm{\tau}_1,\bm{\tau}_2) > 0$. Similarly, we have:
    \begin{equation*}
        \text{MVT: } \quad \text{there exist } \widetilde{\rho}_1,\widetilde{\rho}_2 \in [\rho_1,\rho_2] : \widetilde{\nabla}_{\rho}(g\beta)(\bm{\tau}_1,\bm{\tau}_2) = \sum_{i=1}^2 \frac{1}{2} \frac{\partial g / T}{\partial \rho}(\widetilde{\rho}_i, \widetilde{T}_i),
    \end{equation*}
    with $(\widetilde{\rho}_i, \widetilde{T}_i) \in \widetilde{\mathcal{T}}(\bm{u}_1,\bm{u}_2)$ for $i=1,2$. Now considering \eqref{eq:gibbs}, for all $(\rho,T) \in \widetilde{\mathcal{T}}(\bm{u}_1\bm{u}_2)$ we have:
    \begin{equation*}
        \popt{g/T}{\rho} = \frac{2}{T} \popt{\mathcal{A}}{\rho} + \frac{\rho}{T}\ppoppt{\mathcal{A}}{\rho} = \frac{1}{\rho^3T} \ppoppt{\mathcal{A}}{(1/\rho)} > 0,
    \end{equation*}
    and therefore $\widetilde{\nabla}_{\rho}(g\beta)(\bm{\tau}_1,\bm{\tau}_2) > 0$.
\end{proof}

\subsection{Ideal gas behavior}
For an ideal gas, the specific Helmholtz energy is given as in \eqref{eq:IGEOS}, the pressure $p(\rho,T)$, computed using \eqref{eq:pressurethermo}, satisfies:
\begin{equation}
    p(\rho,T) = \rho \overline{R}T, 
    \label{eq:pideal}
\end{equation}
and the specific Gibbs energy, computed using \eqref{eq:gibbs}, satisfies:
\begin{equation}
    g(\rho,T) = -\frac{\overline{R}T}{\gamma - 1}\ln(T) + \overline{R}T\ln(\rho), 
    \label{eq:gideal}
\end{equation}
where we set $s_0 = 0$.
\begin{proof}
    Let $\bm{u}_1, \bm{u}_2 \in \mathcal{U}$ and $[\rho_1,T_1] = \bm{\tau}(\bm{u}_1)$, $[\rho_2,T_2] = \bm{\tau}(\bm{u}_2)$. We first analyze the density mean \eqref{eq:densitymean}. For the ideal-gas EoS, the symmetrized Itoh-Abe discrete gradient $\widetilde{\nabla}_{\rho}(p\beta)(\bm{u}_1,\bm{u}_2)$ satisfies:
    \begin{align*}
        \widetilde{\nabla}_{\rho}(p\beta)(\bm{u}_1,\bm{u}_2) &= \sum_{\widetilde{{T}}\in \{T_1,T_2\}} \frac12 \frac{p(\rho_2,\widetilde{T})/ \widetilde{T} - p(\rho_1,\widetilde{T})/ \widetilde{T}}{\rho_2 - \rho_1} \\
        &= \sum_{\widetilde{{T}}\in \{T_1,T_2\}} \frac12 \frac{\overline{R}\rho_2 - \overline{R}\rho_1}{\rho_2 - \rho_1} \\
        &= \overline{R}.
    \end{align*}
    The discrete gradient $\widetilde{\nabla}_{\rho}(g\beta)(\bm{u}_1,\bm{u}_2)$ satisfies:
    \begin{align*}
        \widetilde{\nabla}_{\rho}(g\beta)(\bm{u}_1,\bm{u}_2) &= \sum_{\widetilde{{T}}\in \{T_1,T_2\}} \frac12 \frac{g(\rho_2,\widetilde{T})/ \widetilde{T} - g(\rho_1,\widetilde{T})/ \widetilde{T}}{\rho_2 - \rho_1} \\
        &= \sum_{\widetilde{{T}}\in \{T_1,T_2\}} -\frac12\frac{\overline{R}}{\gamma-1} \frac{\ln(\widetilde{T}) - \ln(\widetilde{T})}{\rho_2 - \rho_1} + \frac12 \frac{\overline{R}\ln(\rho_2) - \overline{R}\ln(\rho_1)}{\rho_2 -\rho_1} \\
        &= \overline{R} \left(\overline{\rho}^{\ln}\right)^{-1},
    \end{align*}
    where $\overline{(\cdot)}^{\ln}$ denotes the logarithmic average. The density mean \eqref{eq:densitymean} is thus:
    \begin{equation*}
        \widetilde{\rho}(\bm{u}_1,\bm{u}_2) = \frac{\overline{R}}{\overline{R} \left(\overline{\rho}^{\ln}\right)^{-1}} = \overline{\rho}^{\ln},
    \end{equation*}
    for the ideal-gas EoS. We now consider the $\beta$-components of the discrete gradients. Since $p\beta = \rho \overline{R}$, we have $\widetilde{\nabla}_{\beta}(p\beta)(\bm{u}_1,\bm{u}_2) = 0$ as $p\beta$ is constant with respect to changes in $\beta$. Finally, for $\widetilde{\nabla}_{\beta}(g\beta)(\bm{u}_1,\bm{u}_2)$, we have:
    \begin{align*}
        \widetilde{\nabla}_{\beta}(g\beta)(\bm{u}_1,\bm{u}_2) &= \sum_{\widetilde{{\rho}}\in \{\rho_1,\rho_2\}} \frac12 \frac{g(\tilde{\rho},T_2)/ T_2 - g(\tilde{\rho},T_1)/ T_1}{1/T_2 -1/T_1} \\
        &= \sum_{\widetilde{{\rho}}\in \{\rho_1,\rho_2\}} -\frac12 \frac{\overline{R}}{\gamma-1} \frac{\ln(T_2) - \ln(T_1)}{1/T_2 -1/T_1}+\frac12\frac{\overline{R}\ln(\widetilde{\rho})-\overline{R}\ln(\widetilde{\rho})}{1/T_2 -1/T_1} \\
        &= \frac{\overline{R}}{\gamma - 1}\left(\overline{(1/T)}^{\ln} \right)^{-1}.
    \end{align*}
    Since, according to the ideal-gas EoS, $T = p / (\rho \overline{R})$, the ideal-gas internal-energy mean \eqref{eq:internalenergymean} is:
    \begin{equation*}
        \widetilde{e}(\bm{u}_1,\bm{u}_2) = \frac{\overline{R}}{\gamma - 1} \left(\overline{(\overline{R}\rho/p)}^{\ln}\right)^{-1}.
    \end{equation*}
    Comparison against \cite{ranochapreventing, ranochaentropy, ranochathesis}, shows the Ranocha flux is identical to \eqref{eq:numflux} in the ideal-gas EoS case for proper normalization ($\overline{R}=1$).
\end{proof}

\section{Helmholtz energy expressions}\label{app:eos}
We list the specific Helmholtz energies of some of the EoS used in this article. Let $R \approx 8.314 \joule \cdot \kelv^{-1} \cdot \mol^{-1}$ be the universal gas constant (we use a more accurate value to generate results), and $\overline{R} = R / M$ be the specific gas constant.
\begin{itemize}
    \item The ideal-gas (IG) EoS \cite{holystthermodynamics}: 
    \begin{equation}
        \mathcal{A}(\rho,T) := - \overline{R} T \left[1+ \ln\left(\frac{T^{\frac{1}{\gamma -1}}}{\rho} \right) \right], \label{eq:IGEOS}
    \end{equation}
    where we use $\gamma = 1.4$ for $CO_2$.
    \item The Van der Waals (VdW) EoS \cite{holystthermodynamics}: 
    \begin{equation*}
        \mathcal{A}(\rho,T) := - \overline{R}T \left[1 +\ln\left( \frac{(1-\rho b)T^{\frac{\varsigma}{2}}}{\rho} \right) \right] - a\rho,
    \end{equation*}
    where:
    \begin{equation*}
        a = \frac{27}{64} \frac{\overline{R}^2 T_c^2}{p_c}, \quad b = \frac{1}{8} \frac{\overline{R}T_c}{p_c}.
    \end{equation*}
    \item The Peng-Robinson (PR) EoS \cite{pengnew, senguptafully}:
    \begin{equation*}
        \mathcal{A}(\rho,T) := - \overline{R}T \left[1 +\ln\left( \frac{(1-\rho b)T^{\frac{\varsigma}{2}}}{\rho} \right) \right] - \frac{\alpha(T)}{2\sqrt{2}b}\ln\left(\frac{1+(1+\sqrt{2})b\rho}{1+(1-\sqrt{2})b\rho} \right),
    \end{equation*}
    where:
    \begin{equation*}
        \alpha(T) = a\left[1+\left(0.37464 + 1.54226 \theta - 0.26992 \theta^2\right)\left(1-\sqrt{\frac{T}{T_c}} \right)\right],
    \end{equation*}
    and:
    \begin{equation*}
        a=0.457235\frac{\overline{R}^2 T_c^2}{p_c}, \quad b = 0.077796\frac{\overline{R}T_c}{p_c}.
    \end{equation*}
\end{itemize}
For brevity, we do not report the Kunz-Wagner EoS, but refer to \cite{kunzgerg} and our open-source code \cite{kleinHelmEOS2}.

\section{Discretization of viscous stress and heat flux divergence}\label{app:viscdisc}
For simplicity, we will describe the viscous and heat discretization for $d=2$ and only in the $x_1$-direction. The extensions to $d=3$ and other directions are straightforward. We start with the viscous discretization. In two space dimensions, with velocity $\bm{v}=[v_1,v_2]^T$, the viscous stress tensor can be written as:
\begin{equation*}
    \bm{\tau}
    =
    \mu \left( \nabla \bm{v} + \nabla \bm{v}^\top \right)
    - \mu\,(\nabla \cdot \bm{v})\,I.
\end{equation*}
For the viscous flux in the $x_1$-direction, we evaluate a two-point approximation of the stress-tensor column associated to the $x_1$-direction $\bm{\tau}_h^1 : \mathcal{U} \times \mathcal{U} \rightarrow \mathbb{R}^{2}$ between $\bm{u}_{\bm{I}}$ and $\bm{u}_{\bm{I}+\bm{1}_1}$, and denote the full viscous flux between these points $\bm{f}^{\mathrm{visc}}_{\bm{I}+\bm{1}_1/2}$. The two-point approximation is given by centered differences and arithmetic averaging of the transport coefficients $\mu$. 
Denote $\bm{\tau}_{\bm{I}+\bm{1}_1/2} := \bm{\tau}_h^1(\bm{u}_{\bm{I}},\bm{u}_{\bm{I}+\bm{1}_1})$, then the viscous $x_1$-flux is given by:
\begin{equation*}
    \bm{f}^{\mathrm{visc}}_{\bm{I}+\bm{1}_1/2}
    =
    \begin{bmatrix}
        0 \\
        (\tau_1)_{\bm{I}+\bm{1}_1/2} \\
        (\tau_2)_{\bm{I}+\bm{1}_1/2} \\
        (\tau_1)_{\bm{I}+\bm{1}_1/2}(\overline{v}_1)_{\bm{I}+\bm{1}_1/2}
        + (\tau_2)_{\bm{I}+\bm{1}_1/2}(\overline{v}_2)_{\bm{I}+\bm{1}_1/2}
    \end{bmatrix}.
\end{equation*}
Here, the velocity components $(\overline{v}_1)_{\bm{I}+\bm{1}_1/2},(\overline{v}_2)_{\bm{I}+\bm{1}_1/2}$ are approximated with arithmetic means. The viscosity coefficient is also approximated using an arithmetic mean:
\begin{equation*}
    \overline{\mu}_{\bm{I}+\bm{1}_1/2} = \frac{\mu_{\bm{I}}+\mu_{\bm{I}+\bm{1}_1}}{2},
\end{equation*}
and the required derivatives at the face are discretized by
\begin{align*}
    \left(\partial_{x_1} v_1\right)_{\bm{I}+\bm{1}_1/2}
    &=
    \frac{\Delta(v_1)_{\bm{I}+\bm{1}_1/2}}{\Delta x_1}, \\
    \left(\partial_{x_1} v_2\right)_{\bm{I}+\bm{1}_1/2}
    &=
    \frac{\Delta(v_2)_{\bm{I}+\bm{1}_1/2}}{\Delta x_1}, \\
    \left(\partial_{x_2} v_1\right)_{\bm{I}+\bm{1}_1/2}
    &=
    \frac{1}{2}\left(
    \frac{(v_1)_{(i_1,i_2+1)}-(v_1)_{(i_1,i_2-1)}}{2\Delta x_2}
    +
    \frac{(v_1)_{(i_1+1,i_2+1)}-(v_1)_{(i_1+1,i_2-1)}}{2\Delta x_2}
    \right), \\
    \left(\partial_{x_2} v_2\right)_{\bm{I}+\bm{1}_1/2}
    &=
    \frac{1}{2}\left(
    \frac{(v_2)_{(i_1,i_2+1)}-(v_2)_{(i_1,i_2-1)}}{2\Delta x_2}
    +
    \frac{(v_2)_{(i_1+1,i_2+1)}-(v_2)_{(i_1+1,i_2-1)}}{2\Delta x_2}
    \right),
\end{align*}
where we have written out the grid multi-indices in full for convenience. Thus,
\begin{align*}
    (\tau_1)_{\bm{I}+\bm{1}_1/2}
    &=
    2\mu_{\bm{I}+\bm{1}_1/2}\left(\partial_{x_1} v_1\right)_{\bm{I}+\bm{1}_1/2}
    - \mu_{\bm{I}+\bm{1}_1/2}
    \left( \left(\partial_{x_1} v_1\right)_{\bm{I}+\bm{1}_1/2}
    + \left(\partial_y v\right)_{\bm{I}+\bm{1}_1/2} \right), \\
    (\tau_2)_{\bm{I}+\bm{1}_1/2} 
    &=
    \mu_{\bm{I}+\bm{1}_1/2}
    \left(
    \left(\partial_{x_2} v_1\right)_{\bm{I}+\bm{1}_1/2}
    + \left(\partial_{x_1} v_2\right)_{\bm{I}+\bm{1}_1/2}
    \right).
\end{align*}

For the heat conduction term, we discretize Fourier's law \eqref{eq:fourier} in the same way. The $x_1$-component of the heat flux $\bm{f}_{\bm{I}+\bm{1}_1/2}$ is:
\begin{equation*}
    q_{\bm{I}+\bm{1}_1/2}
    =
    \kappa_{\bm{I}+\bm{1}_1/2}
    \frac{\Delta T_{\bm{I}+\bm{1}_1/2}}{\Delta x_1},
    \qquad
    \kappa_{\bm{I}+\bm{1}_1/2}
    =
    \frac{\kappa_{\bm{I}}+\kappa_{\bm{I}+\bm{1}_1}}{2},
\end{equation*}
which contributes only to the total-energy equation:
\begin{equation*}
    \bm{f}^{\mathrm{heat}}_{\bm{I}+\bm{1}_1/2}
    =
    \begin{bmatrix}
        0 \\ 0 \\ 0 \\ q_{\bm{I}+\bm{1}_1/2}
    \end{bmatrix}.
\end{equation*}
The corresponding divergence contribution is then assembled in conservative flux-difference form as in \eqref{eq:disc}.
\section{Ideal-gas Taylor-Green verification}\label{app:igvalid}
To verify our HelmEOS2 code implementation we compare the results of the viscous Taylor-Green vortex (viscous TGV) experiment against a reference solution \cite{debonissolutions}. The reference solution is computed using the ideal-gas (IG) equation of state (EoS). Hence, for this comparison, we will also compute our results using the IG EoS. The reference quantities for the initial condition \eqref{eq:tgv_initial} are chosen so that the Reynolds number, Mach number and Prandtl number are $\mathrm{Re}=1600, \mathrm{Ma}=0.1, \mathrm{Pr}=0.71$, respectively. Note that the precise Prandtl number used in \cite{debonissolutions} is not reported, hence we use the commonly used value from \cite{wanghighorder}. The reference quantities used in \eqref{eq:tgv_initial} are taken as in \autoref{tab:igtgparams}. We set the dynamic viscosity and heat conductivity of the ideal gas constant at $\mu = \mu_0$ and $\kappa = \kappa_0$, we also assume the specific heat ratio is constant at $\gamma = 1.4$. The mesh and timestep size will be the same as the viscous TGV experiment described in \autoref{sec:tgv}, as this should be sufficient to resolve the turbulent scales at a modest CFL number. We will also use Wray's low-storage Runge-Kutta (RK) method \cite{wrayminimal} again. 

\begin{table}[]
\centering
\begin{tabular}{lll}
\hline
Symbol      & Value                 & Unit                                  \\ \hline
$\rho_0$   & $1.198\rho_c$             & $\kg \cdot \meter^{-3}$               \\
$T_0$     & $294.4444$              & $\kelv $                              \\
$p_0$     & $p(\rho_0,T_0)$         & $\pascal $                            \\
$V_0$     & $0.1 \cdot c(\rho_0, T_0)$  & $\meter \cdot \second^{-1}$       \\
$L$        & $0.001524 \hbox{ }\meter$  & $\meter$                          \\ 
$\mu_0$      & $\rho_0 V_0 L / 1600$    & $\pascal \cdot \second$           \\ 
$\kappa_0$  & $\mu_0 c_p(\rho_0,T_0) / 0.71$  & $\joule \cdot \second^{-1} \cdot \meter^{-1} \cdot \kelv^{-1}$ \\ \hline
\end{tabular}
\caption{Choice of reference quantities in \eqref{eq:tgv_initial} for the simulation of the Taylor-Green vortex with an ideal gas.}
\label{tab:igtgparams}
\end{table}

It is a custom in the literature to plot the dissipation rate of kinetic energy \cite{wanghighorder,debonissolutions,sciacovelliassessment,vanreescomparison}. To standardize our results with respect to the literature, we must compute the kinetic-energy dissipation rate using the following normalized quantities:
\begin{equation*}
    \rho^* := \frac{\rho}{\rho_0}, \quad \bm{v}^* := \frac{\bm{v}}{V_0}, \quad \bm{x}^* := \frac{\bm{x}}{2\pi L}, \quad t^* := \frac{t}{t_c}.
\end{equation*}
Let:
\begin{equation*}
    \mathcal{K}(t) := \int_{\Omega} k(\bm{u}(\bm{x},t))d\Omega.
\end{equation*}
Then we compute the normalized kinetic-energy dissipation rate $\left(\dudt{\mathcal{K}}\right)^*$ from the physical $\dudt{\mathcal{K}}$ as:
\begin{equation*}
    \left(\dudt{\mathcal{K}}\right)^* = \frac{t_c}{\rho_0 V_0^2 |\Omega|}\dudt{\mathcal{K}},
\end{equation*}
where $|\Omega| \in \mathbb{R}_+$ is the volume of the spatial domain $\Omega$. To further assure the kinetic-energy budget as in \eqref{eq:kinenbudget} still holds, the quantities $\Pi, \mathcal{D}, \mathcal{E}$ in \eqref{eq:kinenbudget} are normalized using the same normalization factor:
\begin{equation*}
    \Pi^* = \frac{t_c}{\rho_0 V_0^2 |\Omega|} \Pi, \quad \mathcal{D}^* = \frac{t_c}{\rho_0 V_0^2 |\Omega|} \mathcal{D}, \quad \mathcal{E}^* = \frac{t_c}{\rho_0 V_0^2 |\Omega|} \mathcal{E}.
\end{equation*}
We consistently use normalized quantities in all results related to the kinetic-energy dissipation rate; hence, we drop the normalization notation $(\cdot)^*$ throughout. The results of the verification study are shown in \autoref{fig:dkdt_IG} and \autoref{fig:vort_mag_IG}. The results in \autoref{fig:dkdt_IG} show a close agreement in kinetic-energy dissipation between our second-order GEoS-KEEP scheme and the high-order, dispersion-relation-preserving scheme used in \cite{debonissolutions}, also defined on a mesh of $512^3$ cells. As discussed in \autoref{sec:vtgv}, we expect a near exact match between $\dudt{\mathcal{K}}$ and $\mathcal{E}$, which \autoref{fig:dkdt_IG} shows to be almost attained. The inexact match can be attributed to insufficient grid convergence. The lack of grid convergence is further corroborated by \autoref{fig:vort_mag_IG}, where we reproduce contours of the normalized vorticity magnitude $\omega$ as in \eqref{eq:normvort}. Comparison with the same figure in \cite[Figure 5]{debonissolutions} shows that our solution is nearly grid-converged, which we consider sufficient for our experiments.

\begin{figure}
    \centering
    \includegraphics[width=0.6\linewidth]{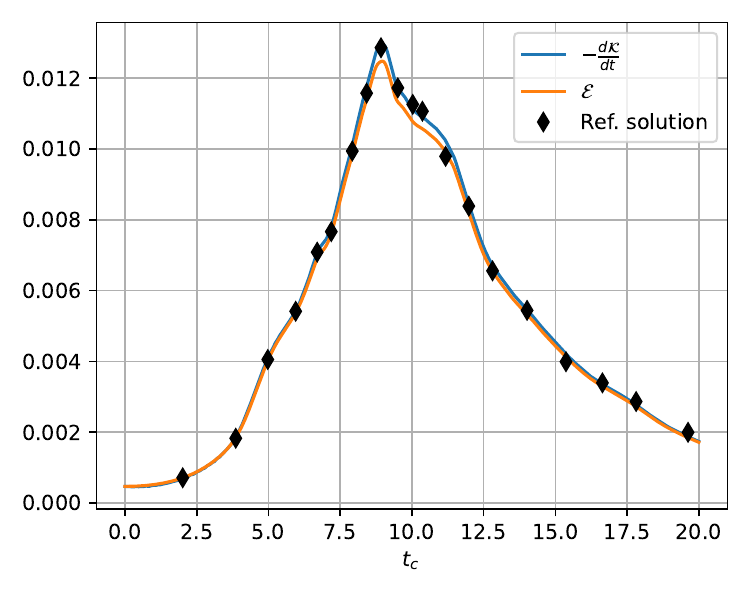}
    \caption{The normalized kinetic-energy dissipation rate \eqref{eq:kinenbudget} as computed by our implementation compared against $\mathcal{E}$ and the reference \cite{debonissolutions} for the viscous TGV experiment using the IG EoS.}
    \label{fig:dkdt_IG}
\end{figure}

\begin{figure}
    \centering
    \includegraphics[width=0.8\linewidth]{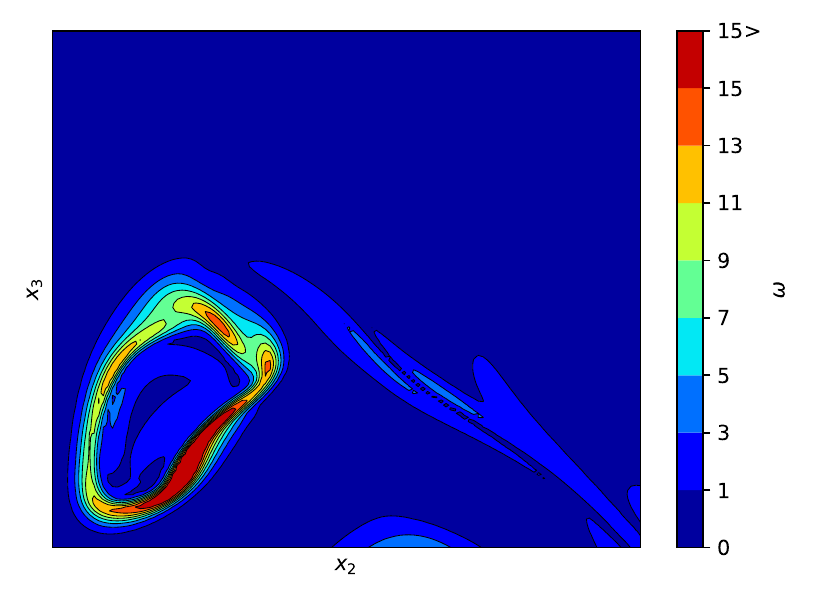}
    \caption{The normalized vorticity magnitude $\omega$ as in \eqref{eq:normvort} at $x_1 = -\pi L$ in the region $x_2\in[0, \frac{\pi L}{2}]$ and $x_3 \in [\frac{\pi L}{2}, \pi L]$ computed for the viscous TGV experiment using the IG EoS, cf. \cite{debonissolutions}. To match \cite[figure 5]{debonissolutions}, the values of $\omega > 15$ are collapsed to the same color value.}
    \label{fig:vort_mag_IG}
\end{figure}

\bibliographystyle{elsarticle-num}
\bibliography{cas-refs}

\end{document}